\documentclass[12pt]{amsart}

\calclayout

\usepackage{amsmath,amsthm,amssymb}
\usepackage{mathrsfs}

\usepackage{colonequals}
\usepackage{enumitem}

\usepackage[colorlinks=true,allcolors=blue,pdftitle={3D vortex approximation construction and estimates for Ginzburg-Landau}]{hyperref}

\usepackage[initials,nobysame,alphabetic,msc-links,backrefs]{amsrefs}

\newtheorem{theorem}{Theorem}[section]
\newtheorem{corollary}{Corollary}[section]
\newtheorem{proposition}{Proposition}[section]
\newtheorem{lemma}{Lemma}[section]
\newtheorem{definition}{Definition}[section]
\newtheorem{remark}{Remark}[section]

\numberwithin{equation}{section}

\DeclareMathOperator{\diver}{div}
\DeclareMathOperator{\curl}{curl}

\newcommand{\de}{\delta}
\newcommand{\la}{\lambda}
\newcommand{\ve}{\varepsilon}

\renewcommand{\L}{\mathbb L}
\newcommand{\N}{\mathbb N}
\newcommand{\R}{\mathbb R}
\newcommand{\C}{\mathbb C}

\renewcommand{\H}{\mathcal H}
\newcommand{\LL}{\mathcal L}
\renewcommand{\O}{\mathcal O}

\newcommand{\B}{\mathfrak B}
\renewcommand{\d}{\mathfrak d}
\newcommand{\GG}{\mathfrak G}
\newcommand{\RR}{\mathfrak R}
\renewcommand{\S}{\mathfrak S}

\renewcommand{\AA}{\mathscr A}
\newcommand{\CC}{\mathscr C}
\newcommand{\E}{\mathscr E}
\newcommand{\XX}{\mathscr X}

\renewcommand{\c}{\mathrm c}
\renewcommand{\deg}{\mathrm{deg}}
\newcommand{\dist}{\mathrm{dist}}
\newcommand{\ex}{\mathrm{ex}}
\newcommand{\loc}{\mathrm{loc}}

\title[3D vortex approximation construction and estimates for Ginzburg-Landau]{3D vortex approximation construction and $\ve$-level estimates for the Ginzburg-Landau functional}
\author{Carlos Rom\'{a}n}
\date{December 20, 2017}
\address{Sorbonne Universit\'{e}s, UPMC Univ Paris 06, CNRS, UMR 7598, Laboratoire Jacques-Louis Lions, 4, place Jussieu 75005, Paris, France
\newline
\& Mathematisches Institut, Universit\"{a}t Leipzig, Augustusplatz 10, 04109 Leipzig, Germany}
\email{roman@math.uni-leipzig.de}
\thanks{This work has been supported by a public grant overseen by the French National Research Agency (ANR) as part of the ``Investissements d'Avenir'' program (reference: ANR-10-LABX-0098, LabEx SMP)}

\begin{document}
\begin{abstract}
We provide a quantitative three-dimensional vortex approximation construction for the Ginzburg-Landau functional. This construction gives an approximation of vortex lines coupled to a lower bound for the energy, optimal to leading order, analogous to the 2D ones, and valid for the first time at the $\ve$-level. 
These tools allow for a new approach to analyze the behavior of global minimizers for the Ginzburg-Landau functional below and near the first critical field in 3D, followed in the forthcoming papers \cites{Rom2,RomSanSer}. In addition, they allow to obtain an $\ve$-quantitative product estimate for the study of Ginzburg-Landau dynamics.
\end{abstract}
\maketitle

%%%%%%%%%%%%%%%%%%%%%%%%%%%%%%%%%%%%%%%%%%%%%%%%%%%%%%%%%%%%%%%%%%%%%%%%%%%%%%%%%%%%%%%%%%%%%%%%%

\section{Introduction}
\subsection{The problem and a brief overview of the state of the art of the subject}
We are interested in studying the full three-dimensional Ginzburg-Landau functional with applied magnetic field
$$
GL_\varepsilon(u,A)=\frac12\int_\Omega |\nabla_A u|^2+\frac{1}{2\varepsilon^2}(1-|u|^2)^2+\frac12\int_{\R^3}|H-H_\ex|^2,
$$
which is a model for superconductors in a magnetic field. 

Here
\begin{itemize}
\item $\Omega$ is a bounded domain of $\R^3$, that we assume to be Lipschitz and simply connected.
\item $u:\Omega\rightarrow \mathbb{C}$ is called the \emph{order parameter}. Its modulus squared (the density of Cooper pairs of superconducting electrons in the Bardeen-Cooper-Schrieffer (BCS) quantum theory) indicates the local state of the superconductor: where $|u|^2\approx 1$ the material is in the superconducting phase, where $|u|^2\approx 0$ in the normal phase.
\item $A:\R^3\rightarrow \R^3$ is the electromagnetic vector potential of the induced magnetic field $H=\curl A$.
\item $\nabla_A$ denotes the covariant gradient $\nabla-iA$.
\item $H_{\ex}:\R^3\rightarrow \R^3$ is a given external (or applied) magnetic field.
\item $\ve>0$ is the inverse of the \emph{Ginzburg-Landau parameter} usually denoted $\kappa$, a non-dimensional parameter depending only on the material. We will be interested in the regime of small $\ve$, corresponding to extreme type-II superconductors. 
\end{itemize}
An essential feature of type-II superconductors is the occurrence of \emph{vortices} (similar to those in fluid mechanics, but quantized) in the presence of an applied magnetic field. Physically, they correspond to normal phase regions around which a superconducting loop of current circulates. Since $u$ is complex-valued, it can have zeros with a nonzero topological degree. Vortices are then \emph{topological defects} of co-dimension 2 and are the crucial objects of interest in the analysis of the model.

\medskip 
Let us introduce the Ginzburg-Landau free energy
$$
F_\ve(u,A)= \frac12\int_\Omega |\nabla_A u|^2+\frac{1}{2\ve^2}(1-|u|^2)^2+|\curl A|^2
$$
and the Ginzburg-Landau energy without magnetic field
$$
E_\ve(u)= \frac12\int_\Omega |\nabla u|^2+\frac{1}{2\ve^2}(1-|u|^2)^2.
$$
In the 1990's, mathematicians became interested in the Ginzburg-Landau model. In the pioneer work \cite{BetBreHel} in the 2D setting (i.e. when $\Omega$ is assumed to be two-dimensional), Bethuel, Brezis, and H\'{e}lein introduced systematic tools and asymptotic estimates to study vortices in the model without magnetic field, which is a complex-valued version of the Allen-Cahn model for phase transitions. A vortex in 2D is an object centered at an isolated zero of $u$, around which the phase of $u$ has a nonzero winding number, called the degree of the vortex. A typical vortex centered at a point $x_0$ behaves like $u=\rho e^{i\varphi}$ with $\rho =f\left(\frac{|x-x_0|}\ve \right)$, where $f(0)=0$ and $f$ tends to $1$ as $r\to +\infty$, i.e. its characteristic core size is $\ve$ and 
$$
\frac1{2\pi}\int_{\partial B(x_0,R\ve)}\frac{\partial \varphi}{\partial \tau}=d\in \mathbb Z
$$
is its degree (also defined as the topological-degree of the map $u/|u|:\partial B(x_0,R\ve)\to S^1$).

In \cite{BetBreHel}, the effect of the external magnetic field was replaced by a Dirichlet boundary condition $u=g$ on $\partial \Omega$, where $g$ is an $S^1$-valued map of degree $d>0$. This boundary condition triggers the occurrence of vortices, allowing only for a fixed number of them. They proved that minimizers $u$ of $E_\ve$ have $d$ vortices of degree +1 and that 
$$
E_\ve(u)\approx \pi d |\log \ve|+W(a_1,\dots,a_d)\quad \mbox{as }\ve\to 0,
$$
where $W$ is the ``renormalized energy'', a function depending only on the vortex-centers $a_i$, which repeal one another according to a Coulomb interaction.
This analysis was then adapted to the study of the free energy by Bethuel and Rivi\`ere \cite{BetRiv}, under a Dirichlet boundary condition on $\partial\Omega$ that forces the presence of a fixed number of vortices.

A new approach was necessary to treat the case of the full model when the number of vortices gets unbounded as $\ve\to 0$. Tools able of handling this difficulty were developed after the works by Jerrard \cite{Jer} and Sandier \cite{San0}. They introduced independently the ball construction method, which allows one to obtain universal lower bounds for two-dimensional Ginzburg-Landau energies in terms of the topology of the vortices. These lower bounds capture the known fact that vortices of degree $d$ cost at least an order $\pi|d|\log\frac1\ve$ of energy. The second tool, that has been widely used in the analysis of
the Ginzburg-Landau model in any dimension after the work by Jerrard and Soner \cite{JerSon}, is the Jacobian or vorticity estimate. The vorticity is defined, for any sufficiently regular configuration $(u,A)$, as
$$
\mu(u,A)=\curl (iu,\nabla_A u)+\curl A,
$$
where $(\cdot,\cdot)$ denotes the scalar product in $\C$ identified with $\R^2$ i.e. $(a,b)=\frac{\overline{a}b+a\overline{b}}2$.
This quantity is the U$(1)$-gauge invariant version of the Jacobian determinant of $u$
and is the analog of the vorticity of a fluid. The vorticity estimate allows one to relate the vorticity $\mu(u,A)$ with Dirac masses supported on co-dimension 2 objects, which in 2D are points naturally derived from the ball construction.
In a series of works summarized in the book \cite{SanSerBook}, Sandier and Serfaty analyzed the full two-dimensional model and characterized the behavior of global minimizers of $GL_\ve$ in different regimes of the applied field (see also \cites{SanSer1,SanSer0,SanSerFree,SanSer2}).

\medskip
Rivi\`ere \cite{Riv}, was the first to study the asymptotic behavior of minimizers of the free energy, under a Dirichlet boundary condition, as $\ve\to 0$ in the 3D setting. Roughly speaking, vortices in 3D are small tubes of radius $O(\ve)$ around the one dimensional zero-set of $u$. In the limit $\ve\to0$ vortices become curves $L$ with an integer multiplicity $d$, whose cost is at least an order $\pi d |L| |\log \ve|$ of energy, where $|L|$ denotes the length of $L$.
In \cite{Riv}, using an $\eta$-ellipticity result, Rivi\`ere identified the limiting one dimensional singular set of minimizers of $F_\ve$ with a mass minimizing current, which corresponds to a \emph{minimal connection}. This concept was introduced in the work by Brezis, Coron, and Lieb \cite{BreCorLie}. A new approach by Sandier in \cite{San}, combined the use of this object with a suitable slicing procedure to obtain the same result of Rivi\`{e}re in the case without magnetic field, and a generalization to higher dimension. We refer the interested reader to \cites{LinRiv1,BetBreOrl,LinRiv2,BouBreMir,AlbBalOrl,SanSha} for further results in dimensions 3 and higher, when the applied magnetic field is zero.

Jerrard, Montero, and Sternberg \cite{JerMonSte} established the existence of locally minimizing vortex solutions to the full three-dimensional Ginzburg-Landau energy.
Recently, Baldo, Jerrard, Orlandi, and Soner \cites{BalJerOrlSon1,BalJerOrlSon2}, via a $\Gamma$-convergence argument, described the asymptotic behavior of the full model as $\ve\to 0$. We point out that conversely to the 2D situation, which is well understood, many questions remain open in 3D, in particular obtaining all the analogues of the 2D results contained in \cite{SanSerBook}. This is due to the more complicated geometry of the vortices in 3D, which have to be understood in the framework of currents and using geometric measure theory. 

\subsection{\texorpdfstring{$\ve$}{Epsilon}-level estimates for the Ginzburg-Landau functional} 
The key in Ginzburg-Landau analysis has proven to be a vortex approximation construction providing both approximation of the vorticity and lower bound. In 2D, this corresponds to the ball construction \cites{San0,Jer,SanSerBook}, which is a purely two-dimensional method that provides $\ve$-quantitative estimates.
In 3D (and higher), based on the Federer-Fleming polyhedral deformation theorem, a not quantitative construction was provided in \cite{AlbBalOrl} and later revisited in \cite{BalJerOrlSon1}.

In this paper we present a new 3D vortex approximation construction, which provides an approximation of vortex filaments coupled to a lower bound for the energy, optimal to leading order, analogous to the 2D ones, and valid for the first time at the $\ve$-level. Roughly speaking, our approximation is made as follows. For configurations $(u_\ve,A_\ve)$ whose free energy is bounded above by a suitable function of $\ve$, we consider a grid of side-length $\de=\de(\ve)$. If appropriately positioned, the grid can be taken to satisfy that $|u_\ve|>5/8$ on every edge of a cube. Then a 2D vorticity estimate implies that the restriction of the vorticity $\mu(u_\ve,A_\ve)$ to the boundary of every cube is well approximated by a linear combination of Dirac masses. Using minimal connections, we connect the points of support of these measures. Our choice of grid ensures a good compatibility between the objects constructed in cubes that share a face. Finally, by considering the distance 
$$
d_{\partial\Omega}(x,y)=\min \{|x-y|,d(x,\partial\Omega)+d(y,\partial\Omega)\},
$$
we construct our approximation close to $\partial\Omega$, using minimal connections defined in terms of this distance. This process yields a closed polyhedral $1$-dimensional current $\nu_\ve$, or, more precisely, a sum in the sense of currents of Lipschitz curves, that approximates well the vorticity $\mu(u_\ve,A_\ve)$ in a suitable norm. 

We may now state our main results.
\begin{theorem}[$\ve$-level estimates for Ginzburg-Landau in 3D]\label{Theorem1} Assume that $\partial\Omega$ is of class $C^2$. For any $m,n,M>0$ there exist $C,\ve_0>0$ depending only on $m,n,M,$ and $\partial\Omega$, such that, for any $\ve<\ve_0$, if $(u_\ve,A_\ve)\in H^1(\Omega,\C)\times H^1(\Omega,\R^3)$ is a configuration such that $F_\ve(u_\ve,A_\ve)\leq M|\log\ve|^m$ then there exists a polyhedral $1$-dimensional current $\nu_\ve$ such that
\begin{enumerate}[leftmargin=*,font=\normalfont]
\item $\nu_\ve /\pi$ is integer multiplicity,
\item $\partial \nu_\ve=0$ relative to $\Omega$,
\item $\mathrm{supp}(\nu_\ve)\subset S_{\nu_\ve}\subset \overline \Omega$ with $|S_{\nu_\ve}|\leq C|\log\ve|^{-q}$, where $q(m,n)\colonequals\frac32 (m+n)$,
\item 
\begin{multline}\label{LowerBound} 
\int_{S_{\nu_\ve}}|\nabla_{A_\ve} u_\ve|^2+\frac{1}{2\ve^2}(1-|u_\ve|^2)^2+|\curl A_\ve|^2\\ \geq |\nu_\ve|(\Omega)\left(\log \frac1\ve-C \log \log \frac1\ve\right)-\frac{C}{|\log\ve|^n},
\end{multline}
\item and for any $\gamma\in(0,1]$ there exists a constant $C_\gamma$ depending only on $\gamma$ and $\partial\Omega$, such that 
\begin{equation}\label{EstimateJ} 
\|\mu(u_\ve,A_\ve) -\nu_\ve\|_{C_T^{0,\gamma}(\Omega)^*}\leq C_\gamma \frac{F_\ve(u_\ve,A_\ve)+1}{|\log \ve|^{q\gamma}}.
\end{equation}
\end{enumerate}
\end{theorem}
Notation and definitions of the objects and spaces involved in this result can be found in the preliminaries (see Section \ref{Preliminaries}).
\begin{remark} Alternatively, the right-hand sides of the lower bound and the vorticity estimate can be expressed in terms of the free energy $F_\ve(u_\ve,A_\ve)$ of the configuration $(u_\ve,A_\ve)$ and a length $\de=\de(\ve)$, which measures how ``close'' $\mu_\ve(u_\ve,A_\ve)$ is to $\nu_\ve$, and which is a parameter of the construction (the side-length of the aforementioned grid). This will be done in the rest of the paper. 

We also remark that the right-hand side of \eqref{EstimateJ} can be made small if $n>m\left(\frac2{3\gamma}-1\right)$.
\end{remark}
The technical assumption that $\partial\Omega$ is of class $C^2$ allows us to find a lower bound for the free energy close to the boundary of the domain. 
If $\partial \Omega$ is only assumed to be Lipschitz, one has the following result.

\begin{theorem}\label{Theorem2} For any $m,n,M>0$ there exist $C,\ve_0>0$ depending only on $m,n,$ and $M$, such that, for any $\ve<\ve_0$, if $(u_\ve,A_\ve)\in H^1(\Omega,\C)\times H^1(\Omega,\R^3)$ is a configuration such that $F_\ve(u_\ve,A_\ve)\leq M|\log\ve|^m$ then, letting $q\colonequals \frac32 (m+n)$ and defining
$$
\Omega_\ve\colonequals \{x\in \Omega \ | \ d(x,\partial\Omega)\geq 2|\log \ve|^{-q}\},
$$
there exists a polyhedral $1$-dimensional current $\nu_\ve$ such that 
\begin{enumerate}[leftmargin=*,font=\normalfont]
\item $\nu_\ve /\pi$ is integer multiplicity,
\item $\partial \nu_\ve=0$ relative to $\Omega$,
\item $\mathrm{supp}(\nu_\ve)\subset S_{\nu_\ve}\subset \overline \Omega$ with $|S_{\nu_\ve}|\leq C|\log\ve|^{-q}$,
\item 
\begin{multline*}
\int_{S_{\nu_\ve}}|\nabla_{A_\ve} u_\ve|^2+\frac{1}{2\ve^2}(1-|u_\ve|^2)^2+|\curl A_\ve|^2\geq |\nu_\ve|(\Omega_\ve)\left(\log \frac1\ve-C \log \log \frac1\ve\right)-\frac{C}{|\log\ve|^n},
\end{multline*}
\item and for any $\gamma\in(0,1]$ there exists a constant $C_\gamma$ depending only on $\gamma$ and $\partial\Omega$, such that 
$$
\|\mu(u_\ve,A_\ve) -\nu_\ve\|_{C_0^{0,\gamma}(\Omega)^*}\leq C_\gamma \frac{F_\ve(u_\ve,A_\ve)+1}{|\log \ve|^{q\gamma}}
$$
\end{enumerate}
\end{theorem}
As a direct consequence of Theorem \ref{Theorem1}, we recover and improve within our work setting, a well known result concerning the convergence as $\ve\to 0$ of the vorticity of families of configurations whose free energy is bounded above by a constant times a power of $|\log \ve|$. Results of the same kind can be found in \cites{JerSon,JerMonSte,SanSer3,AlbBalOrl,BalJerOrlSon1}.
\begin{corollary}
Assume that $\partial\Omega$ is of class $C^2$. Let $\{(u_\ve,A_\ve)\}_\ve$ be a family of configurations of $H^1(\Omega,\C)\times H^1(\Omega,\R^3)$ such that $F_\ve(u_\ve,A_\ve)\leq M |\log \ve|^m$ for some $m\geq 1$ and $M>0$. Then, up to extraction, 
$$
\frac{\mu(u_\ve,A_\ve)}{|\log \ve|^{m-1}}\rightharpoonup \mu \quad \mathrm{in}\ C_T^{0,\gamma}(\Omega)^*
$$
for any $\gamma\in (0,1]$, where $\mu$ is a $1$-dimensional current such that $\mu/\pi$ is integer multiplicity and $\partial \mu=0$ relative to $\Omega$. If $m=1$ then $\mu$ is in addition rectifiable. Moreover,
$$
\liminf_{\ve\to 0} \frac{F_\ve(u_\ve,A_\ve)}{|\log \ve|^m}\geq |\mu|(\Omega).
$$  
\end{corollary}

\subsection{Application to the full Ginzburg-Landau functional}
The behavior of global minimizers for $GL_\ve$ is determined by the strength of the external magnetic field $H_{\ex}$. This model is known to exhibit several phase transitions, which occur for certain critical values of the intensity of $H_{\ex}$. We are interested in the so-called \emph{first critical field}, usually denoted by $H_{\c_1}$. Physically, it is characterized as follows. Below $H_{\c_1}$ the superconductor is everywhere in its superconducting phase $|u|\approx 1$ and the external magnetic field is forced out by the material. This phenomenon is known as the \emph{Meissner effect}. At $H_{\c_1}$, which is of the order of $|\log\ve|$ as $\ve\to 0$, the first vortice(s) appear and the external magnetic field penetrates the material through the vortice(s). 

In the works \cites{Ser,SanSer1,SanSer2}, Sandier and Serfaty derived with high precision the value of the first critical field and rigorously described the behavior of global minimizers of $GL_\ve$ below and near $H_{\c_1}$ in 2D. In the 3D setting, Alama, Bronsard, and Montero \cite{AlaBroMon} identified a candidate expression for $H_{c_1}$ in the case of the ball. Then, Baldo, Jerrard, Orlandi, and Soner \cite{BalJerOrlSon2}, characterized to leading order the first critical field in 3D for a general bounded domain. In the forthcoming papers \cites{Rom2,RomSanSer}, our purpose is to derive with high accuracy $H_{c_1}$, and to characterize the behavior of global minimizers for the full three-dimensional Ginzburg-Landau energy below and near this value. Our arguments crucially use the $\ve$-level estimates in Theorem \ref{Theorem1}.

\medskip 
Since magnetic monopoles do not exist in Maxwell's electromagnetism theory, we may assume that $H_{\ex}\in L_{\loc}^2(\R^3,\R^3)$ is divergence-free. Then, there exists a vector potential $A_{\ex}\in H^1_{\loc}(\R^3,\R^3)$ such that 
$$
\curl A_{\ex}=H_{\ex}, \ \diver A_{\ex}=0\ \mathrm{in}\ \R^3\quad \mathrm{and}\quad A_{\ex}\cdot \nu=0\ \mathrm{on}\ \partial\Omega,
$$
where hereafter $\nu$ denotes the outer unit normal to $\partial\Omega$.

Let us introduce the space
$$
H_{\curl}\colonequals \{ A\in H^1_{\loc}(\R^3,\R^3) \ |\ \curl A\in L^2(\R^3,\R^3)\}.
$$ 
The functional $GL_\varepsilon(u,A)$ is well defined for any pair $(u,A)\in H^1(\Omega,\mathbb{C})\times [A_\ex +H_{\curl}]$. 
The following result is a direct consequence of Theorem \ref{Theorem1}.

\begin{corollary} Theorem \ref{Theorem1} holds true if the hypothesis that $(u_\ve,A_\ve)\in H^1(\Omega,\C)\times H^1(\Omega,\R^3)$ is a configuration such that $F(u_\ve,A_\ve)\leq M|\log \ve|^m$ is replaced with the assumption that $(u_\ve,A_\ve)\in H^1(\Omega,\mathbb{C})\times [A_\ex +H_{\curl}]$ is a configuration such that $GL_\ve (u_\ve,A_\ve)\leq M|\log \ve|^m$.
\end{corollary}

\begin{remark} In particular, this result holds true if $(u_\ve,A_\ve)$ is a minimizing configuration for $GL_\ve$ in $ H^1(\Omega,\C)\times [A_\ex +H_{\curl}]$ and $\|H_\ex\|^2_{L^2(\Omega)}\leq M|\log \ve|^m$. Indeed, this follows by observing that
$$
GL_\ve(u_\ve,A_\ve)=\inf_{H^1(\Omega,\C)\times [A_\ex +H_{\curl}]} GL_\ve(u,A)\leq GL_\ve(1,A_{\ex})=\int_\Omega |A_{\ex}|^2\leq C \int_\Omega |H_{\ex}|^2,
$$
for some universal constant $C$. 
\end{remark}

\subsection{A quantitative product estimate for the study of Ginzburg-Landau dynamics}\label{sec:product}
In this section, we consider the special case $\Omega=(0,T)\times \omega$, with $T>0$ and $\omega\subset \R^2$, i.e. we deal with configurations $(u_\ve,A_\ve)$ which depend both on space and time. We use coordinates $(t,x_1,x_2)$ in three-space and denote $\nabla=(\partial_{x_1},\partial_{x_2})$, $\nabla^\perp =(-\partial_{x_2},\partial_{x_1})$. We consider gauges of the form
$$
A_\ve=(\Phi_\ve,B_\ve),
$$
where $\Phi_\ve:(0,T)\times \omega\to \R$ and $B_\ve:(0,T)\times \omega \to \R^2$. 
By using the notation
$$
X^\perp=(-X_2,X_1),\quad \curl X=\partial_{x_1}X_2-\partial_{x_2}X_1
$$
for vector fields $X$ in the plane, we observe that
$$
\curl A_\ve=(\curl B_\ve,\partial_t B_\ve^\perp -\nabla^\perp \Phi_\ve).
$$
As above, the vorticity in three-space is defined by 
$$
\mu(u_\ve,A_\ve)\colonequals \curl (\langle \partial_t u_\ve -i\Phi_\ve u_\ve, iu_\ve \rangle +\Phi_\ve, \langle \nabla u_\ve -iB_\ve u_\ve, iu_\ve \rangle +B_\ve).
$$
It can be written as
$$
\mu(u_\ve,A_\ve)=(J_\ve,V_\ve),
$$
where
$$
J_\ve\colonequals \curl\langle \nabla u_\ve,iu_\ve\rangle + \curl(1-|u_\ve|)^2 B_\ve
$$
is the space-only vorticity and
$$
V_\ve\colonequals 2\langle i\partial_t u_\ve,\nabla^\perp u_\ve\rangle+\partial_t(1-|u_\ve|^2)B_\ve^\perp-\nabla^\perp(1-|u_\ve|^2)\Phi_\ve
$$
is the velocity. Since $\partial \mu(u_\ve,A_\ve)$ relative to $\Omega$, we have the relation
$$
\partial_t J_\ve + \diver V_\ve =0,
$$
which means that the vorticity $J_\ve$ is transported by $V_\ve$, hence the name velocity.

We let $M_\ve$ be a quantity such that
\begin{equation}\label{condM}
\forall q>0, \quad \lim_{\ve\to 0}M_\ve \ve^q=0,\quad \lim_{\ve\to 0}\frac{\log\ve}{M_\ve^q}=0,\ \mathrm{and}\quad \lim_{\ve\to 0}\frac{\log M_\ve}{\log \ve}=0.
\end{equation}
For example $M_\ve=e^{\sqrt{|\log \ve|}}$ will do. 
Under the hypothesis of Theorem \ref{Theorem2}, our construction provides an approximation for the vorticity in three-space, which in particular yields an approximation for the velocity $V_\ve$ and for the space-only vorticity $J_\ve$. This combined with ideas from \cites{SanSer3,Ser2017}, yields a quantitative three-dimensional product estimate, which allows to control the velocity.

\begin{theorem}\label{thm:prodesti} Assume $\omega$ to be Lipschitz and let $M_\ve$ be as above. Consider a function $f\in C_0^{0,1}([0,T]\times \overline \omega)$ and a spatial vector field $X\in C_0^{0,1}([0,T]\times \overline \omega)$. For any $m,n,M,\Lambda>0$ with $n>\frac13(2-m)$ there exist a universal constant $C>0$ and $\ve_0>0$ depending only on $m,n,M,\Lambda,f$ and $X$, such that, for any $\ve<\ve_0$, if $(u_\ve,A_\ve)=(u_\ve,(\Phi_\ve,B_\ve))\in H^1(\Omega,\C)\times H^1(\Omega,\R^3)$ is a configuration such that $F_\ve(u_\ve,A_\ve)\leq M|\log \ve|^m$ then
\begin{align*}
\int_{(0,T)\times \omega} \frac{|f|^2}\Lambda &|\partial_t u_\ve-iu_\ve \Phi_\ve|^2 +\int_{(0,T)\times \omega}\Lambda|X\cdot (\nabla u_\ve -iu_\ve B_\ve)|^2\\
&\geq \left( |\log \ve|  - C\log M_\ve \right) \left|  \int_{(0,T)\times \omega} f\nu_\ve \wedge (-X_2dx_1+X_1dx_2) \right|\\
&\hspace*{0.5cm}-C\int_{(0,T)\times \omega}\max(|\nu_\ve\wedge dx_1|,|\nu_\ve\wedge dx_2|)+O\left(|\log \ve|^{-\frac12(m+3n)+1}\right),
\end{align*}
where $\nu_\ve$ is the polyhedral $1$-dimensional current $\nu_\ve$ associated to $(u_\ve,A_\ve)$ by Theorem \ref{Theorem2}.
\end{theorem}
\begin{remark}
By choosing
$$
\Lambda=\left(\dfrac{\int_{(0,T)\times \omega} |f|^2|\partial_t u_\ve-iu_\ve \Phi_\ve|^2}{\int_{(0,T)\times \omega}|X\cdot (\nabla u_\ve -iu_\ve B_\ve)|^2}\right)^{1/2},
$$
one obtains a left-hand side in the form of a product (plus error terms), hence the name product estimate. It is worth to mention that the dependence of $\ve$ in terms of $\Lambda,f,$ and $X$ can be found in the proof (see Section \ref{Sec:ProdEstimate}).
\end{remark}

\subsection{A word about the proof of the main results}\label{strategy}
The subtle point of the proof is to obtain a lower bound for the free energy at the $\ve$-level. Here is where minimal connections play a role. The idea of obtaining lower bounds for Ginzburg-Landau energies via the use of minimal connections was first introduced in \cite{San}, in the case of the energy without magnetic field $E_\ve(u)$. When trying to apply this kind of method to obtain lower bounds for the full functional $GL_\ve(u,A)$, the main obstacle is that as soon as the external magnetic field is of the order of the first critical field, the number of vortices is a priori unbounded as $\ve\to 0$. The main challenge in getting a lower bound that works at the $\ve$-level is thus to keep track of the dependence of all the estimates on $\ve$ and $\de(\ve)$, keeping into account that the number of vortex filaments may be unbounded.

Our method goes as follows. 
The choice of grid allows us to show that the restriction of the vorticity to the boundary of a cube $\CC$ can be well approximated by
$$
2\pi \left(\sum_{i=1}^k \delta_{p_i}-\sum_{i=1}^k \delta_{n_i}\right),
$$
where the points $p_i$'s are the (non-necessarily distinct) positive singularities and the points $n_i$'s are the (non-necessarily distinct) negative singularities. We remark that the number of points and their locations depend on $\ve$.

By \cite{BreCorLie}, we know that there exists a $1$-Lipschitz function $\zeta:\R^3\to\R^3$ such that
$$
\sum_{i=1}^k \zeta(p_i)-\sum_{i=1}^k \zeta(n_i)=L(\AA),
$$
where $L(\AA)$ is the length of the minimal connection associated to the configuration of points $\AA=\{p_1,\dots,p_k,n_1,\dots,n_k\}$. Since $|\nabla \zeta|\leq 1$, the co-area formula gives
$$
\int_\CC e_\ve(u_\ve)\geq \int_\CC e_\ve(u_\ve)|\nabla \zeta| \geq \int_{t\in \R}\int_{\Sigma_t} e_\ve(u_\ve)d\H^2dt,
$$
where $e_\ve(u_\ve)=\frac12|\nabla u_\ve|^2+\frac1{4\ve^2}(1-|u_\ve|^2)^2$ and $\Sigma_t=\{\zeta=t\}\cap \CC$. 

At this point, a vortex ball construction on a surface is necessary. Roughly speaking, if $\Sigma_t$ is nice enough and $|u_\ve|\geq 1/2$ on $\partial\Sigma_t$, we expect
$$
\int_{\Sigma_t} e_\ve(u_\ve)d\H^2\geq \pi \deg(u_\ve/|u_\ve|,\partial\Sigma_t)\left(\log \frac1\ve -O(\log|\log \ve|) \right).
$$
It turns out that, for most $t$'s, we have
$$
\deg(u_\ve/|u_\ve|,\partial\Sigma_t)=\#\{i \ | \ \zeta(p_i)>t\}-\#\{i \ | \ \zeta(n_i)>t\}.
$$
By noting that
$$
\int_{t\in\R}\#\{i \ | \ \zeta(p_i)>t\}-\#\{i \ | \ \zeta(n_i)>t\}dt=\sum_{i=1}^k \zeta(p_i)-\sum_{i=1}^k \zeta(n_i)=L(\AA)\approx \frac{1}{2\pi}|\nu_\ve|(\CC),
$$
we are led to
$$
\int_\CC e_\ve(u_\ve) \geq \frac12 |\nu_\ve|(\CC)\left(\log \frac1\ve -O(\log|\log \ve|) \right)+\mathrm{small\ error}.
$$
Unfortunately, we cannot really use the function $\zeta$ in the previous argument, because its regularity is not sufficient to apply the ball construction on most of its level sets. To bypass this issue, we construct a smooth approximation of this function. The difficulties appear when trying to control the errors involved in the previously described method, because a quantitative bound of the second fundamental form of most of the level sets of our smooth approximation of the function $\zeta$ is needed.

In a similar but more involved way, by assuming that $\partial\Omega$ is of class $C^2$, we can obtain a lower bound close to the boundary of the domain. 

It is worth to mention that by density arguments we can assume without loss of generality that $u$ and $A$ are of class $C^1$ in some proofs of this paper. 

\subsection*{Outline of the paper}
The paper is organized as follows.

In Section \ref{Preliminaries} we introduce some basic objects and spaces that are used throughout the paper, we recall some facts from the theory of currents and differential forms, and we describe the choice of grid. 

In Section \ref{Sec:Ball} we provide the ball construction method on a surface, which is one of the key tools used to obtain the lower bound for the free energy.

In Section \ref{Sec:2DVortEstimate} we show a 2D vorticity estimate. The main difference with classical results of the same kind is the space in which we prove the estimate.

In Section \ref{Sec:3Dconstruction} we start by reviewing the concept of minimal connection. Then, we introduce the function $\zeta$ and the function $\zeta$ for $d_{\partial\Omega}$, and state three technical propositions concerning quantitative smooth approximations of these functions. Finally, we present our 3D vortex approximation construction.

Section \ref{Sec:LBcubes} is devoted to the proof of a lower bound for the energy without magnetic field in the union of cubes of the grid, while in Section \ref{Sec:LBboundary} we provide a similar estimate near the boundary of the domain. In these proofs we crucially use the results of Section \ref{Sec:Ball} and Section \ref{Sec:3Dconstruction}.

In Section \ref{Sec:proofMain} we present the proofs of Theorem \ref{Theorem1} and Theorem \ref{Theorem2}, which use the lower bounds obtained in Section \ref{Sec:LBcubes} and Section \ref{Sec:LBboundary}, as well as the 2D vorticity estimate of Section \ref{Sec:2DVortEstimate}.

In Section \ref{Sec:ProdEstimate} we prove the quantitative product estimate for the study of Ginzburg-Landau dynamics.

In Appendix \ref{Sec:AppendixA} we construct a quantitative smooth approximation of the function $\zeta$. In Appendix \ref{Sec:AppendixB} and Appendix \ref{Sec:AppendixC}, we present two different methods to do the same for the function $\zeta$ for $d_{\partial\Omega}$. These are the most technical parts of the paper.

\subsection*{Acknowledgements} I would like to warmly thank my Ph.D. advisors Etienne Sandier and Sylvia Serfaty for suggesting this problem, for their helpful advice, careful reading, and useful comments.

%%%%%%%%%%%%%%%%%%%%%%%%%%%%%%%%%%%%%%%%%%%%%%%%%%%%%%%%%%%%%%%%%%%%%%%%%%%%%%%%%%%%%%%%%%%%%%%%%
 
\section{Preliminaries}\label{Preliminaries}
It is useful to introduce certain concepts and notation from the theory of currents and differential forms. 
We recall that in Euclidean spaces vector fields can be identified with $1$-forms. Indeed, the vector field $F=(F_{x_1},F_{x_2},F_{x_3})$ can be identified with the $1$-form $F_{x_1}dx_1+F_{x_2}dx_2+F_{x_3}dx_3$. We
use the same notation for both the vector field and the $1$-form. 

It is also convenient to recall that a vector field $F$ satisfying the boundary condition $F\times \nu=0$ on $\partial \Omega$ is equivalent to a $1$-form $F$ such that $F_T=0$ on $\partial \Omega$. Here $F_T$ denotes the tangential component of $F$ on $\partial \Omega$. 

We define the superconducting current of a pair $(u,A)\in H^1(\Omega,\C)\times H^1(\Omega,\R^3)$ as the $1$-form
$$
j(u,A)=(iu,d_A u)=\sum_{k=1}^3 (iu,\partial_k u-iA_ku)dx_k.
$$
It is related to the vorticity $\mu(u,A)$ of a configuration $(u,A)$ through
$$
 \mu(u,A)=dj(u,A)+dA.
$$
Thus $\mu(u,A)$ is an exact $2$-form in $\Omega$ acting on couples of vector fields $(X,Y)\in \R^3\times \R^3$ with the standard rule that $dx_i\wedge dx_j(X,Y)=X_iY_j-X_jY_i$. It can also be seen as a $1$-dimensional current, which is defined through its action on $1$-forms by the relation
$$
\mu(u,A)(\phi)=\int_\Omega \mu(u,A)\wedge \phi.
$$
We recall that the boundary of a $1$-current $T$ relative to a set $\Theta$, is the $0$-current $\partial T$ defined by
$$
\partial T(\phi)=T(d\phi)
$$
for all smooth compactly supported $0$-form $\phi$ defined in $\Theta$. In particular, an integration by parts shows that the $1$-dimensional current $\mu(u,A)$ has zero boundary relative to $\Omega$. We denote by $|T|(\Theta)$ the mass of a $1$-current $T$ in $\Theta$.

For $\alpha \in (0,1]$ we let $C^{0,\alpha}(\Omega)$ denote the space of $1$-forms $\phi$ such that $\|\phi\|_{C^{0,\alpha}(\Omega)}<\infty$. $C_0^{0,\alpha}(\Omega)$ denotes the space of $1$-forms $\phi\in C^{0,\alpha}(\Omega)$ such that $\phi=0$ on $\partial \Omega$, while $C_T^{0,\alpha}(\Omega)$ denotes the space of $1$-forms $\phi\in C^{0,\alpha}(\Omega)$ such that $\phi_T=0$ on $\partial \Omega$. The symbol $^*$ is used to denote the dual spaces.

\medskip 
We next recall the definition of topological degree.
\begin{definition} Let $\Sigma$ be a complete oriented surface in $\R^3$. If $\Theta\subset \Sigma$ is a smooth domain, and the map $u:\Sigma\to \C$ does not vanish on $\partial\Theta$, we can define the degree $\deg(u/|u|,\partial\Theta)$ of $u$ restricted to $\partial\Theta$ to be the winding number of the map $u/|u|:\partial\Theta\to S^1$. 
\end{definition}
We observe that, because $\Sigma$ is assumed to be oriented, $\partial\Theta$ carries a natural orientation. In the case that $\partial\Theta$ is not smooth, the topological degree can still be defined by approximation. 

Hereafter $\H^d$ denotes the $d$-dimensional Hausdorff measure, for $d\in \N$.
When meaningful, we sometimes use the notation
$$
F_\ve(u,A,\Theta)\colonequals \int_\Theta e_\ve(u,A)d\H^2,\quad E_\ve(u,\Theta)\colonequals \int_\Theta e_\ve(u)d\H^2,
$$
with $e_\ve(u,A)\colonequals \frac12 |\nabla_A u|^2+\frac1{4\ve^2}(1-|u|^2)^2+|\curl A|^2$, $e_\ve(u)\colonequals\frac12 |\nabla u|^2+\frac1{4\ve^2}(1-|u|^2)^2$.

\subsection{Choice of grid}
Let us fix an orthonormal basis $(e_1,e_2,e_3)$ of $\R^3$ and consider a grid $\GG=\GG(a,R,\de)$ given by the collection of closed cubes $\CC_i\subset \R^3$ of side-length $\de=\de(\ve)$ (conditions on this parameter are given in the lemma below).
In the grid we use a system of coordinates with origin in $a \in \Omega$ and orthonormal directions given by the rotation of the basis $(e_1,e_2,e_3)$ with respect to $R \in SO(3)$. 
From now on we denote by $\RR_1$ (respectively $\RR_2$) the union of all edges (respectively faces) of the cubes of the grid. We have the following lemma.

\begin{lemma}[Choice of grid]\label{Lemma:Grid} 
For any $\gamma\in(0,1)$ there exists a rotation $R_0(\gamma)\in SO(3)$ and constants $c_0(\gamma),c_1(\gamma)>0$, $\de_0(\Omega)\in(0,1)$ such that, for any $\ve,\de>0$ satisfying 
$$
\ve^{\frac{1-\gamma}2}\leq c_0\quad \mathrm{and}\quad c_1\ve^{\frac{1-\gamma}4}\leq \de \leq \de_0,
$$
if $(u_\ve,A_\ve)\in H^1(\Omega,\C)\times H^1(\Omega,\R^3)$ is a configuration such that $F_\ve(u_\ve,A_\ve)\leq \ve^{-\gamma}$ then there exists $b_\ve\in \Omega$ such that the grid $\GG(b_\ve,R_0,\delta)$ satisfies
\begin{subequations}\label{propGrid}
\begin{equation}\label{prop1Grid}
|u_\ve|>5/8\quad \mathrm{on}\ \RR_1(\GG(b_\ve,R_0,\de))\cap \Omega,
\end{equation}
\begin{equation}\label{prop2Grid}
\int\limits_{\RR_1(\GG(b_\ve,R_0,\delta))\cap \Omega}e_\ve(u_\ve,A_\ve)d\H^1\leq  C \de^{-2}F_\ve(u_\ve,A_\ve),
\end{equation}
\begin{equation}\label{prop3Grid}
\int\limits_{\RR_2(\GG(b_\ve,R_0,\delta))\cap \Omega}e_\ve(u_\ve,A_\ve)d\H^2\leq  C \de^{-1}F_\ve(u_\ve,A_\ve),
\end{equation}
\end{subequations}
where $C$ is a universal constant.
\end{lemma}
\begin{proof}
First, let us observe that, by the Cauchy-Schwarz inequality and the co-area formula, we have
\begin{align*}
4F_\ve(u_\ve,A_\ve)&\geq \int_\Omega |\nabla |u_\ve||^2+\frac1{\ve^2}(1-|u_\ve|^2)^2\\
	    &\geq\int_\Omega \frac{|\nabla |u_\ve||(1-|u_\ve|^2)}{\ve}\\
	    &= \int_{t=0}^\infty \left(\int_{\{|u|=t\}}\frac{(1-t^2)}{\ve}d\mathcal{H}^2\right)dt
\end{align*}
Define $T\colonequals \{t\in [5/8,3/4]\ | \ \mbox{Area}(\{|u_\ve|= t\})\leq \ve^\alpha \}$ for $\alpha\colonequals \frac{1-\gamma}2$. From the previous estimate we deduce that 
$$
|T|\geq1/8-C\ve^{1-\alpha} F_\ve(u_\ve,A_\ve),
$$ 
where hereafter $C>0$ denotes a universal constant that may change from line to line. It is easy to check that there exists a constant $c_0(\gamma)>0$ such that $|T|>0$ for any $\ve>0$ satisfying $\ve^{\frac{1-\gamma}2}\leq c_0$.

We observe that by integral geometry formulas (see for instance \cites{LanBook,SanBook}), for any $t\in[0,3/4]$, we have
$$
\mbox{Area}(\{|u_\ve|=t\})= c\int\limits_{R \in SO(3)} \int\limits_{h\in \R^3} \# \left(\{|u_\ve|=t\}\cap L_{R,h}\cap \Omega\right)d\LL(h)d\LL(R),
$$
where $L_{R,h}$ is the rotation with respect to $R \in SO(3)$ and the translation with respect to $h\in \R^3$ of a fixed line $L$ in $\R^3$, 
$\#(A)$ denotes the number of points of the set $A$, and $c$ is a constant depending only on the dimension of the Euclidean space.

We fix a point $a\in \Omega$ and choose $\de_0=\de_0(\Omega)\in(0,1)$ such that $\{a+[0,\de]^3\}\subset \Omega$ for any $0<\de<\de_0$. Observe that, up to an adjustment of $c$, we have
$$
\mathrm{Area}(\{|u_\ve|=t\})= 
\frac{c}{\de}\int\limits_{R \in SO(3)} \int\limits_{b\in \{a+[0,\de]^3\}} \# \left(\{|u_\ve|=t\}\cap \RR_1(\GG(b,R,\de))\cap \Omega\right)d\LL(b)d\LL(R).
$$
Fix $t_0 \in T$ and define
\begin{align*}
G_0&\colonequals \{(R, b)\ | \ \ R \in SO(3),\ b\in \{a+[0,\de]^3\}, \ \{|u_\ve|=t_0\}\cap \RR_1(\GG(b,R,\de)\neq \emptyset  \}.
\end{align*}
By noting that 
$$
|G_0|\leq \frac{\de}c \mathrm{Area}(\{|u_\ve|=t_0\})\leq \frac{\de\ve^\alpha}{c},
$$
we deduce that there exists a fixed rotation $R_0\in SO(3)$ such that 
$$
B_I\colonequals \{b\in \{a+[0,\de]^3\}\ | \  \{|u_\ve|=t_0\}\cap \RR_1(\GG(b,R_0,\de)\neq \emptyset  \}
$$
satisfies $|B_I|\leq C\de \ve^\alpha$. 
 
We observe that, for any $b\in \{a+[0,\de]^3\}\setminus B_I$, 
$$
\mathrm{either} \quad  |u_\ve|>t_0\quad \mathrm{or} \quad |u_\ve|<t_0\quad \mathrm{on}\ \RR_1(\GG(b,R_0,\delta))\cap \Omega.
$$ 
We let
$$
B_{II}\colonequals \left\{b\in \{a+[0,\de]^3\}\setminus B_I \ | \ \{|u_\ve|<t_0\}\cap \RR_1(\GG(b,R_0,\de)\neq \emptyset\right\}
$$
and observe that, for every $b\in B_{II}$, we have $(1-|u_\ve|^2)\geq (1-t_0^2)$. This implies that
$$
\frac{(1-t_0^2)^2}{4\ve^2}|B_{II}|\leq 
\int\limits_{b\in B_{II}} \int\limits_{\RR_1(\GG(b,R_0,\delta))\cap \Omega}e_\ve(u_\ve,A_\ve)d\H^1 d \mathcal L(b)\leq F_\ve(u_\ve,A_\ve)
$$
and thus $|B_{II}|\leq C \ve^2 F_\ve(u_\ve,A_\ve)$.

Now, we define $B_{good}\colonequals \{a+[0,\delta]^3\}\setminus (B_I\cup B_{II})$. Observe that 
$$
|B_{good}|\geq \delta^3-C(\delta \ve^\alpha+\ve^{2-\gamma})
$$
and that there exists a constant $c_1>0$ such that $|B_{good}|\geq \delta^3/2$ for any $\ve,\de>0$ satisfying $c_1\ve^{\frac{\alpha}2}\leq \de$.
Moreover, for any $b\in B_{good}$, we have
$$
|u_\ve|>t_0\quad \mathrm{on}\ \RR_1(\GG(b,R_0,\delta))\cap \Omega.
$$
Next, using a mean value argument we choose $b=b_\ve\in B_{good}$ in such a way that 
$$
\int\limits_{\RR_n(\GG(b_\ve,R_0,\delta))\cap \Omega}e_\ve(u_\ve,A_\ve)d\H^n\leq  C\delta^{n-3} F_\ve(u_\ve,A_\ve)\quad \mathrm{for} \ n=1,2.
$$
First, by \cite{AlbBalOrl}*{Lemma 8.4} there exists $b_\ve \in B_{good}$ such that, for $n=1,2$,
$$
\int\limits_{\RR_n(\GG(b_\ve,R_0,\delta))\cap \Omega}e_\ve(u_\ve,A_\ve)d\H^n\leq \frac2{|B_{good}|}\int\limits_{B_{good}}\int\limits_{\RR_n(\GG(b,R_0,\delta))\cap \Omega}e_\ve(u_\ve,A_\ve)d\H^nd\mathcal L(b).
$$
Second, arguing as in the proof of \cite{AlbBalOrl}*{Lemma 3.11}, we have
$$
\frac1{\de^3}\int\limits_{\{a+[0,\delta]^3\}}\de^{3-n}\int\limits_{\RR_n(\GG(b,R_0,\delta))\cap \Omega}e_\ve(u_\ve,A_\ve)d\H^n d\mathcal L(b)=CF_\ve(u_\ve,A_\ve)\quad \mathrm{for}\ n=1,2.
$$
Then, we deduce that
$$
\int\limits_{\RR_n(\GG(b_\ve,R_0,\delta))\cap \Omega}e_\ve(u_\ve,A_\ve)d\H^n\leq  C\frac{\delta^3}{|B_{good}|}\delta^{n-3} F_\ve(u_\ve,A_\ve)\quad \mathrm{for}\ n=1,2.
$$
Recalling that $|B_{good}|\leq \de^3/2$, the lemma follows.
\end{proof}
From now on we drop the cubes of the grid $\GG(b_\ve,R_0,\de)$, given by Lemma \ref{Lemma:Grid}, whose intersection with $\R^3\setminus \Omega$ is non-empty. We also define 
\begin{equation}
\Theta\colonequals \Omega \setminus \cup_{\CC_l\in \GG} \CC_l\quad\mathrm{and}\quad \partial \GG \colonequals \partial \left(\cup_{\CC_l\in \GG} \CC_l \right)\label{unioncubes}.
\end{equation}
Observe that, in particular, $\partial \Theta=\partial \GG \cup \partial \Omega$.

We remark that $\GG(b_\ve,R_0,\de)$ carries a natural orientation. The boundary of every cube of the grid will be oriented accordingly to this orientation. Each time we refer to a face $\omega$ of a cube $\CC$, it will be considered to be oriented with the same orientation of $\partial\CC$. If we refer to a face $\omega\subset \partial\GG$, then the orientation used is the same of $\partial\GG$.

%%%%%%%%%%%%%%%%%%%%%%%%%%%%%%%%%%%%%%%%%%%%%%%%%%%%%%%%%%%%%%%%%%%%%%%%%%%%%%%%%%%%%%%%%%%%%%%%%

\section{The ball construction method on a surface}\label{Sec:Ball}
In this section we use the method of Jerrard introduced in \cite{Jer} in order to construct balls containing all the zeros of $u$ on a surface. This allows us to obtain a lower bound for the energy without magnetic field. The construction given here follows the one made by Sandier in \cite{San} that corresponds to an adaptation of the method of Jerrard. The following is the main result of this section, which is an extension of \cite{San}*{Proposition 3.5}.
\begin{proposition}\label{ball2}
Let $\tilde \Sigma$ be a complete oriented surface in $\R^3$ whose second fundamental form is bounded by $1$. Let  $\Sigma$ be a bounded open subset of $\tilde \Sigma$. For any $m,M>0$ there exists $\ve_0(m,M)>0$ such that, for any $\ve<\ve_0$, if $u_\ve \in H^1(\Sigma,\C)$ satisfies
\begin{equation}\label{AprioriUpperBound}
E_\ve(u_\ve,\Sigma)\leq M|\log \ve|^m
\end{equation}
and 
$$
|u(x)|\geq \frac12\quad \textnormal{if}\ \d(x,\partial \Sigma)<1, 
$$
where $\d(\cdot,\cdot)$ denotes the distance function in $\tilde\Sigma$, then, letting $d$ be the winding number of $u_\ve/|u_\ve|:\partial\Sigma\to S^1$ and $M_\ve=M|\log \ve|^m$, we have
$$
E_\ve(u_\ve,\Sigma)\geq \pi |d| \left( \log\frac1\ve -\log M_\ve\right).
$$
\end{proposition}
To prove Proposition \ref{ball2} we follow almost readily the proofs of \cite{Jer} and \cite{San}. 
\subsection{Main steps}
Let us define the essential null set $S_E(u_\ve)$ of $u_\ve$ to be the union of those connected components $U_i$ of $\{x \ | \ |u_\ve(x)|<1/2\}$ such that $\deg (u_\ve/|u_\ve|,\partial U_i)\neq 0$.

In the rest of this section each time we refer to a ball $B$ of radius $r$ we mean a geodesic ball of radius $r$ in $\tilde \Sigma$.

\medskip
First, we include $S_E(u_\ve)$ in the union of well-chosen disjoint ``small" balls $B_i$ of radii $r_i>\ve$ such that
$$
E_\ve(u_\ve,B_i)\geq \frac{r_i}{C\ve},
$$
where the constant $C$ does not depend on the second fundamental form of $\Sigma$ when it is assumed to be bounded by $1$. This is possible according to the following lemma.

\begin{lemma}\label{lem1Ball}
Under the hypotheses of Proposition \ref{ball2}, there exist $C,r_0>0$ such that, for any $\ve>0$, there exist disjoint balls $B_1,\dots, B_k$ of radii $r_i$ such that
\begin{enumerate}[font=\normalfont,leftmargin=*]
\item $r_i\geq \ve$ for all $i\in \{1,\dots,k\}$.
\item $S_E(u)\subset \cup_i B_i$ and $B_i\cap S_E(u)\neq \emptyset$ for all $i\in \{1,\dots,k\}$.
\item For all $i\in \{1,\dots,k\}$,
$$
E_\ve(u_\ve,B_i\cap \Sigma)\geq \frac{\min\{r_i,r_0,1\}}{C\ve}.
$$
\end{enumerate}
\end{lemma}
Then the proof involves dilating the balls $B_i$ into balls $B_i'$ by combining them with annuli. A lower bound for $E_\ve(u_\ve,B_i')$ is obtained by combining the lower bound for $E_\ve(u_\ve,B_i)$ and a lower bound for $E_\ve(u_\ve,B_i'\setminus \overline B_i)$.

\begin{lemma}\label{lem2Ball}
Under the hypotheses of Proposition \ref{ball2}, there exist $C,\ve_0,r_0>0$ such that, for any $0<\ve<s<r<r_0$, if $B_r,B_s\subset \Sigma$ are two concentric balls and if $S_E(u_\ve)\cap(B_r\setminus \overline B_s)=\emptyset$ then, letting $d\colonequals\deg(u_\ve/|u_\ve|,\partial B_r)$,
$$
E_\ve(u_\ve,B_r\setminus \overline B_s)\geq |d|\left(\Lambda_\ve\left(\frac{r}{|d|}\right)-\Lambda_\ve\left(\frac{s}{|d|}\right) \right),
$$
where $\Lambda_\ve:\R_+\to \R_+$ is a function that satisfies the following properties
\begin{enumerate}[font=\normalfont,leftmargin=*]
\item $\Lambda_\ve(t)/t$ is decreasing.
\item $\sup_{t\in\R_+} \Lambda_\ve(t)/t\leq 1/(C\ve)$.
\item If $0<\ve<\ve_0$ and $\ve <t<r_0$ then
$$
\left|\Lambda_\ve(t)-\pi\log \frac t\ve \right|\leq C.
$$
\end{enumerate}
\end{lemma}
By taking into consideration the following adaptation of \cite{San}*{Lemma $3.12$}, the proofs of the previous two lemmas are straightforward modifications of the proofs  of
\cite{San}*{Lemma $3.8$} and of \cite{San}*{Lemma $3.9$}.

\begin{lemma}
Let $S_t(x)$ denote the geodesic circle in $\tilde\Sigma$ of radius $t$ centered at $x\in \Sigma$. Under the hypotheses of Proposition \ref{ball2}, there exist $C,\ve_0,r_0>0$ such that, for any $x\in \Sigma$ and for any $\ve,t>0$ satisfying $\ve <\ve_0$ and $\ve<t<r_0$, if $|u_\ve|\leq 1$ on $S_t(x)$ then 
$$
E_\ve(u_\ve,S_t(x))\geq \pi m^2\left(\frac{|d|}{t}-C\right)^++\frac{(1-m)^C}{C\ve},
$$
where $\displaystyle m\colonequals\inf_{y\in S_t}|u_\ve(y)|$ and 
$$
d\colonequals\left\{ 
\begin{array}{cl}
\deg(u_\ve/|u_\ve|,S_t(x))& m\neq 0,\\
0& m=0.
\end{array}
\right.
$$
\end{lemma}
\begin{proof}
By observing that the constants $r_0, r$, and $C$ involved in (B.8), (B.9), and (B.12) in the proof of \cite{San}*{Lemma 3.12} can be chosen independently of the second fundamental form of $\tilde \Sigma$ when it is assumed to be bounded by $1$, then the proof is verbatim the same as that of \cite{San}*{Lemma 3.2}. 
\end{proof}

Lemma \ref{lem1Ball} and Lemma \ref{lem2Ball} allow one to prove the following result, whose proof is a straightforward modification of the proof of \cite{San}*{Proposition 3.10}.
\begin{proposition}\label{prop1Ball}
For any $\ve>0$, let $\{B_i\}_i$ be the family of balls of radii $r_i$ given by Lemma \ref{lem1Ball}. Let
$$
d_i\colonequals\left\{
\begin{array}{cl}
\deg(u_\ve/|u_\ve|,\partial B_i)& \textnormal{if } \overline{B}_i\subset \Sigma_\ve,\\
0& \textnormal{otherwise},
\end{array}\right.
$$
and
$$
t_0\colonequals \min_{\{i \ | \ d_i\neq 0\}}\frac{r_i}{|d_i|}\ (\mathrm{with}\ t_0\colonequals +\infty \textnormal{ if } d_i=0 \ \mathrm{for\ every}\ i).
$$
Then, for any $t\geq t_0$, there exists a family of disjoint geodesic balls $B_1(t),\dots,B_{k(t)}(t)$ of radii $r_i(t)$ in $\tilde \Sigma$ such that 
\begin{enumerate}[font=\normalfont,leftmargin=*]
\item $S_E(u)\subset \cup_i B_i(t)$ and $S_E(u)\cap B_i(t)\neq \emptyset$ for all $i\in\{1,\dots,k(t)\}$.
\item For all $i\in\{1,\dots,k(t)\}$, if $\overline B_i(t)\subset \Sigma$ then $r_i(t)\geq t|d_i(t)|$, where 
$$
d_i(t)\colonequals\deg(u_\ve/|u_\ve|,\partial B_i(t)).
$$
\item For all $i\in\{1,\dots,k(t)\}$,
$$
E_\ve(u_\ve,B_i(t)\cap\Sigma)\geq \min\{r_i(t),r_0,1\}\frac{\Lambda_\ve(t)}{t}.
$$
\end{enumerate}
\end{proposition}

\begin{proof}[Proof of Proposition \ref{ball2}]
We assume that $d\neq 0$, otherwise the result is trivial. Apply Lemma \ref{lem1Ball}, call the resulting balls $B_1,\dots, B_k$, and call $r_1,\dots,r_k$ their radii. From Lemma \ref{lem1Ball} and \eqref{AprioriUpperBound}, we have
$$
\min\{r_i,r_0,1\}\leq C\ve E_\ve(u_\ve,B_i\cap \Sigma)\leq C\ve M_\ve \leq C\ve |\log \ve|^m,
$$
where throughout the proof $C=C(M)>0$ denotes a constant that may change from line to line.
We deduce that there exists $\ve_0(m,M)>0$ such that, for any $i\in\{1,\dots,k\}$ and for any $\ve<\ve_0$,
\begin{equation}\label{radii}
r_i=\min\{r_i,r_0,1\}\leq C\ve M_\ve\quad \mathrm{and}\quad r_i\leq \frac12.
\end{equation}
Since $\d(S_E(u_\ve),\partial \Sigma)<1$ and $B_i\cap S_E(u_\ve)\neq 0$, we conclude that $\overline{B}_i\subset \Sigma$. Thus
\begin{equation}\label{deg}
\sum_{i=1}^k \deg(u_\ve/|u_\ve|,B_i)=d\neq 0.
\end{equation}
As in Proposition \ref{prop1Ball}, let 
$$
t_0=\min_{\{i\ |\ d_i\neq 0\}}\frac{r_i}{|d_i|}.
$$
From \eqref{radii} and \eqref{deg}, we get that $t_0\leq C\ve M_\ve$. Fix $\alpha\in (0,1)$. By reducing the constant $\ve_0$, we deduce that $t_0\leq M_\ve^{-1} |\log \ve|^{\alpha}$ for any $\ve<\ve_0$. Therefore, we may apply Proposition \ref{prop1Ball} with $t=M_\ve^{-1} |\log \ve|^{\alpha}$. This yields balls $B_1(t),\dots,B_{k(t)}(t)$ with radii $r_i(t)$ and degrees $d_i(t)$ such that
$$
\min\{r_i(t),r_0,1\}\leq E_\ve(u_\ve,B_i(t)\cap \Sigma)\frac t{\Lambda_\ve(t)}.
$$
From Lemma \ref{lem2Ball}, we have
$$
r_i(t)=\min\{r_i(t),r_0,1\}\leq M_\ve \frac{M_\ve^{-1}|\log \ve|^{\alpha}}
{C|\log \ve|} \leq C|\log \ve|^{\alpha-1}.
$$
In particular, by possibly further reducing the constant $\ve_0$, we deduce that $\overline{B}_i(t)\subset \Sigma$ for any $i\in \{1,\dots, k(t)\}$ and for any $\ve<\ve_0$. Hence $d=\sum_{i=1}^{k(t)} d_i(t)$. Then, from Proposition \ref{prop1Ball}, $r_i(t)\geq t |d_i(t)|$ and therefore
$$
E_\ve(u_\ve,\Sigma)\geq \sum_{i=1}^{k(t)} |d_i(t)|\Lambda_\ve(t).
$$
Since $\sum_{i=1}^{k(t)} |d_i(t)|\geq |d|$, Lemma \ref{lem2Ball} implies that, for any $\ve<\ve_0$,
$$
E_\ve(u,\Sigma)\geq \pi|d|\left(\log \frac t\ve-C\right)\geq\pi |d|\left(\log \frac 1\ve -\log M_\ve\right).
$$
The proposition is proved.
\end{proof}

\begin{corollary}\label{ball}
Let $\tilde \Sigma$ be a complete oriented surface in $\R^3$ whose second fundamental form is bounded by $Q_\ve=Q|\log\ve|^q$, where $q,Q>0$ are given numbers. Let $\Sigma$ be a bounded open subset of $\tilde \Sigma$. For any $m,M>0$ there exists $\ve_0(m,q,M,Q)>0$ such that, for any $\ve<\ve_0$, if $u_\ve\in H^1(\Sigma,\C)$ satisfies
$$
E_\ve(u_\ve,\Sigma)\leq M|\log \ve|^m
$$
and
$$
|u(x)|\geq \frac12\quad \textnormal{if } \d(x,\partial \Sigma)<Q_\ve^{-1}, 
$$
where $\d(\cdot,\cdot)$ denotes the distance function in $\tilde \Sigma$, then, letting $d$ be the winding number of $u_\ve/|u_\ve|:\partial\Sigma\to S^1$ and $M_\ve=M|\log \ve|^m$ we have 
$$
E_\ve(u_\ve,\Sigma)\geq \pi |d|\left(\log \frac1{\ve} -\log M_\ve Q_\ve \right).
$$
\end{corollary}
\begin{proof}
Let us consider the transformation
$$
\tilde u_\ve(y)=u_\ve \left(\frac{y}{Q_\ve} \right)\quad \textnormal{for }y\in \Sigma_\ve\colonequals Q_\ve \Sigma.
$$
We let $\tilde \Sigma_\ve\colonequals Q_\ve \tilde \Sigma$.
Observe that, by a change of variables, we have
$$
E_\ve(u_\ve,\Sigma)=E_{\tilde \ve}(\tilde u_\ve,\Sigma_\ve),
$$
where $\tilde \ve \colonequals \ve Q_\ve$. It is easy to check that the second fundamental form of $\tilde \Sigma_\ve$ is bounded by $1$. Then a direct application of Proposition \ref{ball2} shows that
$$
E_\ve(u_\ve,\Sigma)=E_{\tilde\ve}(\tilde u_\ve,\Sigma_\ve)\geq \pi |d|\left( \log \frac1{\tilde \ve} -\log M_\ve \right)=\pi |d|\left( \log \frac1\ve -\log M_\ve Q_\ve \right)
$$
for any $0<\ve<\ve_1=\ve_0Q_\ve^{-1}$, where $\ve_0$ is the constant appearing in the proposition.
\end{proof}

%%%%%%%%%%%%%%%%%%%%%%%%%%%%%%%%%%%%%%%%%%%%%%%%%%%%%%%%%%%%%%%%%%%%%%%%%%%%%%%%%%%%%%%%%%%%%%%%%

\section{A 2D vorticity estimate}\label{Sec:2DVortEstimate}
Let $\omega$ be a two-dimensional domain. For a given function $u:\omega\to \C$ and a given vector field $A:\omega\to \R^2$ we define
$$
j(u,A)=(iu,\nabla_A u),\quad \mu(u,A)=dj(u,A) + dA.
$$
We also let
$$
F_\ve(u,A,\omega)= \int_\omega e_\ve(u,A),\quad F_\ve(u,A,\partial \omega)=\int_{\partial \omega}e_\ve(u,A)d\H^1, 
$$
where 
$$
e_\ve(u,A)=|\nabla_A u|^2+\frac{1}{2\ve^2}(1-|u|^2)^2+|\curl A|^2.
$$
We have the following 2D vorticity estimate.
\begin{theorem}\label{Teo:2dEstimate} Let $\omega\subset \R^2$ be a bounded domain with Lipschitz boundary. Let $u:\omega \rightarrow \C$ 
and $A:\omega\rightarrow \R^2$ be $C^1(\overline \omega)$ and such that $|u|\geq 5/8$ on $\partial \omega$. 
Let $\{S_i\}_{i\in I}$ be the collection of connected component of $\{|u(x)|\leq 1/2\}$ whose degree $d_i=\mathrm{deg}(u/|u|,\partial S_i)\neq 0$.
Then, letting $r=\sum_{i\in I} \mathrm{diam}(S_i)$ and assuming $\ve,r\leq 1$, we have
\begin{equation}\label{2DVorticityEstimate}
 \left\| \mu(u,A) -2\pi \sum_{i\in I}d_i\delta_{a_i} \right\|_{C^{0,1}(\omega)^*}\leq C\max(\ve,r)(1+F_\ve(u,A,\omega)+F_\ve(u,A,\partial\omega)),
\end{equation}
where $a_i$ is the centroid of $S_i$ and $C$ is a universal constant.
\end{theorem}
\begin{proof}
As in \cite{SanSerBook}*{Chapter 6}, we set $\chi:\R_+\to \R_+$ to be defined by 
$$
\left\{\begin{array}{ll}\chi(x)=2x & \text{if} \ x\in \left[0,\frac12\right]\\
\chi(x)=1 & \text{if} \ x\in \left[\frac12,\frac32\right]\\
\chi(x)=1+2\left(x-\frac32\right)& \text{if} \ x\in \left[\frac32,2\right]\\
\chi(x)=x& \text{if} \ x \in[2,+\infty).
\end{array}\right.
$$ 
We then set $\tilde u:\omega\rightarrow \C$ by  
$$
\tilde u(x)= \dfrac{\chi(|u|)}{|u|} u
$$
and let
$$
\begin{array}{rlrl}
\tilde j&\colonequals\left(i\tilde u,d_{A} \tilde u\right), &  \tilde \mu&\colonequals d\tilde j+dA.
\end{array}
$$
Observe that $|\tilde u|=1$ and $\tilde \mu=0$ outside of $\cup_{i\in I} S_i$. We claim that
$$
\|\mu(u,A)-\tilde \mu\|_{C^{0,1}(\omega)^*}\leq C\ve(F_\ve(u,A,\omega)+F_\ve(u,A,\partial\omega)).
$$
In fact, by integration by parts, for any function $\zeta \in C^{0,1}(\omega)$, we have
$$
\left|\int_\omega \zeta (\mu(u,A)-\tilde \mu)  \right|\leq\left | \int_\omega \left(\nabla \zeta)^\bot \cdot (j(u,A)-\tilde j\right)\right|+\left |\int_{\partial \omega} \zeta \left(j(u,A)-\tilde j\right)\cdot \vartheta^\bot \right|,
$$
where $\vartheta$ is the outer unit normal to $\partial \omega$ and $x^\bot=(-x_2,x_1)$ for any vector $x=(x_1,x_2)$.
Arguing as in \cite{SanSerBook}*{Lemma 6.2}, we get
\begin{align*}
 \left | \int_\omega \left(\nabla \zeta)^\bot \cdot (j(u,A)-\tilde j\right) \right| &\leq \|\nabla \zeta\|_{L^\infty(\omega)}\int_\omega \frac{||u|^2-|\tilde u|^2|}{|u|}\left|\nabla_A u\right|\\
& \leq 3 \|\nabla \zeta\|_{L^\infty(\omega)}\int_\omega |1-|u|||\nabla_A u|\\
&\leq C\|\nabla \zeta\|_{L^\infty(\omega)}\ve F_\ve(u,A,\omega).
\end{align*}
Since $|\tilde u|=1$ on $\partial \omega$, a simple computation shows that
$$
\left|j(u,A)-\tilde j\right|\leq 2(1-\left|u\right|^2)\left|\nabla_A u\right|\quad \mbox{on }\partial \omega.
$$
By the Cauchy-Schwarz inequality, we find
\begin{align*}
\left | \int_{\partial\omega}\zeta \left(j(u,A)-\tilde j\right)\cdot \vartheta^\bot\right| &\leq 2\|\zeta\|_{C^{0,1}(\omega)}\int_{\partial\omega}(1-\left|u\right|^2)\left|\nabla_A u\right|d\H^1\\
&\leq C \|\zeta\|_{C^{0,1}(\omega)} \ve F_\ve(u,A,\partial\omega).
\end{align*}
Thus
$$
\|\mu(u,A)-\tilde \mu\|_{C^{0,1}(\omega)^*}\leq C\ve(F_\ve(u,A,\omega)+F_\ve(u,A,\partial\omega)),
$$
for some universal constant $C$.
The proof then reduces to proving that 
$$
\left\| \tilde \mu-2\pi \sum_{i\in I}d_i\delta_{a_i} \right\|_{C^{0,1}(\omega)^*}\leq C\max(r,\ve)(1+F_\ve(u,A,\omega)+F_\ve(u,A,\partial\omega)).
$$
Let $\zeta\in C^{0,1}(\omega)$ and observe that
$$
\int_\omega \zeta \tilde \mu =\sum_{i\in I}\int_{S_i} \zeta \tilde \mu= \sum_{i\in I}\zeta(a_i)\int_{S_i} \tilde \mu+ \sum_{i\in I}\int_{S_i} (\zeta -\zeta(a_i)) \tilde \mu.
$$
Since wherever $|\tilde u|=1$ we have $\tilde \mu= d(iu,du)$, Stokes' theorem yields
$$
\int_{S_i} \tilde \mu=\int_{\partial S_i}  (iu,\nabla u)\cdot\tau=2\pi d_i.
$$
Thus
$$
\sum_{i\in I}\zeta(a_i)\int_{S_i} \tilde \mu=2\pi \sum_{i\in I}d_i\zeta(a_i)=2\pi \sum_{i\in I}d_i\int_\omega \zeta\delta_{a_i}
$$
We also observe that, since $\zeta$ is a Lipschitz function, we have
$$
|\zeta(x)-\zeta(a_i)|\leq \|\zeta\|_{C^{0,1}(\omega)}|x-a_i|\leq \|\zeta\|_{C^{0,1}(\omega)} \mathrm{diam}(S_i)
$$
for all $x\in S_i$. 
 
On the other hand, noting that 
$$
\tilde \mu = 2(\partial_{x_1} \tilde u-iA_{x_1} \tilde u)\times (\partial_{x_2} \tilde u -iA_{x_2} \tilde u)+\curl A,
$$
we deduce that $|\tilde \mu|\leq 2|\nabla_A u|^2 + |\curl A|$. Then, letting $F_\ve(u,A,S_i)=\displaystyle \int_{S_i}e_\ve(u,A)$, the Cauchy-Schwarz inequality gives
$$
\int_{S_i}|\tilde \mu|\leq 4 \left(F_\ve(u,A,S_i)+ |S_i|^{\frac12}F_\ve(u,A,S_i)^{\frac12}\right).
$$
Observe that $|S_i|\leq C \mathrm{diam}(S_i)^2$. Collecting our previous computations, we find
$$
\left| \sum_{i\in I}\int_{S_i} (\zeta -\zeta(a_i)) \tilde \mu \right|\leq Cr\|\zeta\|_{C^{0,1}(\omega)} \left(F_\ve(u,A,\omega)+rF_\ve(u,A,\omega)^{\frac12}\right).
$$
Remembering that $\sqrt x\leq 1+x$, we get
$$
\left| \int_\omega \zeta\tilde \mu -2\pi\sum_{i\in I} d_i \int_\omega \zeta \delta_{a_i} \right| \leq C r\|\zeta\|_{C^{0,1}(\omega)}\left(1+F_\ve(u,A,\omega)\right).
$$
This concludes the proof of \eqref{2DVorticityEstimate}.
\end{proof}
Given a three-dimensional Lipschitz domain $\omega \subset \Omega$ contained in a plane, we let $(s,t,0)$ denote coordinates in $\R^3$ such that $\omega\subset \{(s,t,0) \in \Omega\}$. We define $\mu_\ve\colonequals \mu_\ve(u,A)[\partial_s,\partial_t]$, and write $\mu_{\ve,\omega}$ its restriction to $\omega$.
Theorem \ref{Teo:2dEstimate} immediately yields the following corollary.
\begin{corollary}\label{cor:2dVortEstimate}
Let $\gamma\in(0,1)$ and assume that $(u_\ve,A_\ve)\in H^1(\Omega,\C)\times H^1(\Omega,\R^3)$ is a configuration such that $F_\ve(u_\ve,A_\ve)\leq \ve^{-\gamma}$, so that by Lemma \ref{Lemma:Grid} there exists a grid $\GG(b_\ve,R_0,\de)$ satisfying \eqref{propGrid}. Then there exists $\ve_0(\gamma)$ such that, for any $\ve<\ve_0$ and for any face $\omega\subset \RR_2(\GG(b_\ve,R_0,\de))$ of a cube of the grid $\GG(b_\ve,R_0,\de)$,
letting $\{S_{i,\omega}\}_{i\in I_\omega}$ be the collection of connected components of $\{x\in\omega \ | \ |u_\ve(x)|\leq 1/2\}$ whose degree $d_{i,\omega}\colonequals \mathrm{deg}(u_\ve/|u_\ve|,\partial S_{i,\omega})\neq 0$, we have
\begin{multline*}
\left\| \mu_{\ve,\omega} -2\pi \sum_{i\in I_\omega}d_{i,\omega}\delta_{a_{i,\omega}} \right\|_{C^{0,1}(\omega)^*}\leq \\ C\max(r_\omega,\ve)\left(1+\int_\omega e_\ve(u_\ve,A_\ve)d\H^2+\int_{\partial \omega} e_\ve(u_\ve,A_\ve)d\H^1\right),
\end{multline*}
where $a_{i,\omega}$ is the centroid of $S_{i,\omega}$, $r_\omega\colonequals \sum_{i\in I_\omega}\mathrm{diam}(S_{i,\omega})$, and $C$ is a universal constant.
\end{corollary}
In view of the previous corollary, it is important to bound from above $r_\omega$, $d_{i,\omega}$, and $|I_\omega|$. Prior to doing so, let us recall the following result adapted from \cite{Jer}.
\begin{lemma}\label{lem:conncomponents} 
Under the hypotheses of Corollary \ref{cor:2dVortEstimate}, there exists $\ve_0(\gamma)$ such that, for any $\ve<\ve_0$ and for any face $\omega\subset \RR_2(\GG(b_\ve,R_0,\de))$ of a cube of the grid $\GG(b_\ve,R_0,\de)$, letting $\{S_{i,\omega}\}_{i\in I_\omega}$ be the collection of connected components of $\{x\in \omega \ | \ |u(x)|\leq 1/2\}$ whose degree $d_{i,\omega}\neq 0$, we have
$$ 
|d_{i,\omega}|\leq C\int_{S_{i,\omega}} |\nabla_{A_\ve} u_\ve|^2,
$$
where $C$ is a universal constant. 
\end{lemma}

With the aid of the previous lemma we prove the following result.
\begin{lemma}\label{Lemma:Covering} 
Under the hypotheses of Corollary \ref{cor:2dVortEstimate}, there exists $\ve_0(\gamma)$ such that, for any $\ve<\ve_0$ and for any face $\omega\subset \RR_2(\GG(b_\ve,R_0,\de))$ of a cube of the grid $\GG(b_\ve,R_0,\de)$, letting $\{S_{i,\omega}\}_{i\in I_\omega}$ be the collection of connected components of $\{x\in \omega \ | \ |u(x)|\leq 1/2\}$ whose degree $d_{i,\omega}\neq 0$, we have
\begin{align}
|I_\omega|\leq \sum_{i\in I_\omega}|d_{i,\omega}|\leq C\int_\omega e_\ve(u_\ve,A_\ve)d\H^2,\notag\\
r_\omega\leq C\ve \int_\omega e_\ve(u_\ve,A_\ve)d\H^2,\label{diam}
\end{align}
where $C$ is a universal constant. 
\end{lemma}
\begin{proof}
The first assertion immediately follows from Lemma \ref{lem:conncomponents}. To prove \eqref{diam} observe that, by the Cauchy-Schwarz inequality and the co-area formula, we have
\begin{align*}
\int_{\omega}e_\ve(u_\ve,A_\ve)d\H^2&\geq \int_\omega |\nabla |u_\ve||^2+\frac1{\ve^2}(1-|u_\ve|^2)^2d \H^2\\
	    &\geq\int_\omega \frac{|\nabla |u_\ve||(1-|u_\ve|^2)}{\ve}d \H^2\\
	    &=\int_{t=0}^\infty \frac{(1-t^2)}{\ve}\H^1(\{x\in \omega \ | \ |u_\ve(x)|=t\})d t.
\end{align*}
Thus the compact set $\{x\in \omega \ | \ |u(x)|\leq 1/2\}$ can be covered by a finite collection of disjoint balls of total radius smaller than $C\ve \int_{\omega}e_\ve(u_\ve,A_\ve)d\H^2$, which implies \eqref{diam}.
\end{proof}

\begin{remark}
By combining {\normalfont Lemma \ref{Lemma:Covering}} with \eqref{prop3Grid}, we obtain
\begin{align}
\sum_{\omega \subset \RR_2(\GG(b_\ve,R_0,\de))} |I_\omega|&\leq C\int\limits_{\RR_2(\GG(b_\ve,R_0,\de))} e_\ve(u_\ve,A_\ve)d\H^2\leq C\delta^{-1}F_\ve(u_\ve,A_\ve),\label{numberofconncomp}\\
\sum_{\omega \subset \RR_2(\GG(b_\ve,R_0,\de))}\sum_{i\in I_\omega} |d_{i,\omega}|&\leq C\int\limits_{\RR_2(\GG(b_\ve,R_0,\de))} e_\ve(u_\ve,A_\ve)d\H^2\leq C\delta^{-1}F_\ve(u_\ve,A_\ve),\label{numberofpoints}\\
r_{\GG}\colonequals\sum_{\omega \subset \RR_2(\GG(b_\ve,R_0,\de))} r_\omega&\leq C\ve \int\limits_{\RR_2(\GG(b_\ve,R_0,\de))} e_\ve(u_\ve,A_\ve)d\H^2 \leq C \ve \delta^{-1}F_\ve(u_\ve,A_\ve),\label{radiusGrid}
\end{align}
where $\sum_{\omega\subset \RR_2(\GG(b_\ve,R_0,\de))}$ denotes the sum over all the faces $\omega$ of cubes of the grid $\GG(b_\ve,R_0,\de)$.
\end{remark}

%%%%%%%%%%%%%%%%%%%%%%%%%%%%%%%%%%%%%%%%%%%%%%%%%%%%%%%%%%%%%%%%%%%%%%%%%%%%%%%%%%%%%%%%%%%%%%%%%

\section{3D vortex approximation construction}\label{Sec:3Dconstruction}
In this section we construct a new polyhedral approximation of the vorticity $\mu(u_\ve,A_\ve)$ of a configuration $(u_\ve,A_\ve)\in H^1(\Omega,\C)\times H^1(\Omega,\R^3)$ such that $F_\ve(u_\ve,A_\ve)\leq \ve^{-\gamma}$ for some $\gamma\in (0,1)$. The notion of minimal connection, first introduced in \cite{BreCorLie}, plays a key role in our construction. We begin this section by reviewing this concept. We then define the function $\zeta$ and the function $\zeta$ for $d_{\partial\Omega}$, and describe how to smoothly approximate these functions. Lastly, we provide our 3D vortex approximation construction.
\subsection{Minimal connections} 
Consider a collection $\AA=\{p_1,\dots,p_k,n_1,\dots,n_k\}$ of $2k$ points, where the $p_i$'s are the (non necessarily distinct) positive points and the $n_i$'s are the (non necessarily distinct) negative points. We define the length of a minimal connection joining the $p_i$'s to the $n_i$'s by
\begin{equation}\label{lengthMinCon}
L(\AA)\colonequals \min_{\sigma\in\S_k}\sum_{i=1}^k |p_i-n_{\sigma(i)}|,
\end{equation}
where $\S_k$ is the set of permutations of $k$ indices and hereafter $|\cdot|$ denotes the Euclidean distance in $\R^3$. We also define the $1$-current 
$\L(\AA)$, a minimal connection associated to $\AA$, as the sum in the sense of currents of the segments joining $p_i$ to $n_{\sigma(i)}$, where $\sigma\in \S_k$ is a permutation achieving the minimum in \eqref{lengthMinCon}. Although there can be several minimal connections associated to a collection $\AA$, we will make an arbitrary choice of one.

\medskip 
Let us now consider the distance
$$
d_{\partial\Omega}(x_1,x_2)\colonequals \min \{|x_1-x_2|,d(x_1,\partial\Omega)+d(x_2,\partial\Omega)\}\quad x_1,x_2\in \R^3.
$$
We define the length of a minimal connection joining the $p_i$'s to the $n_i$'s through $\partial\Omega$ by
\begin{equation}\label{lengthMinConBdry}
L_{\partial\Omega}(\AA)=\min_{\sigma\in\S_k}\sum_{i=1}^k d_{\partial\Omega}(p_i,n_{\sigma(i)}).
\end{equation}
In this case we define the $1$-current $\L_{\partial\Omega}(\AA)$, a minimal connection through $\partial \Omega$ associated to $\AA$, as the sum in the sense of currents of the segments joining $p_i$ to $n_{\sigma(i)}$ when $d_{\partial\Omega}(p_i,n_{\sigma(i)})=|p_i-n_{\sigma(i)}|$ and the (properly oriented) segments joining $p_i,n_{\sigma(i)}$ to $\partial\Omega$ when $d_{\partial\Omega}(p_i,n_{\sigma(i)})=d(p_i,\partial \Omega)+d(n_{\sigma(i)},\partial\Omega)$, where $\sigma\in \S_k$ is a permutation achieving the minimum in \eqref{lengthMinConBdry}. Once again, if the minimal connection is not unique we make an arbitrary choice of one.

\subsubsection{The function $\zeta$} The following lemma is a particular case of a well-known result proved in \cite{BreCorLie}.

\begin{lemma}\label{LemmaBCL} Let $\AA=\{p_1,\dots,p_k,n_1,\dots,n_k\}$ be a configuration of positive and negative points. Assume, relabeling the points if necessary, that $L(\AA)=\sum_{i=1}^k|p_i-n_i|$. Then there exists a $1$-Lipschitz function $\displaystyle\zeta^*: \cup_{i=1,\dots,k}\{p_i,n_i\} \to \R$ such that
$$
L(\AA)=\sum_{i=1}^{k}\zeta^*(p_i)-\zeta^*(n_i) \quad \mathrm{and} \quad \zeta^*(n_i)=\zeta^*(p_i)-|p_i-n_i|.
$$
\end{lemma}
\begin{definition}[The function $\zeta$]\label{def:funczeta} Let $\AA=\{p_1,\dots,p_k,n_1,\dots,n_k\}$ be a configuration of positive and negative points. Denote by $\zeta^*$ the $1$-Lipschitz function given by Lemma \ref{LemmaBCL}. We define the function $\zeta:\R^3\to \R$ via the formula
$$
\zeta(x)\colonequals \max_{i\in \{1,\dots,k\}} \left(\zeta^*(p_i)-\max_{j\in\{1,\dots,2k\}}d_{(i,j)}(x)\right),
$$
with
$$
d_{(i,j)}(x)\colonequals \langle p_i-x,\nu_{(i,j)}\rangle,\quad \nu_{(i,j)} \colonequals \left\{\begin{array}{cl}\frac{p_i-a_j}{|p_i-a_j|}&\mathrm{if}\ p_i\neq a_j\\0&\mathrm{if}\ p_i=a_j\end{array}\right. ,
$$
where here and in the rest of the paper the points $a_i$ are defined as follows: if $j\in \{1,\dots,k\}$ then $a_j=p_j$, if $j\in \{k+1,\dots,2k\}$ then $a_j=n_{j-k}$.
\end{definition}
\begin{lemma}\label{lem:funczeta} Let $\AA=\{p_1,\dots,p_k,n_1,\dots,n_k\}$ be a configuration of positive and negative points. Denote by $\displaystyle\zeta^*: \cup_{i=1,\dots,k}\{p_i,n_i\} \to \R$ the function given by Lemma \ref{LemmaBCL} and define $\zeta:\R^3\to \R$ as in Definition \ref{def:funczeta}. Then $\zeta$ is a $1$-Lipschitz extension of $\zeta^*$ to $\R^3$.
\end{lemma}
\begin{proof}
It is easy to see that $\zeta$ is a $1$-Lipschitz function. Let us check that
$$
\zeta(p_i)=\zeta^*(p_i)\quad \mathrm{and}\quad \zeta(n_i)=\zeta^*(n_i)
$$
for every $i\in \{1,\dots,k\}$. Observe that
$$
|d_{(i,j)}(x)|=|\langle p_i-x,\nu_{(i,j)}\rangle|\leq |p_i-x|.
$$
But 
$$
d_{(i,j)}(a_l)=|p_i-a_l|\quad \mathrm{for\ any} \ a_l\in \AA.
$$
Thus
$$
\zeta(a_l)=\max_{i\in \{1,\dots,k\}}(\zeta^*(p_i)-|p_i-a_l|).
$$
Since $\zeta^*$ is $1$-Lipschitz, we deduce that $\zeta(a_l)\leq \zeta^*(a_l)$. It follows that $\zeta(p_l)=\zeta^*(p_l)$ for every $l\in \{1,\dots,k\}$.
We conclude the proof by noting that, for any $l\in \{1,\dots,k\}$,
$$
\zeta(n_l)\geq\zeta^*(p_l)-|p_l-n_l|=\zeta^*(n_l).
$$
\end{proof}
Let us remark that this extension is not the same that appears in \cite{BreCorLie}. As pointed out in the introduction (see Section \ref{strategy}), our strategy of proof of the main results combines the use of the co-area formula and the ball construction method applied on the level sets of the function $\zeta$. In Section \ref{Sec:Ball}, we saw that in order to apply the ball construction on a surface we need to control its second fundamental form. But since $\zeta$ is only Lipschitz, we have no control on the second fundamental form of its level sets. For this reason we need to smoothly approximate this function and, moreover, to provide a quantitative estimate of the second fundamental form of the approximation. We have the following technical result, whose proof is postponed to Appendix \ref{Sec:AppendixA}. 

\begin{proposition}[Quantitative smooth approximation of the function $\zeta$]\label{prop:functionzeta} 
Let $\AA=\{p_1,\dots,p_k,n_1,\dots,n_k\}$ be a configuration of positive and negative points. Assume, relabeling the points if necessary, that $L(\AA)=\sum_{i=1}^k|p_i-n_i|$. Define $D_\AA\colonequals\max_{a_i,a_j\in \AA}|a_i-a_j|$ to be the maximum Euclidean distance between any of the points of $\AA$. Then there exist $C,C_0,C_1>0$ such that, for any $\rho\in (0,1/2)$ and for any $0<\la<\la_0(\rho)\colonequals (C_0(2k)^{-6})^{1/\rho}$, there exists a smooth function $\zeta_\la: \R^3 \to \R$ satisfying:
\begin{enumerate}[font=\normalfont,leftmargin=*]
\item $|L(\AA)-\sum_{i=1}^k \zeta_\la(p_i)-\zeta_\la(n_i) |\leq CD_\AA(2k)^6\la^\rho$.
\item $\|\nabla \zeta_\la\|_{L^\infty(\R^3)}\leq 1$.
\item There exists a set $P_\la\subset \R^3$ such that $|\zeta_\la(P_\la)|\leq 2\la k^2$ and that, for any $0<\kappa<\la^{2\rho}/3$, 
$$
C_\kappa\colonequals \{x\ |\ |\nabla \zeta_\la(x)|<\kappa\}\setminus P_\la 
$$
can be covered by $\B_\kappa$, a collection of at most $(2k)^8$ balls of radius $C\la/(\la^{2\rho}-3\kappa)$. Moreover, defining
\begin{equation}\label{setoft}
T_{\kappa}\colonequals \zeta_\la \left(\cup_{B\in \B_\kappa}B\right),
\end{equation}
we have that, for any $t\in \R \setminus (T_\kappa\cup \zeta_\la(P_\la))$, $\{x \ | \ \zeta_\la (x) =t\}$ is a complete submanifold of $\R^3$ whose second fundamental form is bounded by $C_1(\la^2 \kappa)^{-1}$.
\end{enumerate}
\end{proposition}

\subsubsection{The function $\zeta$ for $d_{\partial\Omega}$} When the Euclidean distance is replaced with the distance through $\partial \Omega$ the following lemma can be proved (see \cite{BreCorLie}).

\begin{lemma}\label{LemmaBCLBoundary} Let $\AA=\{p_1,\dots,p_k,n_1,\dots,n_k\}\subset \Omega$ be a configuration of positive and negative points. Assume, relabeling the points if necessary, that $L_{\partial \Omega}(\AA)=\sum_{i=1}^kd_{\partial\Omega}(p_i,n_i)$. Then there exists a function  $\displaystyle\zeta^*: \cup_{i=1,\dots,k}\{p_i,n_i\} \to \R$, $1$-Lipschitz for the distance $d_{\partial\Omega}$, such that
$$
L_{\partial\Omega}(\AA)=\sum_{i=1}^{k}\zeta^*(p_i)-\zeta^*(n_i) \quad \mathrm{and} \quad \zeta^*(n_i)=\zeta^*(p_i)-d_{\partial\Omega}(p_i,n_i).
$$
\end{lemma}
\begin{definition}[The function $\zeta$ for $d_{\partial\Omega}$]
\label{def:funczetaBoundary} Let $\AA=\{p_1,\dots,p_k,n_1,\dots,n_k\}$ be a configuration of positive and negative points. Denote by $\zeta^*$ the function given by Lemma \ref{LemmaBCL}. We define the function $\zeta:\R^3\to \R$ for $d_{\partial\Omega}$ via the formula
$$
\zeta(x)\colonequals \max_{i\in \{1,\dots,k\}} \left(\zeta^*(p_i)-d_i(x,\partial\Omega)\right),
$$
where
$$
d_i(x,\partial\Omega)\colonequals 
\min\left[
\max \left( \max_{j\in\{1,\dots,2k\}}d_{(i,j)}(x),d(p_i,\partial\Omega)-d(x,\partial\Omega)\right),
d(p_i,\partial\Omega)+d(x,\partial\Omega)
\right],
$$
with
$$
d_{(i,j)}(x)= \langle p_i-x,\nu_{(i,j)}\rangle,\quad \nu_{(i,j)}=\left\{\begin{array}{cl}\frac{p_i-a_j}{|p_i-a_j|}&\mathrm{if}\ p_i\neq a_j\\0&\mathrm{if}\ p_i=a_j\end{array}\right. .
$$
\end{definition}
\begin{lemma}\label{lem:funczetaBoundary} Let $\AA=\{p_1,\dots,p_k,n_1,\dots,n_k\}$ be a configuration of positive and negative points. Denote by $\displaystyle\zeta^*: \cup_{i=1,\dots,k}\{p_i,n_i\} \to \R$ the function given by Lemma \ref{LemmaBCLBoundary} and 
define $\zeta:\R^3\to \R$ as in Definition \ref{def:funczetaBoundary}. Then $\zeta$ is a $1$-Lipschitz extension of $\zeta^*$ to $\R^3$, which is constant on $\partial\Omega$.
\end{lemma}
\begin{proof}
It is easy to see that $\zeta$ is a $1$-Lipschitz function. Let us check that
$$
\zeta(p_i)=\zeta^*(p_i)\quad \mathrm{and}\quad \zeta(n_i)=\zeta^*(n_i)
$$
for every $ i\in \{1,\dots,k\}$.
By the proof of Lemma \ref{lem:funczeta}, we know that
$$
d_{(i,j)}(a_l)=|p_i-a_l|\quad \mathrm{for\ any} \ a_l\in \AA.
$$
By the triangular inequality, we deduce that
$$
\max \left( |p_i-a_l|,d(p_i,\partial\Omega)-d(a_l,\partial\Omega)\right)=|p_i-a_l|.
$$
Then
$$
d_i(a_l,\partial\Omega)=
\min\left(|p_i-a_l|,d(p_i,\partial\Omega)-d(a_l,\partial\Omega)
\right)=d_{\partial\Omega}(p_i,a_l),
$$
which implies that
$$
\zeta(a_l)=\max_{i\in \{1,\dots,k\}}(\zeta^*(p_i)-d_{\partial\Omega}(p_i,a_l)).
$$
Since $\zeta^*$ is $1$-Lipschitz for the distance $d_{\partial\Omega}$, we have that $\zeta(a_l)\leq \zeta^*(a_l)$. It follows that $\zeta(p_l)=\zeta^*(p_l)$ for every $l\in \{1,\dots,k\}$. But 
$$
\zeta(n_l)\geq\zeta^*(p_l)-d_{\partial\Omega}(p_l,n_l)=\zeta^*(n_l).
$$
Finally, observe that, for all $x\in \partial\Omega$,
$$
d_i(x,\partial\Omega)\colonequals 
\min\left[
\max \left( \max_{j\in\{1,\dots,2k\}}d_{(i,j)}(x),d(p_i,\partial\Omega)\right),
d(p_i,\partial\Omega)
\right]=d(p_i,\partial\Omega).
$$
Thus 
$$
\zeta(x)=\max_{i\in \{1,\dots,k\}}(\zeta^*(p_i)-d(p_i,\partial\Omega))
$$
for all $x\in \partial\Omega$.
\end{proof}
We remark that this extension is not the same that appears in \cite{BreCorLie}.
In order to provide a lower bound for the free energy close to the boundary, we need to smoothly approximate the function $\zeta$ for $d_{\partial\Omega}$ and provide a quantitative estimate of the second fundamental form of the approximation. 
In this paper we describe two methods of doing this, which may be of independent interest. The first method is based on an analysis of the curvature of the boundary of the domain, which requires it to be of class $C^2$. The second method is based on a polyhedral approximation of $\partial\Omega$, which in addition requires it to have strictly positive Gauss curvature. 

\begin{proposition}\label{prop:functionzetaBoundary}[Quantitative smooth approximation of the function $\zeta$ for $d_{\partial\Omega}$ -- First method]
Assume that $\partial\Omega$ is of class $C^2$. Let $\AA=\{p_1,\dots,p_k,n_1,\dots,n_k\}\subset \Omega$ be a configuration of positive and negative points. Assume, relabeling the points if necessary, that $L_{\partial \Omega}(\AA)=\sum_{i=1}^kd_{\partial\Omega}(p_i,n_i)$. Then there exist constants $\theta_0,C,C_0,C_1$ that depend only on $\partial\Omega$ such that, for any $\rho\in (0,1/4)$ and for any $0<\la<\la_0(\rho)\colonequals (C_0(2k)^{-6})^{1/\rho}$, there exists a smooth function $\zeta_\la: \R^3 \to \R$ satisfying:
\begin{enumerate}[font=\normalfont,leftmargin=*]
\item $|L_{\partial\Omega}(\AA)-\sum_{i=1}^k \zeta_\la(p_i)-\zeta_\la(n_i) |\leq C(2k)^6\la^\rho$.
\item Letting
\begin{equation}\label{Omegalambda}
\Omega_\la\colonequals \{x\in\Omega \ | \ 2\la^{1/\rho}<\dist(x,\partial\Omega)<\theta_0- 2\la^{1/\rho}\},
\end{equation}
we have $\|\nabla \zeta_\la\|_{L^\infty(\Omega_\la)}\leq 1$.
\item $|\zeta_\la(\{x\in\Omega \ | \ \dist(x,\partial\Omega)\leq 2\la^{1/\rho}\})|\leq C\la^{1/\rho}$.
\item There exists a set $P_\la\subset \R^3$ such that $|\zeta_\la(P_\la)|\leq C(2k)^4\la^{3\rho/4}$ and that, for any $0<\kappa<\la^{2\rho}/3$,
$$
C_\kappa\colonequals \{x\in \Omega_\la \ |\ |\nabla \zeta_\la(x)|<\kappa\}\setminus P_\la
$$
can be covered by $\B_\kappa$, a collection of at most $C((2k)^8+\theta_0(2k)^6(\la^{2\rho}-3\kappa)^{-3})$ balls of radius $C\la/(\la^{2\rho}-3\kappa)$. Moreover, defining 
\begin{equation}\label{setoftBoundary}
T_{\kappa}\colonequals \zeta_\la \left(\cup_{B\in \B_\kappa}B\right),
\end{equation}
we have that, for any $t\in \zeta_\la(\Omega_\la) \setminus (T_\kappa\cup \zeta_\la(P_\la))$, $\{x  \ | \ \zeta_\la (x) =t\}$ is a complete submanifold of $\R^3$ whose second fundamental form is bounded by $C_1(\la^2 \kappa)^{-1}$.
\end{enumerate}
\end{proposition}
The proof of this proposition is deferred to Appendix \ref{Sec:AppendixB}.

\begin{proposition}\label{prop:functionzetaBoundary2}[Quantitative smooth approximation of the function $\zeta$ for $d_{\partial\Omega}$ -- Second method]
Assume that $\partial\Omega$ is of class $C^2$ and has strictly positive Gauss curvature. Let $\AA=\{p_1,\dots,p_k,n_1,\dots,n_k\}\subset \Omega$ be a configuration of positive and negative points. Assume, relabeling the points if necessary, that $L_{\partial \Omega}(\AA)=\sum_{i=1}^kd_{\partial\Omega}(p_i,n_i)$. Then there exist constants $\tau_0,C,C_0,C_1$, that depend only on $\partial\Omega$, and a universal constant $C_2>0$ such that, for any $\tau<\tau_0$, for any $\rho\in (0,1/2)$, and for any 
$$
0<\la<\la_0(\rho,\tau)\colonequals \left(C_0\min\left\{(2k)^{-6},(2k)^{-4}\tau^2,(2k)^{-2}\tau^4,\tau^5\right\}\right)^{1/\rho},
$$
there exists a smooth function $\zeta_\la: \R^3 \to \R$ satisfying:
\begin{enumerate}[font=\normalfont,leftmargin=*]
\item $|L_{\partial\Omega}(\AA)-\sum_{i=1}^k \zeta_\la(p_i)-\zeta_\la(n_i) |\leq C (((2k)^6+(2k)^4\tau^{-2}+(2k)^2\tau^{-4})\la^\rho+2k\tau^2)$.
\item Letting
$$
\Omega_\la\colonequals \{x\in\Omega \ | \ \dist(x,\partial\Omega)>2\la\},
$$ 
we have $\|\nabla \zeta_\la\|_{L^\infty(\Omega_\la)}\leq 1$.
\item $|\zeta_\la(\overline{\Omega\setminus \Omega_\la
})|\leq C(\tau^2+\la)$.
\item There exists a set $P_\la\subset \R^3$ such that $|\zeta_\la(P_\la)|\leq 2\la k^2$ and that, for any $0<\kappa<\la^{2\rho}/3$,
$$
C_\kappa\colonequals \{x\in \Omega_\la \ |\ |\nabla \zeta_\la(x)|<\kappa\}\setminus P_\la
$$
can be covered by $\B_\kappa$, a collection of at most $C((2k)^8+\tau^{-8})$ balls of radius $C_2\la/(\la^{2\rho}-3\kappa)$. Moreover, defining 
\begin{equation*}%\label{setoftBoundary}
T_{\kappa}\colonequals \zeta_\la \left(\cup_{B\in \B_\kappa}B\right),
\end{equation*}
we have that, for any $t\in \zeta_\la(\Omega_\la) \setminus (T_\kappa\cup \zeta_\la(P_\la))$, $\{x  \ | \ \zeta_\la (x) =t\}$ is a complete submanifold of $\R^3$ whose second fundamental form is bounded by $C_1(\la^2 \kappa)^{-1}$.
\end{enumerate}
\end{proposition}
We point out that in this proposition the parameter $\tau$ is associated to the polyhedral approximation of $\partial\Omega$. The proof of this proposition is deferred to Appendix \ref{Sec:AppendixC}.

\subsection{Construction of the vorticity approximation\label{subsec:vorticity}}
Let $\gamma\in (0,1)$ and consider a configuration $(u_\ve,A_\ve)\in H^1(\Omega,\C)\times H^1(\Omega,\R^3)$ such that $F_\ve(u_\ve,A_\ve)\leq \ve^{-\gamma}$. Then Lemma \ref{Lemma:Grid} provides a grid $\GG(b_\ve,R_0,\de)$ satisfying \eqref{propGrid}. We begin by constructing our approximation in the cubes of the grid. For each cube $\CC_l\in \GG(b_\ve,R_0,\de)$, Corollary \ref{cor:2dVortEstimate} gives the existence of points $a_{i,\omega}$ and integers $d_{i,\omega}\neq 0$ such that
$$
\mu_{\ve,\omega}\approx 2\pi \sum_{i\in I_\omega} d_{i,\omega}\delta_{a_{i,\omega}},
$$
for each of the six faces $\omega\subset \partial \CC_l$ of the cube $\CC_l$. Observe that, since $\partial\mu(u_\ve,A_\ve)=0$ relative to $\CC_l$, we have
$$
\sum_{\omega\subset \partial \CC_l} \sum_{i\in I_\omega} d_{i,\omega}=0.
$$
Then, we define a configuration $\AA_l\colonequals \{p_1,\dots,p_{k_l},n_1,\dots,n_{k_l}\}$ of positive and negative points associated to $\partial\CC_l$, by repeating the points $a_{i,\omega}$ according to their degree $d_{i,\omega}$, for each of the six faces $\omega$ of the cube $\CC_l$. The previous observation implies that the number of positive points $p_i$'s and negative points $n_i$'s of the collection $\AA_l$ are equal. We note that 
$$
2k_l=\sum_{\omega\subset \partial \CC_l} \sum_{i\in I_\omega} |d_{i,\omega}|.
$$
Consider the minimal connection $\L(\AA_l)$ associated to $\AA_l$. It may happen that the segment connecting some $p_i$ to $n_{\sigma(i)}$ in $\L(\AA_l)$ belongs to one of the faces $\omega$ of the cube $\CC_l$. In this case we define a new connection $\tilde \L(\AA_l)$ by replacing the original segment connecting $p_i$ to $n_{\sigma(i)}$ with a Lipschitz curve connecting $p_i$ to $n_{\sigma(i)}$ from the inside (preserving the orientation), so that its intersection with $\partial\CC_l$ is given by $\{p_i,n_{\sigma(i)}\}$. This process can be performed in such a way that $|L(\AA_l)-|\tilde \L(\AA_l)||$ is less than an arbitrarily small number. We remark that the resulting connection $\tilde \L(\AA_l)$ is a polyhedral $1$-current whose intersection with $\partial \CC_l$ is equal to $\cup_{i=1,\dots,k_l}\{p_i,n_i\}$. We define
$$
\nu_{\ve,\CC_l}\colonequals 2\pi \tilde\L(\AA_l)\quad \mathrm{in}\ \CC_l, 
$$
for every cube $\CC_l\in\GG(b_\ve,R_0,\de)$.

\medskip
We now construct our vorticity approximation in $\overline \Theta$ (recall \eqref{unioncubes}). Once again Corollary \ref{cor:2dVortEstimate} gives the existence of points $a_{i,\omega}$ and integers $d_{i,\omega}\neq 0$ such that
$$
\mu_{\ve,\omega}\approx 2\pi \sum_{i\in I_\omega} d_{i,\omega}\delta_{a_{i,\omega}},
$$
for each face $\omega\subset \RR_2(\GG(b_\ve,R_0,\de))$ of a cube of the grid such that $\omega \subset \partial \GG$. Then, we define a configuration $\AA_{\partial\GG}\colonequals \{p_1,\dots,p_{k_{\partial\GG}},n_1,\dots,n_{k_{\partial\GG}}\}$ of positive and negative points associated to $\partial\GG$ by repeating the points $a_{i,\omega}$ according to their degree $d_{i,\omega}$, for each face $\omega\in \RR_2(\GG(b_\ve,R_0,\de))$ of a cube of the grid such that $\omega \subset \partial \GG$. Observe that, since $\partial \mu (u_\ve,A_\ve)=0$ relative to $\partial \GG$, we have
$$
\sum_{\omega\subset \partial \GG} \sum_{i\in I_\omega} d_{i,\omega}=0,
$$
which ensures that the number of positive points $p_i's$ and negative points $n_i's$ of the collection $\AA_{\partial\Omega}$ are equal. We note that 
$$
2k_{\partial \GG}=\sum_{\omega\subset \partial \GG} \sum_{i\in I_\omega} |d_{i,\omega}|.
$$

One might want to define the vorticity approximation close to the boundary as the minimal connection $\L_{\partial\Omega}(\AA_{\partial \GG})$ through $\partial \Omega$ associated to $\AA_{\partial \GG}$. Unfortunately we cannot do this, because it is not possible to rule out the possibility of having
$$
\L_{\partial\Omega}(\AA_{\partial \GG})\cap (\overline \Omega\setminus \overline \Theta)\neq \emptyset.
$$
For this reason, we consider the distance
$$
\d_{\partial\Omega}(x_1,x_2)\colonequals \min \{\d(x_1,x_2),d(x_1,\partial\Omega)+d(x_2,\partial\Omega)\}\quad x_1,x_2\in \partial \GG,
$$
where $\d$ denotes the geodesic distance on $\partial\GG$.

Let us define 
\begin{equation}\label{lengthMinConBdry2}
L_{\d_{\partial\Omega}}(\AA_{\partial\GG})=\min_{\sigma\in\S_{ k_{\partial\GG}}}\sum_{i=1}^{k_{\partial\GG}} \d_{\partial\Omega}(p_i,n_{\sigma(i)}).
\end{equation}
and the $1$-current $\L_{\d_{\partial\Omega}}(\AA_\GG)$, a minimal connection through $\partial \Omega$ associated to $\AA_{\partial\GG}$ contained in $\overline \Theta$, as the sum in the sense of currents of the geodesics (on $\partial\GG$) joining $p_i$ to $n_{\sigma(i)}$ when $\d_{\partial\Omega}(p_i,n_{\sigma(i)})=\d(p_i,n_{\sigma(i)})$ and the (properly oriented) segments joining $p_i,n_{\sigma(i)}$ to $\partial\Omega$ when $d_{\partial\Omega}(p_i,n_{\sigma(i)})=d(p_i,\partial \Omega)+d(n_{\sigma(i)},\partial\Omega)$, where $\sigma\in \S_{k_\GG}$ is a permutation achieving the minimum in \eqref{lengthMinConBdry2}. If the minimal connection is not unique we make an arbitrary choice of one. 

Performing a replacement argument (from the inside) in $\overline \Theta$, analogous to the one described above, we define a new connection $\tilde \L_{\d_{\partial\Omega}}(\AA_{\partial \GG})$, with $|L_{\d_{\partial\Omega}}(\AA_\GG)-|\tilde \L_{\d_{\partial\Omega}}(\AA_\GG)|$ less than an arbitrarily small number, whose intersection with $\partial \GG$ is equal to $\cup_{i=1,\dots,k_{\partial\GG}}\{p_i,n_i\}$ and which is contained in $\overline \Theta$. We set
$$
\nu_{\ve,\Theta}\colonequals 2\pi \tilde \L_{\d_{\partial\Omega}}(\AA_{\partial \GG})\quad \mathrm{in}\ \overline \Theta.
$$
Finally, we define our polyhedral approximation $\nu_\ve$ of the vorticity $\mu(u_\ve,A_\ve)$ by
\begin{equation}\label{def:vortapprox}
\nu_\ve\colonequals \sum_{\CC_l\in\GG(b_\ve,R_0,\de)} \nu_{\ve,\CC_l}+ \nu_{\ve,\Theta},
\end{equation}
where the sums are understood in the sense of currents. 

We observe that the topological degree depends on the orientation of the domain in which it is computed. If a face $\omega\subset \RR_2(\GG(b_\ve,R_0,\de))$ belongs to two cubes $C_1$ and $C_2$ of the grid, then its associated collection of degrees $d_{i,\omega}$'s for $C_1$ is equal to minus its associated collection of degrees for $C_2$.
Of course the same occurs for those faces $\omega$ belonging to one of the cubes of the grid and to $\partial \GG$.

On the other hand \eqref{prop1Grid} implies that, for any face $\omega\subset \RR_2(\GG(b_\ve,R_0,\de))$, the intersecction between the collection of points $a_{i,\omega}$'s and $\RR_1(\GG(b_\ve,R_0,\de))$ is empty.

By combining these arguments we conclude that the $1$-currents $\nu_{\CC_l}$'s and $\nu_{\Theta}$ have a good compatibility condition between each other. Hence, by construction, $\nu_\ve$ is a polyhedral $1$-current such that $\partial \nu_\ve=0$ relative to $\Omega$. In addition it approximates well $\mu(u_\ve,A_\ve)$ in an appropiate norm, as we shall show in Section \ref{Sec:proofMain}.

\subsubsection{An important remark towards the proof of the lower bound close to the boundary}\label{importantlb}
Let us study the $1$-current $\L_{\d_{\partial\Omega}}(\AA_\GG)$ defined above. We are interested in the situation $\d_{\partial\Omega}(p_i,n_{\sigma(i)})=\d(p_i,n_{\sigma(i)})$, where $\sigma$ denotes a (fixed) permutation achieving the minimum in \eqref{lengthMinConBdry2}. Since $d(x,\partial\Omega)\leq 2\de$ for any $x\in \partial\GG$, we have two possibilities:
\begin{itemize}
\item $\d(p_i,n_{\sigma(i)})=|p_i-n_{\sigma(i)}|$ and there exists a face $\omega\subset \partial\GG$ which contains both points.

\item $\d(p_i,n_{\sigma(i)})\neq|p_i-n_{\sigma(i)}|$ and there exist (different) faces $w_1,\dots,w_{J(i)}\subset \partial\GG$, with $2 \geq J(i)<C^*$ for some constant $C^*>0$ that depends only on the boundary, such that 
$p_i\in \omega_1$, $n_{\sigma(i)}\in \omega_{J(i)}$, and
$$
\d(p_i,n_{\sigma(i)})=|p_i-q_1|+|q_1-q_2|+\dots +|q_{J(i)-2}-q_{J(i)-1}|+|q_{J(i)-1}-n_{\sigma(i)}|,
$$
where the points $q_j$ are such that $q_j\in \partial \omega_j\cap \partial\omega_{j+1}$, $j=1,\dots,J(i)-1$.
\end{itemize}
If the second situation occurs for some $i$, we will enlarge the collection of points $\AA_{\partial \GG}$. We proceed as follows: for any $i$ for which the second situation happens, we add, for $j=1,\dots,J(i)-1$, the point $q_j$ to the collection twice: both as a positive and negative point (i.e. with degree $+1$ and $-1$). This yields a new collection  
\begin{equation}\label{newcollection}
\tilde \AA_{\partial\GG }=\{p_1,\dots,p_{\tilde k_{\partial\GG}},n_1,\dots,n_{\tilde k_{\partial\GG}}\}
\end{equation}
of positive and negative points, which contains $\AA_{\partial\GG}$. In particular, we have that 
\begin{equation}\label{numbertilde}
\tilde k_{\partial\GG}\leq C^* k_{\partial \GG}\leq C^*\sum_{\omega\subset \partial \GG} \sum_{i\in I_\omega} |d_{i,\omega}|.
\end{equation}
Moreover, 
\begin{equation}\label{newmin}
L_{\d_{\partial\Omega}}(\tilde \AA_{\partial\GG})\colonequals \min_{\sigma\in\S_{\tilde k_{\partial\GG}}}\sum_{i=1}^{\tilde k_{\partial\GG}} \d_{\partial\Omega}(p_i,n_{\sigma(i)})=L_{\d_{\partial\Omega}}(\AA_{\partial\GG}).
\end{equation}
The commodity of using this new collection is that there exists a permutation $\sigma^* \in \S_{\tilde k_{\partial\GG}}$ achieving the minimum in \eqref{newmin}, which is naturally derived from the previous construction, such that if $\d_{\partial\Omega}(p_i,n_{\sigma^*(i)})=\d(p_i,n_{\sigma^*(i)})$ then $\d(p_i,n_{\sigma^*(i)})=|p_i-n_{\sigma^*(i)}|$ and there exists a face $\omega\subset\partial \GG$ which contains both points. This in particular implies that 
$$
\d_{\partial\Omega}(p_i,n_{\sigma^*(i)})=\min \{\d(p_i,n_{\sigma^*(i)}),d(p_i,\partial\Omega)+d(n_{\sigma^*(i)},\partial\Omega)\}=d_{\partial\Omega}(p_i,n_{\sigma^*(i)}).
$$
Finally, by \cite{San}*{Lemma 2,2}, which is a slight modification of the previously cited well-known result in \cite{BreCorLie}, there exists a $1$-Lipschitz function $\zeta^*:\cup_{i=1,\dots,\tilde k_{\partial\GG}}\{p_i,n_{\sigma^*(i)}\}\to \R$ such that 
\begin{equation}\label{lengthtilde}
L_{\d_{\partial\Omega}}(\tilde \AA_{\partial\GG})=\sum_{i=1}^{\tilde k_{\partial\GG}}\zeta^*(p_i)-\zeta^*(n_{\sigma^*(i)})\quad \mbox{and}\quad \zeta^*(p_i)-\zeta^*(n_{\sigma^*(i)})=\d_{\partial\Omega}(p_i,n_{\sigma^*(i)}).
\end{equation}
Combining this with our previous observation, we get
\begin{equation}\label{cond*}
\zeta^*(p_i)-\zeta^*(n_{\sigma^*(i)})=d_{\partial\Omega}(p_i,n_{\sigma^*(i)}).
\end{equation}
In particular, we can extend this function by defining the function $\zeta$ for $d_{\partial\Omega}$ as in Definition \ref{def:funczetaBoundary}, and therefore in the proof of the lower bound close to the boundary (see Section \ref{Sec:LBboundary}) will be enough to consider a quantitative smooth approximation of this extension.
It is worth to remark that in this case we cannot ensure that $L_{\d_{\partial\Omega}}(\tilde \AA_{\partial\GG})=L_{\partial\Omega}(\tilde \AA_{\partial\GG})$, but since \eqref{cond*} holds, arguing almost readily as in the proofs of Propositions \ref{prop:functionzetaBoundary} and \ref{prop:functionzetaBoundary2}, one can show that there exists a smooth function $\zeta_\la$ associated to $\tilde \AA_{\partial\GG}$ satisfying these propositions with the quantity $L_{\partial\Omega}(\tilde \AA_{\partial\GG})$ being replaced by $\sum_{i=1}^{\tilde k_{\partial\GG}}\zeta^*(p_i)-\zeta^*(n_{\sigma^*(i)})$.
In particular, by combining this with \eqref{lengthtilde} and \eqref{cond*}, we conclude that
$$
\left| L_{\d_{\partial\Omega}}(\tilde \AA_{\partial\GG})-\sum_{i=1}^{\tilde k_{\partial\GG}}\zeta_\la(p_i)-\zeta_\la(n_{\sigma^*(i)})\right |\leq \mbox{small (quantitative) error}.
$$
The function $\zeta_\la$ that we use in Section \ref{Sec:LBboundary} is precisely the function described here.

\subsubsection{The support of $\nu_\ve$} To end this section we present a result about the support of $\nu_\ve$.
\begin{lemma}\label{Lemma:support}
Let $\gamma\in(0,1)$ and assume that $(u_\ve,A_\ve)\in H^1(\Omega,\C)\times H^1(\Omega,\R^3)$ is a configuration such that $F_\ve(u_\ve,A_\ve)\leq \ve^{-\gamma}$, so that, by Lemma \ref{Lemma:Grid}, there exists a grid $\GG(b_\ve,R_0,\de)$ satisfying \eqref{propGrid}.
For each face $\omega\subset \RR_2(\GG(b_\ve,R_0,\de))$ of a cube of the grid, let $|I_\omega|$ be the number of connected components of $\{x\in \omega \ | \ |u_\ve(x)|\leq 1/2 \}$ whose degree is different from zero. Then, letting
\begin{equation}\label{CubesUsed}
\GG_0\colonequals \{ \CC_l\in \GG \ | \ \textstyle \sum_{\omega\subset\partial \CC_l}|I_\omega| > 0\}
\end{equation}
and defining $\nu_\ve$ by \eqref{def:vortapprox}, we have
$$
\mathrm{supp}(\nu_\ve)\subset S_{\nu_\ve}\colonequals
\bigcup_{\CC_l\in \GG_0} \CC_l \cup \left\{\begin{array}{cl}\overline\Theta&\mathrm{if}\ \sum_{\omega \subset \partial\GG}|I_\omega|>0\\ \emptyset&\mathrm{if}\ \sum_{\omega \subset \partial\GG}|I_\omega|=0\end{array} \right. .
$$
Moreover
$$
|S_{\nu_\ve}|\leq C\de (1+\de F_\ve(u_\ve,A_\ve)),
$$
where $C$ is a constant depending only on $\partial\Omega$.
\end{lemma}
\begin{proof}
The first assertion follows readily from the definition of $\nu_\ve$. Recall that, by \eqref{numberofconncomp}, the number of faces $\omega\in \RR_2(\GG(b_\ve,R_0,\de))$ of a cube of the grid such that $|I_\omega|>0$ is bounded above by $C\de^{-1}F_\ve(u_\ve,A_\ve)$.
We deduce that $\#(\{l \ | \ \CC_l\in \GG_0\})$ is bounded above by $C\de^{-1}F_\ve(u_\ve,A_\ve)$.
By noting that $|\Theta|\leq C\de$, for some constant $C$ depending only on $\partial\Omega$, we conclude that
$$
|S_{\nu_\ve}|\leq \sum_{\CC_l\in\GG_0}|\CC_l|+|\Theta|\leq \de^3\#(\{l \ | \ \CC_l\in \GG_0\}) +C\de \leq 
C\de(1+\de F_\ve(u_\ve,A_\ve)).
$$
\end{proof}

%%%%%%%%%%%%%%%%%%%%%%%%%%%%%%%%%%%%%%%%%%%%%%%%%%%%%%%%%%%%%%%%%%%%%%%%%%%%%%%%%%%%%%%%%%%%%%%%%

\section{Lower bound for \texorpdfstring{$E_\ve(u_\ve)$}{E(u)} far from the boundary}\label{Sec:LBcubes}
In this section we provide a lower bound, in the spirit of \eqref{LowerBound}, for the energy without magnetic field $E_\ve(u_\ve)$ in the union of cubes of the grid $\GG(b_\ve,R_0,\de)$ given by Lemma \ref{Lemma:Grid}. The proof relies on a slicing procedure based on the level sets of the smooth approximation of the function $\zeta$ constructed in Appendix \ref{Sec:AppendixA} and on the ball construction method on a surface of Section \ref{Sec:Ball}.

\begin{proposition}\label{propcubes}
Let $m,M>0$ and assume that $(u_\ve,A_\ve)\in H^1(\Omega,\C)\times H^1(\Omega,\R^3)$ is such that $F_\ve(u_\ve,A_\ve)=M_\ve\leq M|\log \ve|^m$. For any $b,q>0$, there exists $\ve_0>0$ depending only on $b,q,m$, and $M$, such that, for any $\ve<\ve_0$, letting $\GG(b_\ve,R_0,\de)$ denote the grid given by Lemma \ref{Lemma:Grid} with $\de=\de(\ve)=|\log \ve|^{-q}$, and defining $\nu_\ve$ by \eqref{def:vortapprox} and $\GG_0$ by \eqref{CubesUsed}, if  \begin{equation}\label{EnergyBound}
E_\ve(u_\ve,\cup_{\CC_l\in \GG_0} \CC_l)\leq KM_\ve\quad\mathrm{and}\quad \sum_{\CC_l\in\GG_0}\int_{\partial\CC_l} e_\ve(u_\ve)d\H^2\leq K\de^{-1}M_\ve,
\end{equation}
for some universal constant $K$, then
$$
E_\ve(u_\ve,\cup_{\CC_l\in \GG_0} \CC_l)\geq \frac12\sum_{\CC_l\in \GG_0}|\nu_{\ve,\CC_l}|\left(\log \frac1\ve-\log C\frac{M_\ve^{56}|\log \ve|^{7(1+b)}}{\de^{55}}\right)-\frac C{|\log \ve |^b},
$$
where $C$ is a universal constant.
\end{proposition}
\begin{proof} Let us first prove an estimate for each cube of the grid. 
	
\medskip\noindent 
{\bf Step 1. Lower bound via the co-area formula.} 
We consider a cube $\CC_l\in\GG_0$. For each of the six faces $\omega$ of $\CC_l$, denote by $\{S_{i,\omega}\}_{i\in I_\omega}$ the collection of connected components of $\{ x\in \omega \ | \ |u_\ve(x)|\leq 1/2\}$. We define
$$
S_l\colonequals \cup_{\omega\subset \partial \CC_l}\cup_{i\in I_\omega}S_{i,\omega}.
$$
Note that $|u_\ve(x)|>1/2$ for any $x\in \partial \CC_l\setminus S_l$. 

Denote by $\AA_l=\{ p_1,\dots p_{k_l},n_1,\dots, n_{k_l}\}$ the configuration of positive and negative points associated to the cube $\CC_l$ (see Section \ref{subsec:vorticity}). For parameters $\rho\in(0,1/2)$ and $\la=\la(l)\leq (C_0(2k_l)^{-6})^{1/\rho}$ to be chosen later on, let $\zeta_\la$ be the smooth function associated to $\AA_l$ by Proposition \ref{prop:functionzeta}. Here the constant $C_0$ is the universal constant appearing in the proposition. For $\kappa=\kappa(l)<\la^{2\rho}/3$ consider the set $T_{\kappa}$ defined by \eqref{setoft} and observe that 
$$
|T_{\kappa}|\leq C(2k_l)^8 \frac\la{(\la^{2\rho}-3\kappa)},
$$ 
where throughout the proof $C>0$ denotes a universal constant that may change from line to line. 
Letting
$$
\tilde \CC_l\colonequals \{x\ | \ d(x,\CC_l)<C_1\la^2\kappa, \ \mathrm{proj}_{\CC_l}x\not \in S_l\},
$$
where $C_1$ is the universal constant appearing in the third statement of Proposition \ref{prop:functionzeta},
we define $v_\ve:\overline {\tilde \CC_l}\to \C$ via the formula
$$
v_\ve(x)=u_\ve(\mathrm{proj}_{\CC_l}x)\quad x\in \overline{\tilde \CC_l}.
$$ 
Observe that
$$
E_\ve(u_\ve,\CC_l)\geq E_\ve(v_\ve,\tilde \CC_l)-C_1\la^2\kappa \int_{\partial\CC_l\setminus S_l}e_\ve(u_\ve)d\H^2.
$$
In particular, if $\la^2\kappa$ is small enough then $E_\ve(v_\ve,\tilde \CC_l)\leq 2M_\ve$. We also define
$$
U_\la\colonequals \zeta_\la ( \{x\in\partial \tilde \CC_l\ | \ \mathrm{proj}_{\CC_l} x\in S_l\})
$$
and note that 
$$
|U_\la|\leq \sum_{\omega \subset \partial \CC_l}\sum_{i\in I_\omega}\mathrm{diam}(S_{i,\omega})+(2k_l)C_1\la^2\kappa.
$$
Since $|\nabla \zeta_\la|\leq 1$, using the co-area formula, we deduce that
$$
E_\ve(v_\ve, \tilde \CC_l)\geq\int_{\tilde \CC_l}e_\ve(v_\ve)|\nabla \zeta_\la| = \int_{t\in \R}\int_{\{\zeta_\la=t\}\cap \tilde \CC_l} e_\ve(v_\ve)d\H^2 dt.
$$

\medskip\noindent 
{\bf Step 2. Lower bound via the ball construction on a surface.}
We would now like to apply the results of Section \ref{Sec:Ball}. Let us consider a small number $\gamma>0$ and define 
$$
V_\gamma\colonequals \left\{ t\in \R \ \vline  \  \int_{\{\zeta_\la=t\}\cap \tilde \CC_l} e_\ve(v_\ve)d\H^2>\frac1\gamma M_\ve \right\}.
$$
Note that $|V_\gamma|\leq 2K\gamma$.
Finally, let us define 
$$
T_{bad}=T_{\kappa} \cup U_\la \cup V_\gamma\cup \zeta_\la(P_\la),
$$ 
where $P_\la$ is the set appearing in Proposition \ref{prop:functionzeta},  $\Sigma_t\colonequals \{\zeta_\la =t\}\cap \tilde \CC_l$, $t_*\colonequals \min_{a_i\in \AA_l}\zeta_\la(a_i)$, and $\ t^*\colonequals \max_{a_i\in \AA_l}\zeta_\la(a_i) $. For $t\in T_{good}\colonequals [t_*,t^*]\setminus T_{bad}$ it holds that:
\begin{itemize}[leftmargin=*]
\item $\int_{\Sigma_t}e_\ve(v_\ve)d\H^2\leq \gamma^{-1} M_\ve$.
\item $\{\zeta_\la=t\}$ is a surface whose second fundamental form is bounded by $C_1(\la^2\kappa)^{-1}$. Note that this surface is necessarily oriented since it is a level set of $\zeta_\la$. 
\item $\partial \Sigma_t=\{\zeta_\la =t\}\cap \partial \tilde \CC_l$.
\item $|v_\ve(x)|>1/2$ if $d(x,\partial \Sigma_t)<C_1\la^2\kappa$.
\end{itemize}
Then Corollary \ref{ball} \footnote{To apply the corollary, we actually need $\la^{-1},\kappa^{-1},$ and $\gamma^{-1}$ to be bounded above by positive powers of $|\log \ve|$. For this reason, and because of our choice of these parameters in terms of $\de(\ve)$ (see Step 3), we require $\de$ to be a negative power of $|\log \ve|$ in the statement.} yields that, for any $t\in T_{good}$,
$$
\int_{\Sigma_t}e_\ve(v_\ve)d\H^2\geq \pi |\mathrm{deg}(v_\ve,\partial \Sigma_t)|\left( \log \frac 1\ve -\log \frac{C_1M_\ve}{\la^2 \kappa\gamma}  \right).
$$
We point out that we cannot directly apply Corollary \ref{ball} to $u_\ve$ in $\CC_l$. For this reason, we extended the function $u_\ve$ to $\tilde \CC_l$ in the previous step. 

Noting that $\partial \Sigma_t=\partial (\{\zeta_\la \geq t\}\cap \partial \tilde \CC_l)$, we deduce that
$$
\mathrm{deg}(v_\ve,\partial \Sigma_t)=d(t)\colonequals\#\{i \ | \ \zeta_\la (p_i)>t\}-\#\{i \ | \ \zeta_\la (n_i)>t\}.
$$
By combining our previous estimates, we find
\begin{align*}
E_\ve(v_\ve, \tilde \CC_l)&\geq \int_{t\in T_{good}}\int_{\Sigma_t}e_\ve(v_\ve)d\H^2dt\\
&\geq \pi\left( \log \frac 1\ve -\log \frac{C_1M_\ve}{\la^2 \kappa\gamma}  \right)\int_{t\in T_{good}}d(t)dt\\
&\geq \pi\left( \log \frac 1\ve -\log \frac{C_1M_\ve}{\la^2 \kappa\gamma}   \right)\left(\int_{t_*}^{t^*}d(t)dt-\int_{t\in T_{bad}}|d(t)|dt \right).
\end{align*}
But, for any $t\in T_{bad}$,
$$
|d(t)|=|\#\{i \ | \ \zeta_\la (p_i)>t\}-\#\{i \ | \ \zeta_\la (n_i)>t\}|\leq k_l.
$$
Then 
$$
\int_{t\in T_{bad}}|d(t)|dt\leq k_l|T_{bad}|\leq k_l(|T_{\kappa}|+|U_\la|+|V_\gamma|+|\zeta_\la(P_\la)|).
$$
On the other hand, observe that
$$
\int_{t_*}^{t^*}d(t)dt=\int_{t_*}^{t^*} \left(\#\{i \ | \ \zeta_\la (p_i)>t\}-\#\{i \ | \ \zeta_\la (n_i)>t\}\right)dt=\sum_{i=1}^{k_\la}\zeta_\la(p_i)-\zeta_\la(n_i).
$$
Since
$$
\left|L(\AA_l)-\sum_{i=1}^{k_l}\zeta_\la(p_i)-\zeta_\la(n_i)\right|\leq CD_{\AA_l}(2k_l)^6\la^\rho\leq C\de (2k_l)^6\la^\rho
$$
and remembering that $|2\pi L(\AA_l)-|\nu_{\ve,\CC_l}||$ can be taken arbitrarily small, we conclude that 
$$
\int_{t_*}^{t^*}d(t)dt\geq \frac1{2\pi}|\nu_{\ve,\CC_l}|-C\left(\de(2k_l)^6\la^\rho+(2k_l)\la\right).
$$
Collecting our previous computations, we find
$$
E_\ve(u_\ve,\CC_l)\geq \frac12 |\nu_{\ve,\CC_l}|\left(\log \frac1\ve -\log \frac{C_1M_\ve}{\la^2\kappa\gamma}\right)-\E_l,
$$
where 
$$
\E_l\colonequals  C\left(\de(2k_l)^6\la^\rho +k_l(|T_{\kappa}|+|U_\la|+|V_\gamma|+|\zeta_\la(P_\la)|)\right)\log \frac 1\ve \\
+C_1\la^2\kappa \int_{\partial\CC_l\setminus S_l}e_\ve(u_\ve)d\H^2.
$$

\medskip\noindent 
{\bf Step 3. Choice of the parameters.}
We now want to combine the estimates found for cubes in $\GG_0$. Observe that if $\la$ and $\kappa$ are chosen independent of $l$ then
$$
E_\ve(u_\ve,\cup_{\CC_l\in \GG_0} \CC_l)\geq \frac12\sum_{\CC_l\in \GG_0}|\nu_{\ve,\CC_l}|\left(\log \frac1\ve-\log \frac{C_1M_\ve}{\la^2\kappa\gamma}\right)-\sum_{\CC_l\subset\GG_0}\E_l.
$$
Our objective is then to choose the parameters $\la, \kappa=\kappa(\la)$, and $\gamma$ independent of $l$ and such that $\sum_{\CC_l\subset\GG_0}\E_l\leq C|\log \ve|^{-b}$. Since \eqref{numberofpoints} implies that 
$$
\textstyle \sum_{\CC_l\in \GG_0}2k_l\leq C\de^{-1}M_\ve,
$$ 
we can achieve our goal provided that $\la$ satisfies $\la\leq (C_0(C\de^{-1}M_\ve)^{-6})^{1/\rho}$.

Letting $\kappa=\la^{2\rho}/6$, and using \eqref{numberofpoints}, \eqref{radiusGrid}, and \eqref{EnergyBound}, we are led to
$$
\sum_{\CC_l\in\GG_0}\E_l\leq C\log \frac 1\ve
\left(
\frac{M_\ve^6}{\de^5}\la^\rho  +\frac{M_\ve^9}{\de^9}\la^{1-2\rho}+\frac{M_\ve^2}{\de^2}\ve +\frac{M_\ve}{\de} \gamma
\right).
$$
Then, choosing $\rho=6/21$,
$$
\la=\left(\frac{1}{|\log \ve|^{1+b}}\frac{\de^9}{M_\ve^9}\right)^{\frac{21}9},\ \mathrm{and}\quad \gamma= \frac{1}{|\log \ve|^{1+b}}\frac{\de}{M_\ve},
$$
we easily check that there exists $\ve_0>0$ depending only on $b,q,m$, and $M$, such that $\sum_{\CC_l\in \GG_0}\E_l\leq C|\log\ve|^{-b}$ for any $0<\ve<\ve_0$. Thus
$$
E_\ve(u_\ve,\cup_{\CC_l\in \GG_0} \CC_l)\geq \frac12\sum_{\CC_l\in \GG_0}|\nu_{\ve,\CC_l}|\left(\log \frac1\ve-\log C\frac{M_\ve^{56}|\log \ve|^{7(1+b)}}{\de^{55}}\right)-\frac C{|\log\ve|^b}.
$$
The proposition is proved.
\end{proof}

%%%%%%%%%%%%%%%%%%%%%%%%%%%%%%%%%%%%%%%%%%%%%%%%%%%%%%%%%%%%%%%%%%%%%%%%%%%%%%%%%%%%%%%%%%%%%%%%%

\section{Lower bound for \texorpdfstring{$E_\ve(u_\ve)$}{E(u)} close to the boundary}\label{Sec:LBboundary}
In this section we provide a lower bound, in the spirit of \eqref{LowerBound}, for the energy without magnetic field $E_\ve(u_\ve)$ in $\Theta$. The proof relies on a slicing procedure based on the level sets of the smooth approximation of the function $\zeta$ for $d_{\partial\Omega}$ constructed in Appendix \ref{Sec:AppendixB} and on the ball construction method on a surface of Section \ref{Sec:Ball}.

\begin{proposition}\label{propboundary}
Suppose that $\partial\Omega$ is of class $C^2$.
Let $m,M>0$ and assume that $(u_\ve,A_\ve)\in H^1(\Omega,\C)\times H^1(\Omega,\R^3)$ is such that $F_\ve(u_\ve,A_\ve)=M_\ve\leq M|\log \ve|^m$. For any $b,q>0$, there exists $\ve_0>0$ depending only on $b,q,m,M$, and $\partial\Omega$, such that, for any $\ve<\ve_0$, letting $\GG(b_\ve,R_0,\de)$ denote the grid given by Lemma \ref{Lemma:Grid} with $\de=\de(\ve)=|\log \ve|^{-q}$, and defining $\nu_\ve$ by \eqref{def:vortapprox} and $\Theta,\partial\GG$ by \eqref{unioncubes}, if 
\begin{equation}\label{EnergyBoundBoundary}
E_\ve(u_\ve,\Theta)\leq KM_\ve\quad \mathrm{and}\quad \int_{\partial\GG}e_\ve(u_\ve)d\H^2\leq K\de^{-1}M_\ve,
\end{equation}
for some universal constant $K>0$, then 
$$
E_\ve(u_\ve,\Theta)\geq \frac12|\nu_{\ve,\Theta}|\left(\log \frac1\ve-\log C\frac{M_\ve^{124}|\log \ve|^{(20+1/3)(1+b)}}{\de^{123}}\right)-\frac{C}{|\log \ve |^{b}}.
$$
where $C$ is a constant depending only on $\partial\Omega$.
\end{proposition}
\begin{proof} We proceed in a similar way to the proof of Proposition \ref{propcubes}.
	
\medskip\noindent
{\bf Step 1. Lower bound via the co-area formula.} 
For each face $\omega\in \RR_2(\GG(b_\ve,R_0,\de))$ of a cube of the grid such that $\omega \subset \partial \GG$, denote by $\{S_{i,\omega}\}_{i\in I_\omega}$ the collection of connected components of $\{ x\in \omega \ | \ |u_\ve(x)|\leq 1/2\}$. We define
$$
S_{\partial\GG}\colonequals \cup_{\omega\subset \partial \GG}\cup_{i\in I_\omega}S_{i,\omega}.
$$
Note that $|u_\ve(x)|>1/2$ for any $x\in \partial \GG\setminus S_{\partial\GG}$. 

Denote by $\tilde \AA_{\partial\GG}=\{ p_1,\dots p_{\tilde k_{\partial\GG}},n_1,\dots, n_{\tilde k_{\partial\GG}}\}\subset \Omega$ the configuration of positive and negative points associated to $\partial \GG$ defined in \eqref{newcollection}. For parameters $\rho\in (0,1/4)$ and $\la\leq \left(C_0(C\de^{-1}M_\ve)^{-6}\right)^{1/\rho}$
to be chosen later on, let $\zeta_\la$ be the smooth function associated to $\tilde \AA_{\partial\GG}$ by Proposition \ref{prop:functionzetaBoundary}, or, more precisely, the function described in Section \ref{importantlb} for which (up to a relabeling of the points) the quantity $L_{\partial\Omega}(\tilde \AA_{\partial\GG})$ is replaced with $L_{\d_{\partial\Omega}}(\tilde \AA_{\partial\GG})$ in the statement of the proposition. Here the constant $C_0=C_0(\partial\Omega)$ is the constant appearing in the proposition. For $\kappa<\la^{2\rho}/3$
consider the set $T_{\kappa}$ defined by \eqref{setoftBoundary} and observe that 
$$
|T_{\kappa}|\leq C(2\tilde k_{\partial\GG})^8 \frac\la{(\la^{2\rho}-3\kappa)}+C(2 \tilde  k_{\partial\GG})^6\frac\la{(\la^{2\rho}-3\kappa)^4},
$$
where throughout the proof $C>0$ denotes a constant depending only on $\partial\Omega$, that may change from line to line.

Let 
$$
\partial \tilde \GG=\{x\in \Omega\setminus \Theta \ | \ \min_{y\in \partial\GG}\|x-y\|_\infty =C_1\la^2\kappa\},
$$
where $C_1$ is the universal constant appearing in the fourth statement of Proposition \ref{prop:functionzetaBoundary}.
Observe that $\partial\tilde \GG$ corresponds to a shrunk version of the polyhedron $\partial\GG$, or, in other words, a smaller version of $\partial\GG$ with the same shape. Each face $\omega\subset \partial\GG$ has a parallel counterpart face $\tilde \omega\subset \partial\tilde \GG$ which corresponds to a translated and in some cases also a shrunk version of $\omega$. It is easy to see that there exists a bijective function $f:\partial\GG\to \partial \tilde \GG$ mapping any $x\in \omega\subset\partial\GG$ to its unique counterpart point $\tilde x \in \tilde \omega\subset \partial\tilde \GG$. One immediately checks that for any $x,y\in \omega \subset \partial \GG$
$$
|f(x)-x|\leq \sqrt 2C_1\la^2 \kappa \quad \mathrm{and}\quad \frac1{\sqrt2} |x-y|\leq |f(x)-f(y)|\leq |x-y|.
$$
Denoting by $\O$ the open region enclosed by $\partial\GG$ and $\partial\tilde \GG$, we observe that for any $y\in \overline \O$ there exists a unique $x_y\in\partial\GG$ and a unique $t_y\in[0,1]$ such that $y=tx+(1-t)f(x)$. 
Letting
$$
\tilde \O  \colonequals \{y \in \O \ | \ x_y \not\in S_{\partial\GG}\},
$$
we define $v_\ve:\overline{\Theta\cup\tilde\O} \to \C$ by
$$
v_\ve(y)=u_\ve(y)\quad \mathrm{if}\ y\in \overline \Theta, \quad v_\ve(y)=u_\ve(x_y)\quad \mathrm{if}\ y\in \overline {\tilde \O}.
$$ 
Note that $v_\ve$ is a $H^1$-extension of $u_\ve$ and that
$$
E_\ve(v_\ve,\tilde \O)\leq E_\ve(v_\ve,\O)\leq \sqrt 2 C_1\la^2\kappa \int_{\partial \tilde \GG}e_\ve(v_\ve)d\H^2\leq 2 C_1\la^2\kappa \int_{\partial\GG}e_\ve(u_\ve)d\H^2.
$$
Thus 
$$
E_\ve(u_\ve,\Theta)\geq E_\ve(v_\ve,\overline \Theta\cup \tilde \O)-E_\ve(v_\ve,\tilde \O)\geq E_\ve(v_\ve,\overline \Theta\cup \tilde \O)-2C_1\la^2\kappa \int_{\partial\GG}e_\ve(u_\ve)d\H^2.
$$
In particular, if $\la^2\kappa$ is small enough then $E_\ve(v_\ve,\overline  \Theta\cup \tilde \O)\leq 2M_\ve$. We also define
$$
U_\la\colonequals \zeta_\la ( \{y\in\partial (\overline \Theta\cup\tilde \O)\setminus \partial\Omega \ | \ x_y \in S_{\partial\GG}\})
$$
and note that by \eqref{numberofpoints}, \eqref{radiusGrid}, and \eqref{numbertilde}, we have 
$$
|U_\la|\leq |S_{\partial\GG}|+ \sqrt 2 (2\tilde k_{\partial\Omega})C_1\la^2\kappa\leq C\ve \de^{-1}M_\ve + C\de^{-1}M_\ve\la^2\kappa.
$$
Since $|\nabla \zeta_\la|\leq 1$, using the co-area formula, we deduce that
$$
E_\ve(v_\ve,\overline \Theta\cup \tilde \O)\geq\int_{\overline \Theta\cup \tilde \O}e_\ve(v_\ve)|\nabla \zeta_\la| = \int_{t\in \R}\int_{\{\zeta_\la=t\}\cap (\overline \Theta\cup \tilde \O)} e_\ve(v_\ve)d\H^2 dt.
$$
We remark that if $\de$ and $\la^2\kappa$ are small enough then $\overline \Theta\cup \tilde \O \subset \Omega_\la \cup  \{x\in \Omega \ | \ d(x,\Omega)\leq 2\la^{1/\rho}\}$, where $\Omega_\la$ is the set defined by \eqref{Omegalambda}.

\medskip\noindent
{\bf Step 2. Lower bound via the ball construction on a surface.}
We would now like to apply the results of Section \ref{Sec:Ball}. Let us consider a small number $\gamma>0$ and define 
$$
V_\gamma\colonequals \left\{ t\in \R \ \vline  \  \int_{\{\zeta_\la=t\}\cap 
(\overline \Theta\cup \tilde \O)} e_\ve(v_\ve)d\H^2>\frac1\gamma M_\ve \right\}.
$$
Note that $|V_\gamma|\leq 2K\gamma$.
Finally, let us define 
$$
T_{bad}=T_{\kappa} \cup U_\la \cup V_\gamma\cup \zeta_\la(\{x\in \Omega \ | \ d(x,\Omega)\leq 2\la^{1/\rho}\})\cup \zeta_\la(P_\la),
$$
where $P_\la$ is the set appearing in Proposition \ref{prop:functionzetaBoundary},  $\Sigma_t\colonequals \{\zeta_\la =t\}\cap (\overline \Theta\cup \tilde \O)$, $t_*\colonequals \min_{a_i\in \AA_{\partial\GG}}\zeta_\la(a_i)$, and $\ t^*\colonequals \max_{a_i\in \AA_{\partial\GG}}\zeta_\la(a_i) $. For $t\in T_{good}\colonequals [t_*,t^*]\setminus T_{bad}$ it holds that:
\begin{itemize}[leftmargin=*]
\item $\int_{\Sigma_t}e_\ve(v_\ve)d\H^2\leq \gamma^{-1} M_\ve$.
\item $\{\zeta_\la=t\}=\{x\in\Omega_\la \ | \ \zeta_\la(x)=t\}$ is a surface whose second fundamental form is bounded by $C_1(\la^2\kappa)^{-1}$. Note that this surface is necessarily oriented since it is a level set of $\zeta_\la$. 
\item $\partial \Sigma_t=\{\zeta_\la =t\}\cap (\partial(\overline \Theta\cup \tilde \O)\setminus \partial\Omega)$.
\item $|v_\ve(x)|>1/2$ if $d(x,\partial \Sigma_t)<C_1\la^2\kappa$.
\end{itemize}
Then Corollary \ref{ball} yields that, for any $t\in T_{good}$,
$$
\int_{\Sigma_t}e_\ve(v_\ve)d\H^2\geq \pi |\mathrm{deg}(v_\ve,\partial \Sigma_t)|\left( \log \frac 1\ve -\log \frac{C_1M_\ve}{ \la^2 \kappa\gamma}  \right).
$$
We point out that we cannot directly apply Corollary \ref{ball} to $u_\ve$ in $\Theta$. For this reason, we extended the function $u_\ve$ to $\overline{\Theta\cup \tilde \O}$ in the previous step. 

Noting that $\partial \Sigma_t=\partial (\{\zeta_\la \geq t\}\cap (\partial(\overline \Theta\cup \tilde \O)\setminus \partial\Omega))$, we deduce that
$$
\mathrm{deg}(v_\ve,\partial \Sigma_t)=d(t)\colonequals\#\{i \ | \ \zeta_\la (p_i)>t\}-\#\{i \ | \ \zeta_\la (n_i)>t\}.
$$
By combining our previous estimates, we find
\begin{align*}
E_\ve(v_\ve, \overline \Theta\cup \tilde \O)&\geq \int_{t\in T_{good}}\int_{\Sigma_t}e_\ve(v_\ve)d\H^2dt\\
&\geq \pi\left( \log \frac 1\ve -\log \frac{C_1M_\ve}{\la^2 \kappa\gamma }  \right)\int_{t\in T_{good}}d(t)dt\\
&\geq \pi\left( \log \frac 1\ve -\log \frac{C_1M_\ve}{\la^2 \kappa\gamma } \right)\left(\int_{t_*}^{t^*}d(t)dt-\int_{t\in T_{bad}}|d(t)|dt \right).
\end{align*}
But, for any $t\in T_{bad}$,
$$
|d(t)|=|\#\{i \ | \ \zeta_\la (p_i)>t\}-\#\{i \ | \ \zeta_\la (n_i)>t\}|\leq \tilde  k_{\partial\GG}.
$$
Then 
\begin{multline*}
\int_{t\in T_{bad}}|d(t)|dt\leq \tilde  k_{\partial\GG}|T_{bad}|\leq   \tilde k_{\partial\GG}(|T_{\kappa}|+|U_\la|+|V_\gamma|\\+|\zeta_\la(\{x\in \Omega \ | \ d(x,\Omega)\leq 2\la^{1/\rho}\})|+|\zeta_\la(P_\la)|).
\end{multline*}
On the other hand, observe that
$$
\int_{t_*}^{t^*}d(t)dt=\int_{t_*}^{t^*} \left(\#\{i \ | \ \zeta_\la (p_i)>t\}-\#\{i \ | \ \zeta_\la (n_i)>t\}\right)dt=\sum_{i=1}^{\tilde k_{\partial\GG}}\zeta_\la(p_i)-\zeta_\la(n_i).
$$
Since 
$$
\left|L_{\d_{\partial\Omega}}(\tilde \AA_{\partial\GG})-\sum_{i=1}^{\tilde k_{\partial\GG}}\zeta_\la(p_i)-\zeta_\la(n_i)\right|\leq C(2\tilde k_{\partial\GG})^6\la^\rho
$$
and remembering that $|2\pi L_{\d_{\partial\Omega}}(\tilde\AA_{\partial\GG})-|\nu_{\ve,\Theta}||$ can be taken arbitrarily small, we conclude that 
$$
\int_{t_*}^{t^*}d(t)dt\geq \frac1{2\pi}|\nu_{\ve,\Theta}|- C(2\tilde k_{\partial\GG})^6\la^\rho.
$$
Collecting our previous computations, we find
$$
E_\ve(u_\ve,\Theta)\geq \frac12 |\nu_{\ve,\Theta}|\left(\log \frac1\ve -\log \frac{C_1M_\ve}{\la^2\kappa\gamma}\right)-\E_\Theta,
$$
where 
\begin{multline*}
\E_\Theta\colonequals C\big((2\tilde k_{\partial\GG})^6\la^\rho +\tilde k_{\partial\GG}\big(|T_{\kappa}|+|U_\la|+|V_\gamma|+|\zeta_\la(\{x\in \Omega \ | \ d(x,\Omega)\leq 2\la^{1/\rho}\})|\\
+|\zeta_\la(P_\la)|\big)\big) \log \frac1{\ve}
+ 2C_1\la^2\kappa \int_{\partial\GG}e_\ve(u_\ve)d\H^2.
\end{multline*}

\medskip\noindent
{\bf Step 3. Choice of the parameters.}
We now choose the parameters $\rho$, $\la$, $\kappa(\la)$ and $\gamma$. 
Observe that \eqref{numberofpoints} and \eqref{numbertilde}, imply that 
$$
(2\tilde k_{\partial\GG})\leq C\de^{-1}M_\ve.
$$ 
Letting $\kappa=\la^{2\rho}/6$ and using \eqref{EnergyBoundBoundary}, we are led to
$$
\E_\Theta\leq C \log \frac1\ve \left( \frac{M_\ve^6}{\de^6}\la^\rho+\frac{M_\ve^4}{\de^4}\la^{3\rho/4} +\frac{M_\ve^9}{\de^9}\la^{1-2\rho}+\frac{M_\ve^7}{\de^7}\la^{1-8\rho}+\frac{M_\ve^2}{\de^2}\ve +\frac{M_\ve}{\de} \gamma  \right).
$$
Then, choosing $\rho=6/55$,
$$
\la=\left(\frac1{|\log \ve|^{1+b}}\frac{\de^6}{M_\ve^6}\right)^{\frac{55}6},\ \mathrm{and}\quad \gamma= \frac1{|\log \ve|^{1+b}}\frac{\de}{M_\ve},
$$
we easily check that there exists $\ve_0>0$ depending only on $b,q,m,M$, and $\partial\Omega$, such that $\E_\Theta\leq C|\log\ve|^{-b}$ for any $0<\ve<\ve_0$. Thus
$$
E_\ve(u_\ve,\Theta)\geq \frac12|\nu_{\ve,\Theta}|\left(\log \frac1\ve-\log C\frac{M_\ve^{124}|\log \ve|^{(20+1/3)(1+b)}}{\de^{123}}\right)-\frac{C}{|\log \ve |^{b}}.
$$
This concludes the proof. 
\end{proof}

\begin{remark} We remark that if in addition $\partial\Omega$ has strictly positive Gauss curvature then we can use the smooth approximation of the function $\zeta$ for $d_{\partial\Omega}$ constructed in Appendix \ref{Sec:AppendixC}. By using Proposition \ref{prop:functionzetaBoundary2} with $\tau=\de$ and arguing similarly as above, one can prove that 
$$
E_\ve(u_\ve,\Theta)\geq \frac12|\nu_{\ve,\Theta}|\left(\log \frac1\ve-\log C\frac{M_\ve^{56}|\log \ve|^{7(1+b)}}{\de^{55}}\right)-C\de M_\ve |\log \ve|-\frac{C}{|\log \ve |^{b}}.
$$
\end{remark}

%%%%%%%%%%%%%%%%%%%%%%%%%%%%%%%%%%%%%%%%%%%%%%%%%%%%%%%%%%%%%%%%%%%%%%%%%%%%%%%%%%%%%%%%%%%%%%%%%

\section{Proof of the main results}\label{Sec:proofMain}
\subsection{Proof of Theorem \ref{Theorem1}}
First, using the results of the previous two sections we prove \eqref{LowerBound}.
\begin{proof}[Proof of \eqref{LowerBound}]
Since the energy $F_\ve(u_\ve,A_\ve)$ is gauge invariant, it is enough to prove the result in the Coulomb gauge, i.e.
$$
\diver A_\ve=0 \ \mathrm{in}\ \Omega
\quad \mathrm{and}\quad A_\ve\cdot \nu =0 \ \mathrm{on}\ \partial\Omega.
$$
We immediately check that 
$$
\|A_\ve\|_{H^1(\Omega,\R^3)}\leq C\|\curl A_\ve\|_{L^2(\Omega, \R^3)},
$$
where throughout the proof $C>0$ denotes a universal constant that may change from line to line. By Sobolev embedding theorem we have 
$$
\|A_\ve\|_{L^p(\Omega,\R^3)}\leq C\|A_\ve\|_{H^1(\Omega,\R^3)}
$$
for any $1\leq p\leq 6$. Observe that
$$
\int_\Omega |\nabla u_\ve|^2\leq \int_\Omega |\nabla_{A_\ve} u_\ve|^2 +|u_\ve|^2|A_\ve|^2\leq \int_\Omega |\nabla_{A_\ve} u_\ve|^2+ (|u_\ve|^2-1)|A_\ve|^2+ |A_\ve|^2
$$
By the Cauchy-Schwarz inequality, we have
$$
\int_\Omega (|u_\ve|^2-1)|A_\ve|^2\leq \left(\int_\Omega (1-|u_\ve|^2)^2\right)^{\frac 12}\left(\int_\Omega |A_\ve|^4\right)^{\frac 12} \leq C\ve F_\ve(u_\ve,A_\ve).
$$
Thus
$$
E_\ve(u_\ve)\leq CF_\ve(u_\ve,A_\ve).
$$
Let us consider the grid $\GG(b_\ve,R_0,\de)$ given by Lemma \ref{Lemma:Grid}. It is not hard to see that, up to an adjustment of the constant appearing in the lemma, we can require our grid to additionally satisfy the inequalities
\begin{equation}\label{gridE}
\int_{\RR_1(\GG(b_\ve,R_0,\de))}e_\ve(u_\ve)d\H^1\leq C\de^{-2}F_\ve(u_\ve,A_\ve),\quad \int_{\RR_2(\GG(b_\ve,R_0,\de))}e_\ve(u_\ve)d\H^2\leq C\de^{-1}F_\ve(u_\ve,A_\ve).
\end{equation}
We define the polyhedral $1$-current $\nu_\ve$ by \eqref{def:vortapprox}. We recall the notation introduced in Lemma \ref{Lemma:support} and observe that 
$$
\int_{S_{\nu_\ve}}|\nabla u_\ve|^2\leq \int_{S_{\nu_\ve}}|\nabla_{A_\ve} u_\ve|^2+\int_{S_{\nu_\ve}}(|u_\ve|^2-1)|A_\ve|^2+\int_{S_{\nu_\ve}}|A_\ve|^2.
$$
Using H\"older's inequality, we find
$$
\int_{S_{\nu_\ve}}(|u_\ve|^2-1)|A_\ve|^2\leq \||u_\ve|^2-1\|_{L^2(S_{\nu_\ve})}|S_{\nu_\ve}|^{\frac16}\|A_\ve\|^2_{L^6(S_{\nu_\ve},\R^3)}
$$
and
$$
\int_{S_{\nu_\ve}}|A_\ve|^2\leq |S_{\nu_\ve}|^{\frac23}\|A_\ve\|_{L^6(S_{\nu_\ve},\R^3)}^2.
$$
We are led to
$$
\int_{S_{\nu_\ve}}|\nabla u_\ve|^2\leq \int_{S_{\nu_\ve}}|\nabla_{A_\ve} u_\ve|^2+CF_\ve(u_\ve,A_\ve)\left(\ve|S_{\nu_\ve}|^{\frac16}+|S_{\nu_\ve}|^{\frac23}\right),
$$
which implies that
\begin{equation}\label{energy2}
\frac12\int_{S_{\nu_\ve}}|\nabla_{A_\ve}u_\ve|^2+\frac1{2\ve^2}(1-|u_\ve|^2)^2\geq E_\ve(u_\ve,S_{\nu_\ve})-CF_\ve(u_\ve,A_\ve)\left(\ve|S_{\nu_\ve}|^{\frac16}F_\ve(u_\ve,A_\ve)^\frac12+|S_{\nu_\ve}|^{\frac23}\right).
\end{equation}
Thanks to \eqref{gridE}, we can apply Proposition \ref{propcubes} and Proposition \ref{propboundary} with $b=n$ and $q>0$ (in particular $\de=\de(\ve)=|\log \ve|^{-q}$). We then deduce that there exists $\ve_0>0$, depending only on $q,m,n,M,$ and $\partial\Omega$, such that, for any $0<\ve<\ve_0$,
$$
E_\ve(u_\ve,S_{\nu_\ve})\geq \frac12|\nu_\ve|(\Omega)|\left(\log \frac1\ve-\log C\frac{M_\ve^{124}|\log \ve|^{(20+1/3)(1+n)}}{\de^{123}}\right)-\frac{C}{|\log \ve |^{n}},
$$
where $C$ is a constant depending only on $\partial\Omega$. 
By combining this with \eqref{energy2} and Lemma \ref{Lemma:support}, we are led to 
\begin{multline*}
\int_{S_{\nu_\ve}}|\nabla_{A_\ve}u_\ve|^2+\frac1{2\ve^2}(1-|u_\ve|^2)^2+|\curl A_\ve|^2\\
\geq
|\nu_\ve|(\Omega)\left(\log \frac1\ve-\log C\frac{M_\ve^{124}|\log \ve|^{(20+1/3)(1+n)}}{\de^{123}}\right)\\
-CM_\ve \de^{\frac23}-\frac C{|\log\ve|^n}.
\end{multline*}
By letting $q=q(m,n)=\frac32 (m+n)$, we have $M_\ve \de^{\frac23}\leq C |\log\ve|^{-n}$. This concludes the proof of the lower bound.
\end{proof}

Before presenting the proof of \eqref{EstimateJ} for $\gamma=1$, let us prove the following lemma.
\begin{lemma}
Let $(u_\ve,A_\ve)\in H^1(\Omega,\C)\times H^1(\Omega,\R^3)$. Then there exists a constant $C>0$ depending only on $\partial\Omega$, such that
\begin{equation}
\|\mu(u_\ve,A_\ve)\|_{C^0(\Omega)^*}\leq CF_\ve(u_\ve,A_\ve).\label{C^0Bound}
\end{equation}
\end{lemma}
\begin{proof}
By definition 
\begin{align*}
\mu(u_\ve,A_\ve)&=\frac i2d\left(u_\ve d_{A_\ve}\bar u_\ve-\bar u_\ve d_{A_\ve}u_\ve\right)+dA_\ve \\
&=\frac i2 \left(du_\ve \wedge d_{A_\ve}\bar u_\ve +u_\ve d(d_{A_\ve}\bar u_\ve)-d\bar u_\ve \wedge d_{A_\ve} u_\ve -\bar u_\ve d(d_{A_\ve} u_\ve)  \right)+d A_\ve.
\end{align*}
Simple computations show that
\begin{align*}
\mu(u_\ve,A_\ve)&=\frac i2 (du_\ve \wedge d_{A_\ve}\bar u_\ve-iu_\ve d\bar u_\ve \wedge A_\ve -d\bar u_\ve \wedge d_{A_\ve}u_\ve+i\bar u_\ve d u_\ve \wedge A_\ve )+d A_\ve\\
&=\frac i2 \left(d_{A_\ve} u_\ve \wedge d_{A_\ve}\bar u_\ve-d_{A_\ve} \bar u_\ve \wedge d_{A_\ve} u_\ve \right)+dA_\ve=id_{A_\ve}u_\ve \wedge d_{A_\ve}\bar u_\ve +dA_\ve.
\end{align*}
Integrating on $\Omega$ and using the Cauchy-Schwarz inequality, we find
$$
\int_\Omega \mu(u_\ve,A_\ve)\leq 2\left(F_\ve(u_\ve,A_\ve)+F_\ve(u_\ve,A_\ve)^\frac12 |\Omega|^\frac12\right).
$$
Then we easily check that there exists a constant $C(\partial\Omega)>0$ such that 
$$
\left| \int_\Omega \mu(u_\ve,A_\ve)\wedge \phi \right| \leq C\|\phi\|_{C^0(\Omega)}F_\ve(u_\ve,A_\ve),
$$
for any continuous $1$-form $\phi$, which implies \eqref{C^0Bound}.
\end{proof}
\begin{proof}[Proof of \eqref{EstimateJ} for $\gamma=1$] As in the proof of \eqref{LowerBound}, we consider the grid $\GG(b_\ve,R_0,\de)$ given by Lemma \ref{Lemma:Grid} and the polyhedral $1$-current $\nu_\ve$ defined by \eqref{def:vortapprox}. The parameter $\de$ is defined as above.

Let $\phi \in C_T^{0,1}(\Omega)$ be a $1$-form. Note that
\begin{multline}\label{decomp}
\left|\int_\Omega (\mu(u_\ve,A_\ve)-\nu_\ve)\wedge \phi \right|\leq \\
\sum_{\CC_l\in \GG(b_\ve,R_0,\de)} \left|\int_{\CC_l} (\mu(u_\ve,A_\ve)-\nu_{\ve,\CC_l})\wedge \phi\right| + \left|\int_\Theta (\mu(u_\ve,A_\ve)-\nu_{\ve,\Theta})\wedge \phi \right|.
\end{multline}
First, we consider a cube $\CC_l\in\GG(b_\ve,R_0,\de)$ and define  $\phi_l=\int_{\CC_l}\phi$. Observe that
\begin{equation}\label{meanvalue}
\|\phi-\phi_l\|_{C^0(\CC_l)}\leq \delta  \|\phi\|_{C^{0,1}(\CC_l)}
\end{equation}
and that
\begin{multline*}
\left| \int_{\CC_l} (\mu(u_\ve,A_\ve)-\nu_{\ve,\CC_l}) \wedge \phi \right| \leq \\
\left|\int_{\CC_l} (\mu(u_\ve,A_\ve)-\nu_{\ve,\CC_l})\wedge (\phi-\phi_l)\right|+\left|\int_{\CC_l} (\mu(u_\ve,A_\ve)-\nu_{\ve,\CC_l})\wedge \phi_l\right|.
\end{multline*}
Using \eqref{meanvalue}, we deduce that
\begin{equation}\label{estimatecube1}
\left| \int_{\CC_l} (\mu(u_\ve,A_\ve)-\nu_{\varepsilon,\CC_l})\wedge (\phi-\phi_l) \ \right|\leq \delta \|\mu(u_\ve,A_\ve)-\nu_{\varepsilon,\CC_l}\|_{C^0(\CC_l)^*} \|\phi\|_{C^{0,1}(\CC_l)}.
\end{equation}
On the other hand, since $\phi_l$ is a constant, there exist a function $f_l$ such that 
$$
\phi_l=df_l,\quad \int_{\CC_l}f_l=0.
$$ 
In particular
$$
\|f_l\|_{C^{0,1}(\CC_l)}\leq |\phi_l|.
$$
By an integration by parts, we have
$$
\int_{\CC_l} (\mu(u_\ve,A_\ve)-\nu_{\varepsilon,\CC_l}) \wedge \phi_l= \sum_{\omega\subset \partial \CC_l}\int_{\omega}\left(\mu_{\ve,\omega}-2\pi\sum_{i\in I_\omega}d_{i,\omega}\delta_{a_{i,\omega}}\right)f_l.
$$
Here, we have used the notation introduced in Section \ref{Sec:2DVortEstimate} and the fact that the restriction of $\nu_{\varepsilon,\CC_l}$ to each of the six faces $\omega$ of the cube $\CC_l$ is equal to $2\pi\sum_{i\in I_\omega}d_{i,\omega}\delta_{a_{i,\omega}}$. Corollary \ref{cor:2dVortEstimate} then yields that
\begin{multline}\label{estimatecube2}
\left| \int_{\CC_l} (\mu_\varepsilon(u_\ve,A_\ve)-\nu_{\varepsilon,\CC_l}) \wedge \phi_l \ \right|\leq \\
C_0\sum_{\omega \subset \partial \CC_l}\max(r_\omega,\ve)\left(1+\int_\omega e_\ve(u_\ve,A_\ve)d\H^2 \right.
\left.
+\int_{\partial \omega} e_\ve(u_\ve,A_\ve)d\H^1\right) \|f_l\|_{C^{0,1}(\CC_l)}.
\end{multline}
where throughout the proof $C_0$ denotes a universal constant that may change from line to line. Using \eqref{estimatecube1} and \eqref{estimatecube2}, we deduce that
\begin{multline*}
\|\mu(u_\ve,A_\ve)-\nu_{\varepsilon,\CC_l}\|_{C^{0,1}(\CC_l)^*}\leq \delta \|\mu(u_\ve,A_\ve)-\nu_{\varepsilon,\CC_l}\|_{C^0(\CC_l)^*}\\
+C_0\sum_{\omega \subset \partial \CC_l}\max(r_\omega,\ve)\left(1+\int_\omega e_\ve(u_\ve,A_\ve)d\H^2+\int_{\partial \omega} e_\ve(u_\ve,A_\ve)d\H^1\right)
\end{multline*}
for any cube $\CC_l\in\GG(b_\ve,R_0,\de)$. Then by summing over the cubes of the grid, we obtain
\begin{multline*}
\|\mu(u_\ve,A_\ve)-\nu_\ve\|_{C^{0,1}(\Omega\setminus \Theta)^*}\leq \delta \|\mu(u_\ve,A_\ve)-\nu_\ve\|_{C^0(\Omega\setminus \Theta)^*} \\
+C_0\max(r_\GG,\ve)\left(1+2\int_{\RR_2(\GG(b_\ve,R_0,\de))}e_\ve(u_\ve,A_\ve)d\H^2\right.
\left.+8\int_{\RR_1(\GG(b_\ve,R_0,\de))}e_\ve(u_\ve,A_\ve)d\H^1\right).
\end{multline*}
Using \eqref{prop2Grid}, \eqref{prop3Grid}, and \eqref{radiusGrid}, we find
\begin{align}
\|\mu(u_\ve,A_\ve)-\nu_\ve\|_{C^{0,1}(\Omega\setminus \Theta)^*}&\leq \delta \|\mu(u_\ve,A_\ve)-\nu_\ve\|_{C^0(\Omega\setminus \Theta)^*}\label{estimateunioncubes}\\
&+C_0\max\left(\ve\de^{-1}F_\ve(u_\ve,A_\ve),\ve\right)\left(1+\delta^{-2}F_\ve(u_\ve,A_\ve)\right)\notag
\end{align}

\medskip
We now provide an estimate for the last term in \eqref{decomp}. Observe that if $\ve$ is sufficiently small, and since $\partial\Omega$ is of class $C^2$, for any $ y\in \overline\Theta$, there exists a unique $x_y=\mathrm{proj}_{\partial \Omega} y$ such that $y=x_y-t_y\nu(x_y)$, for some $t_y\geq 0$, where $\nu(x_y)$ is the outer unit normal to $\partial\Omega$ at $x_y$. We define $f:\Theta\to \R$ by
$$
f(y)=f(x_y-t_y\nu(x))=-t_y\phi(x_y)\cdot \nu(x_y).
$$
By noting that, for any $y\in \Theta$,
$$
\nabla f(y)=\left(\phi(x_y)\cdot \nu(x_y)\right)\nu(x_y)=\phi (x_y),
$$
one can easily check that
$$
\|f\|_{C^{0,1}(\Theta)}\leq \|\phi\|_{C^{0,1}(\Theta)}
\quad\mathrm{and}\quad \|\phi-\nabla f\|_{C^0(\Theta)}\leq \delta \|\phi\|_{C^{0,1}(\Theta)}.
$$
We now write
$$
\int_{\Theta} (\mu(u_\ve,A_\ve)-\nu_{\ve,\Theta}) \wedge \phi =\int_{\Theta} (\mu(u_\ve,A_\ve)-\nu_{\ve,\Theta})\wedge (\phi-d f)+\int_{\Theta} (\mu(u_\ve,A_\ve)-\nu_{\ve,\Theta}) \wedge d f.
$$
Observe that
$$
\left| \int_{\Theta} (\mu(u_\ve,A_\ve)-\nu_{\ve,\Theta})\wedge (\phi-d f) \ \right|\leq \|\mu(u_\ve,A_\ve)-\nu_{\ve,\Theta}\|_{C^0(\Theta)^*} \|\phi-d f\|_{C^0(\Theta)}.
$$
On the other hand, by an integration by parts, we find
$$
\int_{\Theta} (\mu(u_\ve,A_\ve)-\nu_{\ve,\Theta}) \wedge d f=\sum_{\omega \subset \partial \GG}\int_{\omega}(\mu_{\ve,\omega}-2\pi\sum_{i\in I_\omega}d_{i,\omega}\de_{a_{i,\omega}})f.
$$
Here, we have used the fact that the restriction of $\nu_{\varepsilon,\Theta}$ to each of the faces $\omega$ of a cube of the grid such that $\omega\subset\partial \GG$ is equal to $2\pi\sum_{i\in I_\omega}d_{i,\omega}\delta_{a_{i,\omega}}$. We then deduce that
\begin{align*}
\left|\int_{\Theta} (\mu(u_\ve,A_\ve)-\nu_{\ve,\Theta}) \wedge \phi\right|&\leq  
\delta \|\mu(u_\ve,A_\ve)-\nu_\ve\|_{C^0(\Theta)^*}\\
&+C_0\max\left(\ve\de^{-1}F_\ve(u_\ve,A_\ve),\ve\right)\left(1+\delta^{-2}F_\ve(u_\ve,A_\ve)\right)\|\phi\|_{C^{0,1}(\Theta)}.
\end{align*}
By combining this with \eqref{decomp} and \eqref{estimateunioncubes}, we find
\begin{align*}
\|\mu(u_\ve,A_\ve)-\nu_\ve\|_{C_T^{0,1}(\Omega)^*}& \leq \delta \|\mu(u_\ve,A_\ve)-\nu_\ve\|_{C^0(\Omega)^*}\\
&+C_0\max\left(\ve\de^{-1}F_\ve(u_\ve,A_\ve),\ve\right)\left(1+\delta^{-2}F_\ve(u_\ve,A_\ve)\right).
\end{align*}
Observe now that
$$
\|\mu(u_\ve,A_\ve)-\nu_\ve\|_{C^0(\Omega)^*}\leq \|\mu(u_\ve,A_\ve)\|_{C^0(\Omega)^*}+\|\nu_\ve\|_{C^0(\Omega)^*}.
$$
From \eqref{LowerBound}, we deduce that
$$
\|\nu_\ve\|_{C^0(\Omega)^*} \leq C_0 \frac{F_\ve(u_\ve,A_\ve)}{|\log \ve|}
$$
%for any $0<\ve<\ve_0$, where $\ve_0$ depends on $M,m,$ and $\partial\Omega$.
By combining the previous two estimates with \eqref{C^0Bound}, we get
\begin{equation}\label{C0}
\|\mu(u_\ve,A_\ve)-\nu_\ve\|_{C^0(\Omega)^*}\leq CF_\ve(u_\ve,A_\ve),
\end{equation}
where $C$ is a constant depending only on $\partial\Omega$. This implies that 
\begin{equation}\label{estgamma1}
\|\mu(u_\ve,A_\ve)-\nu_\ve\|_{C_T^{0,1}(\Omega)^*}\leq  C\delta F_\ve(u_\ve,A_\ve)+C_0\max\left(\ve\de^{-1}F_\ve(u_\ve,A_\ve),\ve\right)\left(1+\delta^{-2}F_\ve(u_\ve,A_\ve)\right).
\end{equation}
From this, \eqref{EstimateJ} for $\gamma=1$ follows.
\end{proof}

The proof of \eqref{EstimateJ} for $\gamma\in(0,1)$ uses the following simple interpolation fact, as in \cite{JerSon}.
\begin{lemma}\label{interpolation}
Assume $\mu$ is a Radon measure on $\Omega$. Then for any $\gamma\in(0,1)$,
$$
\|\mu\|_{C_0^{0,\gamma}(\Omega)^*}\leq\|\mu\|_{C_0^0(\Omega)^*}^{1-\gamma}\|\mu\|_{C_0^{0,1}(\Omega)^*}^\gamma.
$$
\end{lemma}
\begin{proof}[Proof of \eqref{EstimateJ} for $\gamma\in(0,1)$]
Note that $\|\mu\|_{C_0^{0,\gamma}(\Omega)^*}\leq \|\mu\|_{C_T^{0,\gamma}(\Omega)^*}$ for any $1$-current $\mu$. By combining the previous lemma with \eqref{C0} and \eqref{estgamma1}, we are led to
$$
\|\mu(u_\ve,A_\ve)-\nu_\ve\|_{C_0^{0,\gamma}(\Omega)^*}\leq CF_\ve(u_\ve,A_\ve)^{1-\gamma} \de^\gamma (F_\ve(u_\ve,A_\ve)+1)^\gamma\leq C\de^\gamma (F_\ve(u_\ve,A_\ve)+1)
$$
for any $\gamma\in(0,1)$, where $C>0$ is a constant depending only on $\gamma$ and $\partial\Omega$. Then the proof reduces to proving that this estimate is still valid when we replace the norm $\|\cdot \|_{C_0^{0,\gamma}(\Omega)^*}$ with $\|\cdot \|_{C_T^{0,\gamma}(\Omega)^*}$. Arguing as in the proof of \cite{JerMonSte}*{Proposition 3.1}, we conclude that \eqref{EstimateJ} holds for $\gamma\in (0,1)$.
\end{proof}

\subsection{Proof of Theorem \ref{Theorem2}}
\begin{proof}
As in the proof of Theorem \ref{Theorem1}, we consider the grid $\GG(b_\ve,R_0,\de)$ given by Lemma \ref{Lemma:Grid} and the polyhedral $1$-current $\nu_\ve$ defined by \eqref{def:vortapprox}. 

Let us first prove the lower bound. The main difference with the proof of \eqref{LowerBound} is that in this case we cannot use Proposition \ref{propboundary} and therefore we cannot provide a lower bound for the free energy close to the boundary. But, by arguing in the same fashion as before, we immediately check that
\begin{multline}\label{lowerunion}
\int_{S_{\nu_\ve}}|\nabla_{A_\ve}u_\ve|^2+\frac1{2\ve^2}(1-|u_\ve|^2)^2+|\curl A_\ve|^2\geq|\nu_\ve|(\Omega\setminus\Theta)\left(\log \frac1\ve-\log C\frac{M_\ve^{56}|\log \ve|^{7(1+n)}}{\de^{55}}\right)\\
-CM_\ve \de^{\frac23}-C|\log\ve|^{-n},
\end{multline}
where $\de=\de(\ve)=|\log \ve|^{-q}$ with $q>0$. Choosing once again $q=\frac32(m+n)$, and noting that  by the definition of $\Theta$ (recall \eqref{unioncubes}) we have $\Omega_\ve \subset \Omega\setminus \Theta$, we get the lower bound.

\medskip
We now prove the vorticity estimate for $\gamma=1$. In this case, we work in the space $C_0^{0,1}(\Omega)^*$ instead of $C_T^{0,1}(\Omega)^*$. Let $\phi\in C_0^{0,1}(\Omega)$ be a $1$-form. We begin by observing that \eqref{estimateunioncubes} also holds in this case. 

Since $\phi=0$ on $\partial\Omega$, we have that
$$
\|\phi\|_{C^0(\Theta)}\leq C_0 \de\|\phi\|_{C^{0,1}(\Theta)},
$$
and therefore
\begin{equation}\label{est1}
\left|\int_{\Theta} (\mu(u_\ve,A_\ve)-\nu_\ve) \wedge \phi  \right|\leq C_0\de \|\mu(u_\ve,A_\ve)-\nu_\ve \|_{C^0(\Theta)^*} \|\phi\|_{C^{0,1}(\Theta)}.
\end{equation}
On the other hand, from \eqref{lowerunion}, we have
$$
\|\nu_\ve\|_{C^0(\Omega\setminus \Theta)^*} \leq C_0 \frac{F_\ve(u_\ve,A_\ve)}{|\log \ve|}.
$$
From \eqref{numberofpoints}, we deduce that
$$
\|\nu_\ve\|_{C^0(\Theta)^*} \leq C_0 \de \frac{F_\ve(u_\ve,A_\ve)}\de=C_0 F_\ve(u_\ve,A_\ve).
$$
Therefore
$$
\|\nu_\ve\|_{C^0(\Omega)^*} \leq C_0 F_\ve(u_\ve,A_\ve).
$$
By combining this with \eqref{C^0Bound}, we get
$$
\|\mu(u_\ve,A_\ve)-\nu_\ve\|_{C^0(\Omega)^*}\leq CF_\ve(u_\ve,A_\ve).
$$
From this, \eqref{estimateunioncubes}, and \eqref{est1} we obtain the vorticity estimate for $\gamma=1$. The estimate for $\gamma\in (0,1)$ directly follows from Lemma \ref{interpolation}.
\end{proof}

%%%%%%%%%%%%%%%%%%%%%%%%%%%%%%%%%%%%%%%%%%%%%%%%%%%%%%%%%%%%%%%%%%%%%%%%%%%%%%%%%%%%%%%%%%%%%%%%%

\section{Proof of the quantitative product estimate}\label{Sec:ProdEstimate}
In this section, we use ideas from \cite{Ser2017}*{Appendix A}. As in Section \ref{sec:product}, we view things in three dimensions, where the first dimension is time and the last  two are spatial dimensions. 

We consider $X\in C_0^{0,1}([0,T]\times \overline \omega,\R^2)$, a compactly supported spatial vector field  depending on time, and a function $f\in C_0^{0,1}([0,T]\times \overline \omega)$. Let $K$ denote the union of the supports of $X$ and $f$. In order to reduce ourselves to the situation where $X$ and $f$ are locally constants, we use a partition of unity at a small scale: let $M_\ve$ be as in \eqref{condM} and let us consider a covering of $K$ by $m(\ve)$ balls of radius $ M_\ve^{-1/4}$ (with bounded overlap), and let $\{D_k\}_{k=1}^{m(\ve)}$ be an indexation of this sequence of balls and $\{\chi_k\}_{k=1}^{m(\ve)}$ a partition of unity associated to this covering such that $\sum_{k=1}^{m(\ve)} \chi_k= 1$ and $\|\nabla \chi_k\|_{L^\infty} \le M_\ve^{1/4}$ for any $k=1,\dots,m(\ve)$. For each $k\in \{1,\dots,m(\ve)\}$, let then $X_k$ and $f_k$ be the averages of $X$ and $f$ in $D_k$. Then, working only in $D_k$, without loss of generality, we can assume that $X_k$ is aligned with the first space coordinate vector $e_1$, with  $(e_t, e_1, e_2)$ forming an orthonormal frame and the coordinates in that frame  being denoted by $(t,w,\sigma)$. We will assume first that $X_k,f_k\neq 0$. Let us define for each $k,\sigma$ the set
$$
\Theta_{k,\sigma}= \{(t,w)\ |\ (t, w,\sigma) \in  D_k\},
$$  
which is a slice of $D_k$ (hence a two-dimensional ball). 
Let us  write 
$\mu_{\ve, k,\sigma}$ for $\mu_\ve(e_t,e_1)$ restricted to $\Theta_{k,\sigma}$.  In other words, if $\xi $ is a smooth test-function on $\Omega_{\sigma,k}$, we have
\begin{equation}\label{Jeps} 
\int_{\Theta_{k,\sigma}} \mu_{\ve, k, \sigma}\wedge \xi= -\int_{\Theta_{k,\sigma}} \left(  \langle d u_\ve -iu_\ve A_\ve, iu_\ve\rangle +A_\ve\right)\wedge d\xi 
\end{equation} 
where $d$ denotes the  differential in the slice $\Theta_{k,\sigma}$.

For a given $\Lambda >0$, we let $g_k$ be the constant metric on $\Theta_{k,\sigma}$ defined by $g_k(e_t,e_t)=\sqrt{\Lambda}/|f_k|$, $g_k(e_1,e_1)=1/(\sqrt\Lambda|X_k|)$, and $g_k(e_t,e_1)=0$.

We then apply the ball construction method in each set $\Theta_{k,\sigma}$. Instead of constructing balls for the flat metric, we construct geodesic balls for the constant metric $g_k$, i.e. here, ellipses. 

\begin{lemma}\label{ballmetric} Let $\Theta_{k,\sigma}\subset \R^2$ be as above and denote 
$$
\Theta^\ve_{k,\sigma}=\{x\in \Theta_{k,\sigma} \ | \ \mathrm{dist}_{g_k}(x,\partial\Theta_{k,\sigma})>\ve\}.
$$
Assume that
\begin{multline*}
F_{\ve,k, \sigma}\colonequals\frac12 \int_{\Theta_{k,\sigma} } |\partial_t u_\ve-iu_\ve\Phi_\ve|^2 +|\partial_w u_\ve -iu_\ve B_{\ve,w}|^2
\\+ \frac1{2\ve^2} (1-|u_\ve|^2)^2+|\partial_t B_{\ve,w}-\partial_w \Phi_\ve|^2\leq M_\ve
\end{multline*}
with $M_\ve$ as in \eqref{condM}.
Then if $\ve$ is small enough,  there exist a universal constant $C>0$ and a finite collection of disjoint closed  balls $B=\{B_i\}_{i\in I}$ for the metric $g_k$ of centers $a_i$ and radii $r_i$ such that, letting $C_{k,\Lambda}\colonequals\max\left\{\frac{|f_k|}{\Lambda|X_k|},\frac{\Lambda|X_k|}{|f_k|},|f_k||X_k|,\frac1{|f_k||X_k|}\right\}$ and $R_{k,\Lambda}\colonequals \max\left\{\frac{\sqrt{\Lambda}}{|f_k|}\frac1{\sqrt{\Lambda}|X_k|}\right\}$, we have 
\begin{enumerate}[leftmargin=*,font=\normalfont]
\item $r(B)=\sum_i r_i \leq R_{k,\Lambda}|f_k|^2|X_k|^2M_\ve^{-1}$,
\item letting $V=\Theta_{k,\sigma}^\ve\cap \cup_{i\in I}B_i$,
$$
\{||u_\ve(x)|-1|\geq 1/2 \} \cap \Theta_{k,\sigma}^\ve\subset V,
$$
\item writing $d_i= \deg(u_\ve/|u_\ve|, \partial B_i)$ if $B_i \subset \Theta^\ve_{k,\sigma}$ and $d_i=0$ otherwise, we have for each $i$,
\begin{multline}\label{lbmetric}
\frac12\int_{B_i\cap \Theta_{k,\sigma}^\ve} \frac{|f_k|^2}{\Lambda} |\partial_t u_\ve-iu_\ve \Phi_\ve|^2+\Lambda |X_k|^2 |\partial_w  u_\ve -iu_\ve B_{\ve,w} |^2\\
+ |f_k|^5|X_k|^5M_\ve^{-2}|\partial_t B_{\ve,w}-\partial_w \Phi_\ve|^2\\ 
\geq \pi|d_i| (|\log \ve|- C\log M_\ve +\log(C_{k,\Lambda}^{-1}|f_k|^2|X_k|^2))|f_k||X_k|,
\end{multline}
\item and letting $\mu_{\ve,k,\sigma}= 2\pi \sum_i  d_i \delta_{a_i}$, we have for any $\xi \in C^{0, 1}_0(\Theta_{k,\sigma })$, 
$$
\left|\int_{\Theta_{k,\sigma}} (J_{\ve,k,\sigma}-\mu_{\ve,k,\sigma})\wedge \xi\right| \leq C \|\xi\|_{C^{0,1}}C_{k,\Lambda}|f_k|^2||X_k|^2M_\ve^{-1}F_{\ve,k,\sigma}.
$$
\end{enumerate}
\end{lemma}

\begin{proof}
Let us consider the sets
$$
\tilde\Theta_{k,\sigma}=\left\{(s,z) \ \vline  \ \left(\frac{|f_k|}{\sqrt{\Lambda}}s,\sqrt{\Lambda}|X_k|z\right)\in  \Theta_{k,\sigma}\right\},\quad 
\tilde\Theta_{k,\sigma}^\ve=\{x\in \tilde\Theta_{k,\sigma} \ | \ d(x,\partial \tilde \Theta_{k,\sigma})>\ve\}
$$
and define, for $(s,z)\in \tilde \Theta_{k,\sigma }$, the function
$$
\tilde u_\ve(s,z)=u_\ve\left(\frac{|f_k|}{\sqrt \Lambda}s,\sqrt{\Lambda}|X_k|z,\sigma\right)
$$
and the vector field
$$
\tilde A_\ve(s,z)=[\tilde A_{\ve,s},\tilde A_{\ve,z}](s,z)=\left[\sqrt{\Lambda}|X_k|\Phi_\ve,\frac{|f_k|}{\sqrt\Lambda}B_{\ve,w}\right]\left(\frac{|f_k|}{\sqrt\Lambda}s,\sqrt{\Lambda}|X_k|z,\sigma\right)
$$
The first three items are a consequence of \cite{SanSerBook}*{Proposition 4.3}. We start by noting that
by making the change of variables $t=\frac{\sqrt \Lambda}{|f_k|}s$ and $w=\frac1{\sqrt{\Lambda}|X_k|}z$, we obtain
$$
\frac12 \int_{\tilde \Theta_{k,\sigma} } |\nabla_{\tilde A_\ve} \tilde u_\ve|^2 + \frac1{2\ve^2} (1-|\tilde u_\ve|^2)^2+|\partial_s \tilde  A_{\ve,z}-\partial_z \tilde A_{\ve,s}|^2\leq C_{k,\Lambda} F_{\ve,k,\sigma}.
$$
The co-area formula provides the existence of $m_\ve$ with $M_\ve^{-1}\leq m_\ve \leq 2M_\ve^{-1}$ such that setting $W\colonequals \{(s,z) \ | \ |\tilde u_\ve(s,z)|\leq 1-m_\ve\}$ has perimeter (for the Euclidean metric) bounded by $C C_{k,\Lambda} \ve M_\ve^2$. We may then apply this proposition to the configuration $(\tilde u_\ve,\tilde A_\ve)$, with initial radius $r_0=CC_{k,\Lambda}\ve M_\ve^2$ and final radius $r_1=|f_k|^2|X_k|^2 M_\ve^{-1}$ (provided $\ve\ll 1$ such that $r_1\leq 1$). This yields a collection of disjoint closed balls $\tilde B=\{\tilde B_i\}_{i\in I}$ such that
\begin{itemize}[leftmargin=*]
\item $r(\tilde B)=|f_k|^2|X_k|^2M_\ve^{-1}$,
\item $\tilde V=\tilde \Theta_{k,\sigma}^\ve \cap \cup_i \tilde B_i$ covers $\{(s,z)\ | \ ||\tilde u_\ve(s,z)|-1|\geq \frac12\}\cap \tilde \Theta_{k,\sigma}^\ve$,
\item and writing $d_i= \deg(v_\ve/|v_\ve|, \partial \tilde B_i)$ if $\tilde B_i \subset \tilde \Theta^\ve_{k,\sigma}$ and $d_i=0$ otherwise, we have for each i,
$$
\frac12\int_{(\tilde B_i \cap \Theta_{k,\sigma}^\ve)\setminus W} |\nabla_{\tilde A_\ve} (\tilde u_\ve/|\tilde u_\ve|) |^2 +r_1(r_1-r_0) |\partial_s A_{\ve,z}-\partial_z A_{\ve,s}|^2   \geq \pi |d_i| \left(\log \frac{r_1}{r_0}-C\right).
$$
\end{itemize}
But, by \cite{SanSerBook}*{Lemma 3.4}, we have
\begin{align*}
\int_{(\tilde B_i \cap \Theta_{k,\sigma}^\ve)\setminus W} |\nabla_{\tilde A_\ve} \tilde u_\ve |^2&\geq 
\int_{(\tilde B_i \cap \Theta_{k,\sigma}^\ve)\setminus W}|\tilde u_\ve|^2 |\nabla_{\tilde A_\ve} (\tilde u_\ve/|\tilde u_\ve|) |^2+ |\nabla |\tilde u_\ve||^2\\ 
&\geq (1-m_\ve)^2\int_{(\tilde B_i \cap \Theta_{k,\sigma}^\ve)\setminus W} |\nabla_{\tilde A_\ve} (\tilde u_\ve/|\tilde u_\ve|) |^2.
\end{align*}
Thus 
\begin{align*}
\frac12\int_{(\tilde B_i \cap \Theta_{k,\sigma}^\ve)\setminus W} |\nabla_{\tilde A_\ve} \tilde u_\ve |^2 +r_1(r_1-r_0) &|\partial_s \tilde A_{\ve,z}-\partial_z \tilde A_{\ve,s}|^2 \\
&\geq(1-m_\ve)^2\pi |d_i| \left(\log \frac{r_1}{r_0}-C\right)\\
&\geq \pi |d_i|\left(|\log \ve|-C\log M_\ve+\log(C_{k,\Lambda}^{-1}|f_k|^2|X_k|^2)\right).
\end{align*}
In particular, we deduce that 
\begin{equation}\label{UpperD}
D=\sum_{i\in I} |d_i|\leq \frac{C_{k,\Lambda}F_{\ve,k,\sigma}}{|\log \ve|}.
\end{equation}

Then, by changing variables once again, we obtain balls $B_i$, the images of the $\tilde B_i$'s by the change of variable, which are geodesic balls for the metric $g_k$ and whose sum of radii is bounded by $\max\left\{\frac{\sqrt{\Lambda}}{|f_k|},\frac{1}{\sqrt{\Lambda}|X_k|}\right\} r(\tilde B)\leq R_{k,\Lambda}|f_k|^2|X_k|^2M_\ve^{-1}$. Items (2) and (3) immediately follow from the change of variables and the properties satisfied by $\tilde u_\ve$ and $\tilde A_\ve$.
	
Item (4) follows from \cite{SanSerBook}*{Theorem 6.1}. Indeed, denoting $\tilde a_i$ the center of $\tilde B_i$ and letting $\tilde \mu_{\ve,k,\sigma}=2\pi\sum_i d_i\delta_{\tilde a_i}$, this theorem yields, for any $\tilde \xi\in C_0^{0,1}(\tilde \Theta_{k,\sigma})$, that
$$
\left|\int_{\tilde\Theta_{k,\sigma} }(\mu(\tilde u_\ve,\tilde A_\ve)-\tilde\mu_{\ve,k,\sigma})\wedge \tilde \xi\right| \leq C \|\xi\|_{C^{0,1}}r(\tilde B)C_{k,\Lambda} F_{\ve,k,\sigma},
$$
where $C$ is a universal constant. But, by change of variables, we have
$$
\int_{\tilde\Theta_{k,\sigma} }\left(\mu(\tilde u_\ve,\tilde A_\ve)-\tilde \mu_{\ve,k,\sigma}\right)\wedge \tilde\xi=
\int_{\Theta_{k,\sigma} }\left(J_{\ve,k,\sigma}-2\pi \sum_{i\in I} d_i \delta_{a_i}\right)\wedge \xi,
$$
where $a_i$ is the center of the ball $B_i$, i.e. the image by the change of variables of $\tilde a_i$, and $\xi(t,w)=\tilde \xi\left(\frac{|f_k|}{\sqrt{\Lambda}}t,\sqrt{\Lambda}|X_k|w\right)$. This concludes the proof. 
\end{proof}
Throughout the rest of this section $C>0$ denotes a universal constant that may change from line to line. 

\begin{proof}[Proof of Theorem \ref{thm:prodesti}]
We proceed similarly to \cites{SanSer3,Ser2017}.  We set $\vartheta_{\ve,k,\sigma}$ to be the $\mu_{\ve, k ,\sigma}$ of  Lemma \ref{ballmetric} (item 4) if the assumption $F_{\ve,k,\sigma}\leq M_\ve$ is verified, and $0$ if not.
We note that 
\begin{equation}\label{difVorApp}
\|J_{\ve,k,\sigma}- \vartheta_{\ve,k,\sigma}\|_{C^{0,1}_0(\Theta_{k,\sigma})^*}\leq C( C_{k,\Lambda}|f_k|^2|X_k|^2 +1)M_\ve^{-1/2}F_{\ve,k,\sigma}
\end{equation} 
is true in all cases. Indeed, either $F_{\ve, k,\sigma} \leq M_\ve$ in which case the result is true by item 4 of Lemma \ref{ballmetric} since $M_\ve^{-1/2} \geq M_\ve^{-1}$, or $\vartheta_{\ve, k,\sigma}=0$ in which case, for any $\xi \in C^{0,1}_0(\Theta_{k,\sigma})$, starting from \eqref{Jeps} and writing $|\langle \nabla u_\ve-iA_\ve u_\ve, iu_\ve\rangle|\leq |\nabla u_\ve-iA_\ve u_\ve|+|1-|u_\ve|^2||\nabla u_\ve -iA_\ve u_\ve|$ (note that $|1-|u_\ve||\leq |1-|u_\ve|^2|$), we obtain   with the Cauchy-Schwarz inequality,  using the boundedness of $\Theta_{k,\sigma}$, 
\begin{align*}
\left|\int_{\Theta_{k,\sigma}} J_{\ve,k, \sigma}\wedge \xi \right|&\leq 
\|\nabla \xi\|_{L^\infty}\int_{\Theta_{k,\sigma}}(|\partial_t u_\ve-iu_\ve \Phi_\ve|+|\partial_w u_\ve-i u_\ve B_{\ve,w}|)(1+|1-|u_\ve|^2|)\\
& \ \ \ +\|\xi\|_{L^\infty}\int_{\Theta_{k,\sigma}}|\partial_t B_{\ve,w}-\partial_w \Phi_\ve|\\
&\leq C \|\xi\|_{C^{0,1}}  (\sqrt{F_{\ve,k, \sigma}}+ \ve F_{\ve, k,\sigma}).
\end{align*} 
But since $F_{\ve,k,\sigma} \geq M_\ve$, we have $\sqrt{F_{\ve,k, \sigma}} +\ve F_{\ve, k,\sigma}  \leq 2 M_\ve ^{-1/2} F_{\ve,k, \sigma} $ and thus we find that \eqref{difVorApp} holds as well. 
%By \eqref{UpperD}, we also have that 
%\begin{equation}\label{bornnu}
%\int_{\Theta_{k,\sigma}} |\nu_{\ve, k,\sigma}| \leq \frac{C}{|\log\ve|}C_{k,\Lambda} F_{\ve, k,\sigma} .
%\end{equation}

We may now write that 
\begin{align*}
\int_{\Theta_{k,\sigma}}\chi_k &\Big(\frac{|f_k|^2}{\Lambda} |\partial_t u_\ve-iu_\ve \Phi_\ve|^2+\Lambda |X_k|^2 |\partial_w  u_\ve -iu_\ve B_{\ve,w} |^2\\
&+|f_k|^5|X_k|^5M_\ve^{-2}|\partial_t B_{\ve,w}-\partial_w \Phi_\ve|^2\Big)   \\ 
&\hspace*{0.5cm}\geq ( |\log \ve|  - C\log M_\ve+\log(C_{k,\Lambda}^{-1}|f_k|^2|X_k|^2))\left|\int_{\Theta_{k,\sigma}}|f_k||X_k|\chi_k \vartheta_{\ve,k,\sigma} \right|\\
&\hspace*{1cm}-C C_{k,\Lambda}R_{k,\Lambda}^2|f_k|^3|X_k|^3M_\ve^{-3/4}F_{\ve,k,\sigma} .
\end{align*}
Indeed, if we are in a slice where $\vartheta_{\ve,k,\sigma}=0$, this is trivially true. If not, we apply \eqref{lbmetric} and obtain
\begin{align*}
\int_{{(\cup_i B_i)\cap \Theta_{k,\sigma}^\ve}}\chi_k &\Big(\frac{|f_k|^2}{\Lambda} |\partial_t u_\ve-iu_\ve \Phi_\ve|^2+\Lambda |X_k|^2 |\partial_w  u_\ve -iu_\ve B_{\ve,w} |^2\\
&+|f_k|^5|X_k|^5M_\ve^{-2} |\partial_t B_{\ve,w}-\partial_w \Phi_\ve|^2\Big)   \\ 
&\hspace*{0.5cm}\geq 2\pi\sum_{i\in I} |d_i|\min_{B_i}\chi_k ( |\log \ve|  - C\log M_\ve+\log(C_{k,\Lambda}^{-1}|f_k|^2|X_k|^2))|f_k||X_k|.
\end{align*}
Besides, we have $\min_{B_i}\chi_k\geq \chi_k(a_i)-C\max\left\{\frac{\sqrt \Lambda}{|f_k|},\frac1{\sqrt \Lambda |X_k|}\right\}r_i\|\chi_k\|_{C^{0,1}}$. Plugging in, using item 1 of Lemma \ref{ballmetric} and 
\eqref{UpperD}, yields the desired inequality.

Combining with \eqref{difVorApp}, we find
\begin{align*}
\int_{\Theta_{k,\sigma}}\chi_k &\Big(\frac{|f_k|^2}{\Lambda} |\partial_t u_\ve-iu_\ve \Phi_\ve|^2+\Lambda |X_k|^2 |\partial_w  u_\ve -iu_\ve B_{\ve,w} |^2\\
&+|f_k|^5|X_k|^5M_\ve^{-2}|\partial_t B_{\ve,w}-\partial_w \Phi_\ve|^2\Big)   \\ 
&\hspace*{0.5cm}\geq ( |\log \ve|  - C\log M_\ve +\log(C_{k,\Lambda}^{-1}|f_k|^2|X_k|^2))\left|\int_{\Theta_{k,\sigma}}|f_k||X_k|\chi_k J_{\ve,k,\sigma} \right|\\
&\hspace*{1cm} -C \tilde C_{k,\Lambda} \tilde R^2_{k,\Lambda}M_\ve^{-3/4}F_{\ve,k,\sigma}
-C(\tilde C_{k,\Lambda}+1)|f_k||X_k||\log \ve|M_\ve^{-1/2}F_{\ve,k,\sigma},
\end{align*}
where $\tilde C_{k,\Lambda}\colonequals C_{k,\Lambda}|f_k||X_k|$ and $\tilde R_{k,\Lambda}\colonequals R_{k,\Lambda}|f_k||X_k|$.

Let us observe that 
$$
\log(C_{k,\Lambda}^{-1}|f_k|^2|X_k|^2)\geq \log \min(\Lambda,\Lambda^{-1})+\log \min(C_{k,1}^{-1}|f_k|^2|X_k|^2,1).
$$ 
Moreover (see \eqref{veloc})
\begin{align*}
\log \min(C_{k,1}^{-1}|f_k|^2|X_k|^2,1)&\left|\int_{\Theta_{k,\sigma}}|f_k||X_k|\chi_k J_{\ve,k,\sigma} \right|\\
&\geq -\left|\log \min\{C_{k,1}^{-1}|f_k|^2|X_k|^2,1\}\right||f_k||X_k| \int_{\Theta_{k,\sigma}} \chi_k |V_\ve|\\
&\geq -C\max_{s\in(0,1)}s|\log s|  \int_{\Theta_{k,\sigma}} \chi_k |V_\ve|\\
&=-C\int_{\Theta_{k,\sigma}} \chi_k |V_\ve|.
\end{align*}

Notice that, since we assumed that $X_k$ is along the direction $e_1$, we have
\begin{equation}\label{veloc}
J_{\ve,k,\sigma}=J(e_t,e_1)=V_\ve\cdot e_2=V_\ve \cdot \frac{X_k^\perp}{|X_k|}
\end{equation}
so we may bound
$$
\left|\int_{\Theta_{k,\sigma}} |f_k||X_k|\chi_k J_{\ve,k,\sigma}\right|\geq \left| \int_{\Theta_{k,\sigma}}\chi_k f_k V_\ve\cdot X_k^\perp \right|.
$$
We also observe that 
$$
\int_{\Theta_{k,\sigma}}\chi_k |X_k|^2|\partial_w u_\ve - iu_\ve B_{\ve,w}|^2= \int_{\Theta_{k,\sigma}}\chi_k |X_k \cdot (\nabla u_\ve-iu_\ve B_\ve)|^2.
$$
Plugging in and integrating with respect to $\sigma$, yields 
\begin{align*}
\int_{D_k}\chi_k &\Big(\frac{|f_k|^2}{\Lambda} |\partial_t u_\ve-iu_\ve \Phi_\ve|^2+\Lambda |X_k\cdot (\nabla u_\ve -iu_\ve B_\ve)|^2\\
&+|f_k|^5|X_k|^5M_\ve^{-2}|\partial_t B_{\ve,w}-\partial_w \Phi_\ve|^2\Big)\\
&\hspace*{0.5cm}\geq ( |\log \ve|  - C\log M_\ve-\log \max(\Lambda,\Lambda^{-1})) \left|  \int_{D_k} \chi_k f_k V_\ve \cdot X_k^\perp \right|-C\int_{D_k}\chi_k |V_\ve|\\
&\hspace*{1cm} -C \tilde C_{k,\Lambda} \tilde R^2_{k,\Lambda}M_\ve^{-3/4}F_{\ve,k}
-C(\tilde C_{k,\Lambda}+1)|f_k||X_k||\log \ve|M_\ve^{-1/2}F_{\ve,k},
\end{align*}
where $F_{\ve,k}\colonequals \int F_{\ve,k,\sigma}d\sigma$. It is important to note that since $\tilde C_{k,\Lambda},\tilde R_{k,\Lambda}\leq C$ for some constant $C$ depending only on $\Lambda$, $\|f\|_\infty$, and $\|X\|_\infty$, this holds as well if $f_k=0$ or $X_k=0$.

We may next replace $X_k$ by $X$ and $f_k$ by $f$ in the first two terms of the left-hand side and the $\int_{D_k} \chi_k f_k V_\ve \cdot X_k^\perp$ term, and using that $|X-X_k|\leq  C M_\ve^{-1/4} \|X\|_{C^{0,1}}$ and $|f-f_k|\leq  C M_\ve^{-1/4} \|f\|_{C^{0,1}}$, the error thus created is bounded above by
$$ 
CC(\Lambda)C_0(f,X)C_1(f,X) M_\ve^{-1/4}|\log \ve|F_\ve(D_k), 
$$
where $C_0(f,X)\colonequals\max\{\|f\|_{\infty},\|X\|_{\infty},1\}$, $C_1(f,X)\colonequals \max\{\|f\|_{C^{0,1}},\|X\|_{C^{0,1}},1\}$, $C(\Lambda)=\max\{\frac{1}{\Lambda},\Lambda\}$, and 
$$
F_\ve(D_k)\colonequals \frac12 \int_{D_k}|\partial_t u_\ve - iu_\ve \Phi_\ve|^2+ |\nabla u_\ve -iu_\ve B_\ve|^2 +\frac1{2\ve^2}(1-|u_\ve|^2)^2+|\curl A_\ve|^2.
$$
Here we have used that by definition of $V_\ve$, we have $\int_{D_k} |V_\ve|\leq C F_{\ve, k}(D_k)$. Let us note that
$$
|f_k|^5|X_k|^5M_\ve^{-2}|\partial_t B_{\ve,w}-\partial_w \Phi_\ve|^2\leq C_0(f,X)^{10} M_\ve^{-2} \int_{D_k} |\curl A_\ve|^2 \leq 2C_0(f,X)^{10}M_\ve^{-2}F_\ve(D_k),
$$
$\tilde C_{k,\Lambda}\leq C(\Lambda) C_0(f,X)^6$, and $\tilde R^2_{k,\Lambda}\leq C(\Lambda) C_0(f,X)^6$.
Summing over $k$, using that $\sum_k \chi_k=1$ in $K$ (the union of the supports of $f$ and $X$) and the finite overlap of the covering, we are led to
\begin{multline*}
\int_{(0,T)\times \omega} \frac{|f|^2}{\Lambda} |\partial_t u_\ve-iu_\ve \Phi_\ve|^2+\Lambda |X\cdot (\nabla u_\ve -iu_\ve B_\ve)|^2\\
\geq \left( |\log \ve|  - C\log M_\ve-\log\max(\Lambda,\Lambda^{-1}) \right) \left|  \int_{(0,T)\times \omega} fV_\ve \cdot X^\perp \right|\\
-C\int_{(0,T)\times \omega}\sum_k \chi_k|V_\ve|-\mathrm{Error}(\ve,\Lambda,f,X),
\end{multline*}
where
\begin{align*}
\mathrm{Error}(\ve,\Lambda,f,X)= CF_\ve(u_\ve,A_\ve)&\Big(C(\Lambda) C_0(f,X)C_1(f,X)|\log \ve|M_\ve^{-1/4}+C_0(f,X)^{10} M_\ve^{-2}\\
&+C^2(\Lambda) C_0(f,X)^{12}M_\ve^{-3/4}+C(\Lambda)C_0(f,X)^8|\log \ve|M_\ve^{-1/2}\Big).
\end{align*}
We may now use Theorem \ref{Theorem2} (notice that $X$, $f$, and $\chi_k$ are compactly supported in $\omega\times (0,T)$) to replace the velocity by its polyhedral approximation. Since we assume that $F_\ve(u_\ve,A_\ve)\leq M|\log \ve|^m$, Theorem \ref{Theorem2} with $n>\frac13(2-m)$ implies that 
$$
\int_{(0,T)\times \omega} fV_\ve\cdot X^\perp =\int_{(0,T)\times \omega} f \nu_\ve \wedge (-X_2dx_1+X_1dx_2)+O\left(|\log \ve|^{-\frac12(m+3n)}\right)
$$
and
$$
\int_{(0,T)\times \omega}\sum_k \chi_k |V_\ve|\leq C\int_{(0,T)\times \omega}\max(|\nu_\ve\wedge dx_1|,|\nu_\ve\wedge dx_2|)+O\left(|\log \ve|^{-\frac12(m+3n)}\right)
$$
which gives
\begin{multline*}
\int_{(0,T)\times \omega} \frac{|f|^2}{\Lambda} |\partial_t u_\ve-iu_\ve \Phi_\ve|^2+\Lambda |X\cdot (\nabla u_\ve -iu_\ve B_\ve)|^2\\
\geq \left( |\log \ve|  - C\log M_\ve-\log \max(\Lambda,\Lambda^{-1})\right) \left|  \int_{(0,T)\times \omega} f\nu_\ve \wedge (-X_2dx_1+X_1dx_2) \right|\\
- C\int_{(0,T)\times \omega}\max(|\nu_\ve\wedge dx_1|,|\nu_\ve\wedge dx_2|) +O\left(|\log \ve|^{-\frac12(m+3n)+1}\right)-\mathrm{Error}(\ve,\Lambda,f,X).
\end{multline*}
We observe that $-\frac12(m+3n)+1<0$ and therefore the error associated to our approximation goes to zero as $\ve\to 0$. In addition, for $\ve$ small enough depending on $\Lambda,f,$ and $X$, we have 
$$
\max(\Lambda,\Lambda^{-1})\leq M_\ve \quad\mathrm{and}\quad \mathrm{Error}(\ve,\Lambda,f,X)=O(|\log \ve|^{-\frac12(m+3n)+1}).
$$
This concludes the proof.
\end{proof}
%%%%%%%%%%%%%%%%%%%%%%%%%%%%%%%%%%%%%%%%%%%%%%%%%%%%%%%%%%%%%%%%%%%%%%%%%%%%%%%%%%%%%%%%%%%%%%%%%

\appendix
\section{Smooth approximation of the function \texorpdfstring{$\zeta$}{zeta}}\label{Sec:AppendixA}
In this section we present the proof of Proposition \ref{prop:functionzeta}. We will regularize the function $\zeta$ by convolution. Having into account that we need to provide a quantitative estimate of the second fundamental form of the approximation, we will first displace a bit the points of the collection. The reason for this is that the displacement will ensure that the vectors $\nu_{(i,j)}$ (see Definition \ref{def:funczeta}) will satisfy ``good'' angle conditions between each other. This will permit us to characterize the set where the gradient of the convolution is small, which will translate into a control on its second fundamental form. 

\subsection{Displacement of the points}
Next, we prove some basic geometric properties which allow us to perform the displacement of the points (see Proposition \ref{displacement}).
\begin{lemma}\label{dipole}
Consider four points $x_1,x_2,x_3,x_4\in \R^3$  with $x_1\neq x_2$ and a number 
$$
D\geq \max_{1\leq i<j\leq 4}|x_i-x_j|.
$$ 
There exists $\vartheta_1>0$ independent of $D$ such that
for any $0<\vartheta<\vartheta_1$ there exists a point $x_4'\in\R^3$ such that
\begin{equation}\label{dipole1}
\left|\frac{x_2-x_1}{|x_2-x_1|}\times \frac{x_4'-x_3}{|x_4'-x_3|} \right|\geq \vartheta
\end{equation}
and 
\begin{equation}\label{dipole2}
|x_4-x_4'|\leq 3D\vartheta,\quad \max_{1\leq i \leq 3}|x_i-x_4'|\leq D.
\end{equation}
\end{lemma}
By translating the points it is enough to prove the following lemma.
\begin{lemma} Consider three points $x_1,x_2,x_3\in \R^3$ with $x_1\neq x_2$ and a number
$$
D\geq\max_{1\leq i<j\leq 3}|x_i-x_j|.
$$ 
There exists $\vartheta_2>0$ independent of $D$ such that for any $0<\vartheta<\vartheta_2$ there exists a point $x_3'\in \R^3$ such that 
$$
\left|\frac{x_2-x_1}{|x_2-x_1|}\times \frac{x_3'-x_1}{|x_3'-x_1|} \right|, \ \left|\frac{x_2-x_1}{|x_2-x_1|}\times \frac{x_3'-x_2}{|x_3'-x_2|}\right| \geq \vartheta
$$
and 
$$
|x_3-x_3'|\leq 3D\vartheta,\quad \max_{1\leq i \leq 2}|x_i-x_3'|\leq D.
$$
\end{lemma}
\begin{proof}
Let us consider the cylinder whose axis is $\{tx_1+(1-t)x_2 \ | \ t\in [-D,D]\}$ and whose radius is $r=\frac{\vartheta \sqrt 2 D}{\sqrt{ (1-\vartheta^2)}}$. Note that $r<2D\vartheta$ for $\vartheta\leq 1-2^{-1/2}$. 
Remembering that $|u\times v|=|u||v||\sin\theta|$, where $\theta$ is the angle formed by $u$ and $v$, it is easy to check that the cylinder previously defined contains all the points $y\in \R^3$ with $|x_1-y|,|x_2-y|\leq D$ such that 
$$
\left|\frac{x_2-x_1}{|x_2-x_1|}\times \frac{y-x_1}{|y-x_1|} \right|< \vartheta \quad \mbox{or}\quad \left|\frac{x_2-x_1}{|x_2-x_1|}\times \frac{y-x_2}{|y-x_2|}\right| <\vartheta.
$$
Then simple trigonometric manipulations show that there exists a point $x_3'\in\R^3$ such that 
$$
|x_3-x_3'|\leq r+\frac{r^2}D\leq 3D\vartheta,\quad \max_{1\leq i\leq 2}|x_i-x_3'|\leq D,
$$
and
$$
\left|\frac{x_2-x_1}{|x_2-x_1|}\times \frac{x_3'-x_1}{|x_3'-x_1|} \right|, \ \left|\frac{x_2-x_1}{|x_2-x_1|}\times \frac{x_3'-x_2}{|x_3'-x_2|}\right| \geq \vartheta
$$
for any $\vartheta$ small enough. 
\end{proof}

\begin{lemma}\label{tripole} 
Consider six points $x_1,x_2,x_3,x_4,x_5,x_6\in \R^3$ and a number 
$$
D\geq \max_{1\leq i<j\leq 6}|x_i-x_j|.$$ 
There exists $\vartheta_3>0$ independent of $D$ such that for any $0<\vartheta<\vartheta_3$, if 
$$
\left|\frac{x_2-x_1}{|x_2-x_1|}\times \frac{x_4-x_3}{|x_4-x_3|} \right|\geq \vartheta
$$
then there exists a point $x_6'\in \R^3$ such that 
\begin{equation}\label{tripole1}
\left|\det\left(\frac{x_2-x_1}{|x_2-x_1|},\frac{x_4-x_3}{|x_4-x_3|}, \frac{x_6'-x_5}{|x_6'-x_5|}\right) \right|\geq \vartheta^2
\end{equation}
and 
\begin{equation}\label{tripole2}
|x_6-x_6'|\leq 3D\vartheta,\quad \max_{1\leq i \leq 5}|x_i-x_6'|\leq D .
\end{equation}
\end{lemma}
By translating the points it is enough to prove the following lemma.

\begin{lemma}
Consider four points $x_1,x_2,x_3,x_4\in \R^3$ and a number 
$$
D\geq \max_{1\leq i<j\leq 4}|x_i-x_j|.
$$
There exists $\vartheta_4>0$ independent of $D$ such that for any $0<\vartheta<\vartheta_4$, if 
$$
\left|\frac{x_2-x_1}{|x_2-x_1|}\times \frac{x_3-x_1}{|x_3-x_1|} \right|\geq \vartheta
$$
then there exists a point $x_4'\in \R^3$ such that 
$$
\left|\det\left(\frac{x_2-x_1}{|x_2-x_1|},\frac{x_3-x_1}{|x_3-x_1|}, \frac{x_4'-x_1}{|x_4'-x_1|}\right) \right|\geq \vartheta^2
$$
and 
$$
|x_4-x_4'|\leq 3D\vartheta,\quad \max_{1\leq i \leq 3}|x_i-x_4'|\leq D.
$$
\end{lemma}
\begin{proof}Let $P$ denote the plane where the points $x_1,x_2,x_3$ are contained.
Given any point $y\in \R^3\setminus P$, observe that
$$
\left| \det\left(\frac{x_2-x_1}{|x_2-x_1|},\frac{x_3-x_1}{|x_3-x_1|}, \frac{y-x_1}{|y-x_1|}\right) \right|=h\cdot \left|\frac{x_2-x_1}{|x_2-x_1|}\times \frac{x_3-x_1}{|x_3-x_1|} \right|\geq h\vartheta,
$$
where $h$ denotes the height of the parallelepiped formed by the vectors $$
\frac{x_2-x_1}{|x_2-x_1|},\frac{x_3-x_1}{|x_3-x_1|}
,\frac{y-x_1}{|y-x_1|}.
$$
It is easy to check that 
$$
h=\frac{d(y,P)}{|y-x_1|}.
$$
Choosing $h=\vartheta$ we are left with finding a point $x_4'$ such that $\vartheta|x_4'-x_1|=d(y,P)$. Then simple trigonometric manipulations show that there exists a point $x_4'$ such that
$$
|x_4-x_4'|\leq 3D\vartheta,\quad \max_{1\leq \leq 3}|x_i-x_4'|\leq D,\quad \mathrm{and}\quad \vartheta|x_4'-x_1|=d(x_4',P),
$$
for any $\vartheta$ small enough.
\end{proof}

The previous lemmas allow us to prove the following result.
\begin{proposition}\label{displacement}
Let $\AA=\{a_1,\dots,a_m\}$ be a collection of $m$ non necessarily distinct points. Define $D_\AA\colonequals\max_{1\leq i <j\leq m}|a_i-a_j|$ to be the maximum Euclidean distance between any of the points of $\AA$ and assume that $D_\AA>0$. Then there exists a collection of points $\AA'=\{b_1,\dots,b_m\}$ such that for any $\vartheta<\min\{m^{-6},\vartheta_1,\vartheta_3\}$, where the numbers $\vartheta_1,\vartheta_3$ are the constants appearing in Lemma \ref{dipole} and Lemma \ref{tripole} respectively, the following hold
\begin{enumerate}[font=\normalfont,leftmargin=*]
\item $b_i\neq b_j$ for any $i\neq j$.
\item Define
$$
\nu_{(i,j)}\colonequals \frac{b_i-b_j}{|b_i-b_j|}\quad\mathrm{for}\ (i,j)\in \Lambda_m\colonequals\{(p,q) \ |\ 1\leq p<q\leq m\}.
$$
Then for any $\alpha,\beta,\gamma\in \Lambda_m$ with $\alpha\neq \beta\neq \gamma$, we have
$$
|\nu_\alpha\times \nu_\beta|\geq \vartheta \quad \mathrm{and}\quad |\mathrm{det}(\nu_\alpha,\nu_\beta,\nu_\gamma)|\geq \vartheta^2.
$$
\item $|a_l-b_l|\leq CD_\AA l^5\vartheta$ for any $l\in \{1,\dots,m\}$, where $C$ is a universal constant.
\item $\max_{1\leq i< j\leq m}|b_i-b_j|\leq CD_\AA$, where $C$ is a universal constant.
\end{enumerate}
\end{proposition}
\begin{proof} 
We proceed by induction. Without loss of generality we may assume that $a_1\neq a_2$. We define $b_1=a_1$, $b_2=a_2$, $d_1=d_2=0$ and $D_1= D_2=D_\AA$.

Assume that we have defined a collection $\{b_1,\dots,b_l\}$ with $2<l<m$ such that
\begin{itemize}
\item For any $\alpha,\beta,\gamma\in \Lambda_l$ with $\alpha\neq \beta\neq \gamma$, we have
$$
|\nu_\alpha\times \nu_\beta|\geq \vartheta \quad \mathrm{and}\quad |\mathrm{det}(\nu_\alpha,\nu_\beta,\nu_\gamma)|\geq \vartheta^2.
$$
\item $|a_i-b_i|\leq d_l$ for any $i\in \{1,\dots,l\}$.
\item $|b_i-b_j|\leq D_l$ for any $i,j\in \{1,\dots,l\}$.
\end{itemize}
Observe that by applying Lemma \ref{dipole} with the points $x_1=b_1$, $x_2=b_2$, $x_3=b_1$, $x_4=a_{l+1}$, and the number $D=D_l+d_l$, we find a point $x_4'$ satisfying \eqref{dipole1} and \eqref{dipole2}. By repeating this argument at most $l^3$ times, we find a point $b_{l+1}'$ such that the collection $\{b_1,\dots,b_l,b_{l+1}'\}$ satisfies
$$
|\nu_\alpha\times \nu_\beta|\geq \vartheta
$$ 
for any $\alpha,\beta\in \Lambda_{l+1}$ with $\alpha\neq \beta$. Moreover
$$
|a_{l+1}-b_{l+1}'|\leq 3l^3 (D_l+d_l)\vartheta\quad \mathrm{and}\quad |b_i-b_{l+1}'|\leq D_l+d_l\ \mathrm{for\ any}\ i\in \{1,\dots,l\}.
$$ 
We further displace the point $b_{l+1}'$ in order to additionally satisfy the condition on the determinants. Applying Lemma \ref{tripole} with $x_1=b_1$, $x_2=b_2$, $x_3=b_1$, $x_4=b_3$, $x_5=b_1$, $x_6=b_{l+1}'$ and $D=D_l+d_l$, we find a point $x_6'$ satisfying \eqref{tripole1} and \eqref{tripole2}. By repeating this argument at most $l^5$ times, we find a point $b_{l+1}$ such that the collection $\{b_1,\dots,b_l,b_{l+1}\}$ satisfies
$$
|\nu_\alpha\times \nu_\beta|\geq \vartheta
\quad\mathrm{and}\quad |\mathrm{det}(\nu_\alpha,\nu_\beta,\nu_\gamma)|\geq \vartheta^2
$$
for any $\alpha,\beta,\gamma\in \Lambda_{l+1}$ with $\alpha\neq\beta\neq\gamma$. Moreover
$$
|a_{l+1}-b_{l+1}|\leq 3(l^3+l^5)(D_l+d_l)\vartheta\quad \mathrm{and}\quad 
|b_i-b_{l+1}|\leq D_l+d_l\ \mathrm{for\ any}\ i\in \{1,\dots,l\}.
$$
Summarizing, the collection $\{b_1,\dots,b_l, b_{l+1}\}$ satisfies
\begin{itemize}
\item For any $\alpha,\beta,\gamma\in \Lambda_{l+1}$ with $\alpha\neq \beta\neq \gamma$, we have
$$
|\nu_\alpha\times \nu_\beta|\geq \vartheta \quad \mathrm{and}\quad |\mathrm{det}(\nu_\alpha,\nu_\beta,\nu_\gamma)|\geq \vartheta^2.
$$
\item $|a_i-b_i|\leq d_{l+1}\colonequals 6l^5(D_l+d_l)\vartheta$ for any $i\in \{1,\dots,l+1\}$.
\item $|b_i-b_j|\leq D_{l+1}\colonequals D_l+d_l$ for any $i,j\in \{1,\dots,l+1\}$.
\end{itemize}
This concludes the induction step. It only remains to find upper bounds for the recursively defined distances $d_l$ and $D_l$.
Observe that 
$$
D_{l+1}=D_l+d_l\leq D_l+6(l-1)^5(D_{l-1}+d_{l-1})\vartheta\leq D_l(1+6l^5\vartheta),\quad  D_1=D_2=D_\AA.
$$
We immediately check that if $\vartheta\leq m^{-6}$ then, for any $l\in \{1,\dots,m\}$,
$$
D_l\leq D_m\leq D_\AA(1+6m^5\vartheta)^m\leq D_\AA\left(1+\frac{6}{m}\right)^m\leq CD_\AA,
$$
where $C$ is a universal constant. Moreover if $\vartheta\leq m^{-6}$ then, for any $l\in\{1,\dots,m\}$,
$$
d_l\leq 6l^5D_{l-1}\vartheta\leq 6CD_\AA l^5\vartheta. 
$$
\end{proof}
\subsection{Proof of Proposition \ref{prop:functionzeta}}
\begin{proof}
Let us begin by defining the function $\zeta_\la$. Let $\AA'=\{p_1',\dots,p_k',n_1',\dots,n_k'\}$ be the collection of points given by Proposition \ref{displacement} for $D=D_\AA$. 
Observe that for any $\vartheta\leq C_0(2k)^{-6}$, where $C_0=\min(1,\vartheta_1,\vartheta_3)$, we have
$$
L(\AA')\leq \sum_{i=1}^k|p_i'-n_i'|\leq \sum_{i=1}^k|p_i-n_i|+|p_i-p_i'|+|n_i-n_i'|\leq L(\AA)+C D_\AA (2k)^6\vartheta.
$$
An analogous argument shows that $L(\AA)\leq L(\AA')+CD_\AA (2k)^6\vartheta$. Therefore
\begin{equation}\label{difMinCon}
|L(\AA)-L(\AA')|\leq CD_\AA (2k)^6\vartheta,
\end{equation}
where throughout the proof $C>0$ denotes a universal constant that may change from line to line. 
Remember that by Lemma \ref{LemmaBCL} there exists a $1$-Lipschitz function $\zeta^*:\cup_{i=1,\dots,k}\{p_i',n_i'\}\to \R$ such that 
$$
L(\AA')=\sum_{i=1}^k\zeta^*(p_i')-\zeta^*(n_i').
$$
Define the function $\zeta$ as in Definition \ref{def:funczeta}, i.e. set
$$
\zeta(x)\colonequals \max_{i\in \{1,\dots,k\}} \left(\zeta^*(p_i')-\max_{
j\in \{1,\dots,2k\}}d_{(i,j)}(x)\right),
$$
where
$$
d_{(i,j)}(x)\colonequals \langle p_i'-x,\nu_{(i,j)}\rangle,\quad \nu_{(i,j)}=\left\{\begin{array}{cl}\frac{p_i'-a_j'}{|p_i'-a_j'|}&\mathrm{if}\ p_i'\neq a_j'\\0&\mathrm{if}\ p_i'=a_j'\end{array}\right. .
$$
Lemma \ref{lem:funczeta} yields that $\zeta:\R^3\to \R$ is a $1$-Lipschitz function such that
$$
\displaystyle\sum_{i=1}^{k}\zeta(p_i')-\zeta(n_i')=\sum_{i=1}^k\zeta^*(p_i')-\zeta^*(n_i')= L(\AA').
$$
Next, we regularize the function $\zeta$. Let $\varphi\in C_c^\infty(B(0,1),\R_+)$ be a mollifier such that $\int_{\R^3} \varphi (x)dx=1$. Letting 
\begin{equation}\label{deflambda}
\la\colonequals \vartheta^{1/ \rho}\quad \mathrm{for} \ \rho\in(0,1/2),
\end{equation}
we define
$$
\zeta_\la(\cdot)\colonequals \varphi_\la * \zeta(\cdot)=\int_{\R^3} \varphi_\la(\cdot-y)\zeta(y)dy\quad \mathrm{with}\ \varphi_\la(\cdot)=\frac1\la \varphi\left(\frac{\cdot}{\la}\right).
$$

\medskip\noindent
{\bf Argument for the first statement.} Observe that $\|\zeta-\zeta_\la\|_{L^\infty(\R^3)}\leq \la$ from which we deduce that 
\begin{equation}\label{difMinCon2}
\left|L(\AA')-\sum_{i=1}^k \zeta_\la(p_i')-\zeta_\la(n_i')\right|\leq 2k\lambda.
\end{equation}
By combining \eqref{difMinCon} with \eqref{difMinCon2}, we obtain
\begin{equation}\label{dif1}
\left|L(\AA)-\sum_{i=1}^k \zeta_\la(p_i')-\zeta_\la(n_i')\right|\leq CD_\AA (2k)^6\vartheta+ 2k\lambda.
\end{equation}
On the other hand, note that
$$
\left|\sum_{i=1}^k \zeta(p_i)-\zeta(n_i)-\sum_{i=1}^k \zeta(p_i')-\zeta(n_i')\right|\leq \sum_{i=1}^k|p_i-n_i|+|p_i'-n_i'|\leq CD_\AA (2k)^6\vartheta.
$$
By combining the previous estimate with \eqref{difMinCon2}, we get
\begin{equation}\label{dif2}
\left|\sum_{i=1}^k \zeta(p_i)-\zeta(n_i)-\sum_{i=1}^k \zeta_\la(p_i')-\zeta_\la(n_i')\right|\leq CD_\AA (2k)^6\vartheta+ 2k\lambda.
\end{equation}
Then by \eqref{dif1} and \eqref{dif2}, we deduce that
$$
\left|L(\AA)-\sum_{i=1}^k \zeta_\la(p_i)-\zeta_\la(n_i) \right|\leq CD_\AA(2k)^6\la^\rho.
$$

\medskip\noindent
{\bf Argument for the second statement.} Note that 
\begin{equation}\label{gradzeta}
\nabla \zeta_\la(x)=\int_{B(x,\la)}\varphi_\la(x-y)\nabla \zeta (y) dy
\end{equation}
for any $x\in \R^3$. We define 
$$
\Lambda\colonequals \{(i,j)\ | \ 1\leq i\leq k,\ 1\leq j\leq 2k,\ i\neq j\}.
$$
Then, letting
$$
\zeta_{(i,j)}(\cdot)\colonequals \zeta(p_i')-d_{(i,j)}(\cdot) \quad \mathrm{for}\ (i,j)\in \Lambda,
$$
observe that, for almost every $y\in \R^3$,
$$
\nabla \zeta (y)=\nu_{(i,j)}\quad \mathrm{if}\ \zeta(y)=\zeta_{(i,j)}(y)\ \mathrm{for\ some}\ (i,j)\in\Lambda.
$$
Since $|\nu_{(i,j)}|=1$ for any $(i,j)\in \Lambda$, we have
$$
|\nabla \zeta_\la(x)|\leq \int_{B(x,\la)}\varphi_\la(x-y)|\nabla \zeta (y)| dy\leq \int_{B(x,\la)}\varphi_\la(x-y)dy=1
$$
for any $x\in \R^3$.

\medskip\noindent
{\bf Argument for the third statement.}
We now analyze the set of points whose gradient is small in modulus.
From \eqref{gradzeta}, we deduce that
$$
\nabla \zeta_\la(x)=\sum_{\alpha\in \Lambda}\sigma_\alpha\nu_\alpha, \quad \mathrm{where}\ \sigma_\alpha=\int_{B(x,\la)}\varphi_\la(x-y)\mathbf{1}_{\{\zeta(y)=\zeta_\alpha(y)\}}dy\quad \mathrm{for} \ \alpha\in \Lambda.
$$
Observe that $\sigma_\alpha\in[0,1]$ and that $\sum_{\alpha\in \Lambda}\sigma_\alpha=1$. We conclude that, for any $x\in \R^3$, $\nabla \zeta_\la(x)$ is a convex combination of the vectors $\nu_\alpha$'s, $\alpha\in \Lambda$. By Caratheodory's theorem, we deduce that $\nabla\zeta_\la(x)$ is a convex combination of at most four of them. 

Let us consider indices $i,j\in \{1,\dots,k\}$ with $i\neq j$. We let 
$$
P_{i,j}\colonequals \{y\in\R^3 \ | \ \zeta_{(i,j)}(y)=\zeta_{(j,i)}(y) \}
$$
and observe that
$$
P_{i,j}=\{\zeta(p_i')-\zeta(p_j')-\langle p_i'+p_j'-2y,\nu_{(i,j)}\rangle =0\}.
$$
A simple computation shows that
$$
\langle y_1-y_2,\nu_{(i,j)}\rangle =0
$$
for any $y_1,y_2\in P_{i,j}$ with $y_1\neq y_2$.
This implies that $P_{i,j}$ is a plane whose normal is $\nu_{(i,j)}$ and therefore 
$$
\zeta_{(i,j)}(y)=\zeta_{(j,i)}(y)=\frac{\zeta^*(p_i')+\zeta^*(p_j')-\langle p_i'+p_j' ,\nu_{(i,j)} \rangle }2 
$$
for any $y\in P_{i,j}$. We define 
$$
P_\la \colonequals \{y\in \R^3 \ | \ d(y,P)\leq 2\la \},\quad \mathrm{where} \ P\colonequals \cup_{1\leq i<j\leq k} P_{i,j}.
$$
We immediately check that $|\zeta_\la (P_\la)| \leq C\la k^2$.

Consider a number $\kappa<\vartheta^2/3$ and a point
$$
x\in \{y\in \R^3 \ | \ |\nabla \zeta_\la (y)|<\kappa \}\setminus P_\la.
$$
We observe that, since $x\not\in P_\la$, if there exists a point $y\in \overline{B(x,\la)}$ and indices $i,j\in\{1,\dots,k\}$ with $i\neq j$ such that
$$
\zeta(y)=\zeta_{(i,j)}(y)
$$
then, for any $z\in \overline{B(x,\la)}$,
$$
\zeta(z)\neq \zeta_{(j,i)}(z)
$$
This implies that $\nabla \zeta_\la(x)$ is a convex combination of at most four vectors, where if one of them happens to be $\nu_{(i,j)}$ for some $i,j\in \{1,\dots,k\}$ with $i\neq j$ then all the other vectors are different from $\nu_{(j,i)}=-\nu_{(i,j)}$. 
Recalling that the points of the collection $\AA'$ are such that  
$$
|\nu_\alpha \times \nu_\beta|>\vartheta\quad\mathrm{and}\quad |\mathrm{det}(\nu_\alpha,\nu_\beta,\nu_\gamma)|>\vartheta^2
$$
for any $\alpha,\beta,\gamma \in \{(i,j)\ | \ 1\leq i \leq k,\ 1 \leq j\leq 2k,\ i< j\}\subsetneq \Lambda$ with $\alpha\neq \beta\neq \gamma$ we deduce that $\nabla \zeta_\la(x)$ is a convex combination of at most four vectors that satisfy the previous properties.

Let us now show that $\nabla \zeta_\la(x)$ cannot be a convex combination of three or fewer of the vectors $\nu_\alpha$'s, $\alpha \in \Lambda$. We have three cases to consider:
\begin{itemize}
\item If there exists $\alpha\in \Lambda$ such that $\nabla \zeta_\la=\nu_\alpha$ in $B(x,\la)$ then 
$$
|\nabla \zeta_\la (x)|=|\nu_\alpha|=1\ \mathrm{in}\ B(x,\la).
$$ 

\item If there exist $\alpha,\beta \in \Lambda$ with $\alpha\neq \beta$ such that 
$$
\nabla \zeta_\la(x)=\sigma \nu_\alpha + (1-\sigma)\nu_\beta \ \mathrm{in} \ B(x,\la),
$$ 
for some $\sigma \in (0,1)$, then
\begin{align*}
|\nabla \zeta_\la (x)|&=|\sigma \nu_\alpha +(1-\sigma)\nu_\beta|\\
&\geq \max \left\{|\nabla \zeta_\la(x)\times \nu_\alpha|,|\nabla \zeta_\la(x)\times \nu_\beta| \right\} \\
&= \max \left\{(1-\sigma)|\nu_\alpha\times \nu_\beta|,\sigma|\nu_\alpha\times \nu_\beta| \right \}\\
&\geq \max\{\sigma,1-\sigma\}\vartheta\geq\frac \vartheta2.
\end{align*}

\item If there exist $\alpha,\beta,\gamma \in \Lambda$ with $\alpha\neq \beta \neq \gamma$ such that  
$$
\nabla \zeta_\la(x)=\sigma_\alpha \nu_\alpha + \sigma_\beta\nu_\beta+\sigma_\gamma \nu_\gamma\ \mathrm{in}\ B(x,\la),
$$ 
for some numbers $\sigma_\alpha,\sigma_\beta,\sigma_\gamma \in (0,1)$ with $\sigma_\alpha+\sigma_\beta+\sigma_\gamma=1$, then, assuming without loss of generality that $\sigma_\alpha\geq \frac13$, we have
$$
|\nabla \zeta_\la (x)|\geq \sigma_\alpha |\nu_\alpha \cdot (\nu_\beta \times \nu_\gamma)|= \sigma_\alpha |\mbox{det}(\nu_\alpha,\nu_\beta,\nu_\gamma)|\geq \frac {\vartheta^2}3.
$$
\end{itemize}
Since $\kappa<\frac{\vartheta^2}3$ we deduce that the three cases considered above cannot occur. Therefore we conclude that there exist $\alpha,\beta,\gamma,\eta \in \Lambda$ with $\alpha\neq\beta\neq\gamma\neq\eta$ such that  
$$
\nabla \zeta_\la(x)=\sigma_\alpha \nu_\alpha+\sigma_\beta \nu_\beta+\sigma_\gamma \nu_\gamma+\sigma_\eta \nu_\eta \ \mathrm{in} \ B(x,\la),
$$
for some $\sigma_\alpha,\sigma_\beta,\sigma_\gamma,\sigma_\eta \in (0,1)$ with $\sigma_\alpha+\sigma_\beta+\sigma_\gamma+\sigma_\eta=1$. 

Let us solve consider the system of equations
\begin{equation}\label{system}
\zeta_\alpha(y)=\zeta_\beta(y)=\zeta_\gamma(y)=\zeta_\eta(y).
\end{equation}
We claim that this system admits a unique solution $y\in \R^3$ which in addition satisfies $$
|x-y|\leq \frac{C\la}{(\vartheta^2-3\kappa)}.
$$ 
Writing $\tilde y=y-x$, we observe that $\tilde y$ satisfies the linear system of equations $A\tilde y=B$, where
$$
A=\left(\begin{array}{c}\nu_\alpha-\nu_\beta\\ \nu_\gamma-\nu_\beta\\ \nu_\eta-\nu_\beta
\end{array}\right)\quad\mathrm{and}\quad B=\left(\begin{array}{c}\zeta_\alpha(x)-\zeta_\beta(x)\\ \zeta_\gamma(x)-\zeta_\beta(x) \\ \zeta_\eta(x)-\zeta_\beta(x)
\end{array}\right).
$$
Let us check that $|\mathrm{det}(A)|\geq 4(\vartheta^2-3\kappa)$. Note that without loss of generality we can assume that $\sigma_\alpha\leq \frac14$. Observe that
$$
\nabla \zeta_\la (x)-\nu_\beta=\sigma_\alpha(\nu_\alpha-\nu_\beta)+\sigma_\gamma(\nu_\gamma-\nu_\beta)+\sigma_\eta(\nu_\eta-\nu_\beta).
$$
By Cramer's rule, we have
$$
\sigma_\alpha= \frac{\mbox{det}(\nabla \zeta_\la(x)-\nu_\beta,\nu_\gamma-\nu_\beta,\nu_\eta-\nu_\beta)}{\mbox{det}(\nu_\alpha-\nu_\beta,\nu_\gamma-\nu_\beta,\nu_\eta-\nu_\beta)}.
$$
Simple computations show that
$$
\mbox{det}(\nabla \zeta_\la(x)-\nu_\beta,\nu_\gamma-\nu_\beta,\nu_\eta-\nu_\beta)=-\mbox{det}(\nu_\beta,\nu_\gamma,\nu_\eta)+f(\nabla \zeta_\la(x)),
$$
where $|f(\nabla \zeta_\la(x))|\leq 3|\nabla \zeta_\la(x)|\leq 3\kappa$. Therefore 
$$
|\mbox{det}(\nabla \zeta_\la(x)-\nu_\beta,\nu_\gamma-\nu_\beta,\nu_\eta-\nu_\beta)|\geq |\mbox{det}(\nu_\beta,\nu_\gamma,\nu_\eta)|-|f(\nabla \zeta_\la(x))| \geq \vartheta^2 -3\kappa.
$$ 
We deduce that
$$
|\mbox{det}(\nu_\alpha-\nu_\beta,\nu_\gamma-\nu_\beta,\nu_\eta-\nu_\beta)|=\frac{|\mbox{det}(\nabla \zeta_\la(x)-\nu_\beta,\nu_\gamma-\nu_\beta,\nu_\eta-\nu_\beta)|}{\sigma_\alpha}\geq 4(\vartheta^2-3\kappa).
$$
On the other hand, note that there exist $x_\alpha,x_\beta,x_\gamma,x_\eta$ in $B(x,\la)$ such that 
$$
\zeta(x_\alpha)=\zeta_\alpha(x_\alpha), \;
\zeta(x_\beta)=\zeta_\beta(x_\beta), \;
\zeta(x_\gamma)=\zeta_\gamma(x_\gamma), \;
\zeta(x_\eta)=\zeta_\eta(x_\eta).
$$
Since
$$
B=\left(\begin{array}{c}
\zeta_\alpha(x)-\zeta_\alpha(x_\alpha)+\zeta(x_\alpha)-\zeta(x_\beta)+\zeta_\beta(x_\beta)-\zeta_\beta(x)\\

\zeta_\gamma(x)-\zeta_\gamma(x_\gamma)+\zeta(x_\gamma)-\zeta(x_\beta)+\zeta_\beta(x_\beta)-\zeta_\beta(x)\\

\zeta_\eta(x)-\zeta_\eta(x_\eta)+\zeta(x_\eta)-\zeta(x_\beta)+\zeta_\beta(x_\beta)-\zeta_\beta(x)
\end{array}\right),
$$
we deduce that $|B|\leq 3\la$. Hence the linear system of equations $A\tilde y=B$ admits a unique solution which satisfies
$$
|\tilde y|=|y-x|= |A^{-1}B|\leq \frac{C\la}{\vartheta^2-3\kappa}.
$$
Summarizing, if $x\in \{y\in \R^3\ | \ |\nabla \zeta_\la(y)|<\kappa\}\setminus P_\la$ with $\kappa<\vartheta^2/3$ then there exist $\alpha,\beta,\gamma,\eta\in \Lambda$ with $\alpha\neq \beta\neq \gamma\neq \eta$ such that the unique solution $y\in \R^3$ to  \eqref{system} lies in the ball $B(x,C\lambda/(\vartheta^2-3\kappa))$. We conclude that the set 
$$
C_\kappa= \{x \ | \ |\nabla \zeta_\la (x)|<\kappa\}\setminus P_\la
$$
can be covered by $\B_\kappa$, a collection of at most $\binom{|\Lambda|}{4}\leq (2k)^8$ balls of radius $C\la/(\vartheta^2-3\kappa)$. Observing that
$$
|D^2\zeta_\la(x)|\leq \frac C{\la^2}
$$
for any $x\in \R^3$, and letting 
$$
T_\kappa=\zeta_\la (\cup_{B\in \B_\kappa}B),
$$
we deduce that, for any $t\in \R\setminus (T_\kappa\cup \zeta_\la(P_\la))$, $\{x \ | \ \zeta_\la(x)=t\}$ is a complete submanifold of $\R^3$ whose second fundamental form is bounded by 
$$
C \frac{\sup_{\R^3} |D^2 \zeta_\la|}{\inf_{\R^3\setminus ((\cup_{B\in \B_\kappa}B)\cup P_\la)} |\nabla \zeta_\la| }\leq\frac C{\la^2\kappa}.
$$
Recalling the relation between $\la$ and $\vartheta$ (see \eqref{deflambda}), the proposition follows.
\end{proof}

%%%%%%%%%%%%%%%%%%%%%%%%%%%%%%%%%%%%%%%%%%%%%%%%%%%%%%%%%%%%%%%%%%%%%%%%%%%%%%%%%%%%%%%%%%%%%%%%%

\section{Smooth approximation of the function \texorpdfstring{$\zeta$}{zeta} for the distance through the boundary -- First method}\label{Sec:AppendixB}
In this section of the Appendix we prove Proposition \ref{prop:functionzetaBoundary}. We will smoothly approximate the function $\zeta$ for $d_{\partial\Omega}$ by convolution, after displacing the points $a_i$ as in Appendix \ref{Sec:AppendixA}. The main points of the proof are:
\begin{itemize}
\item Since $\partial\Omega$ is assumed to be of class $C^2$, if we reduce the analysis to a small neighborhood close to the boundary then the gradient of the distance to the boundary at every point of this neighborhood is given by the normal to the boundary at the unique projection to the boundary of this point.
\item We will characterize the set where the distance to the boundary is equal to one or two of the functions $\zeta_{i,j}$'s, while the gradient vectors of these functions do not satisfy ``good'' angle conditions between each other in the sense described in the previous section. One can show that the image of this set (by our smooth approximation) has small measure. To prove this fact, we will present an argument based on the curvature of the boundary. The assumption that $\partial \Omega$ is of class $C^2$ gives an upper bound for the maximal principal curvature at each point of the boundary, which roughly speaking implies that the boundary ``cannot wiggle too much''.
\item We will adapt the last part of the proof of Proposition \ref{prop:functionzeta}. Arguing in the same fashion, but using the inverse function theorem, we can show that the set where the distance to the boundary is equal to three of the functions $\zeta_{i,j}$'s can be covered by a quantitative number of small balls.
\end{itemize}
\begin{proof}
Let us begin by defining the function $\zeta_\la$. Let $\AA'=\{p_1',\dots,p_k',n_1',\dots,n_k'\}$ be the collection of points given by Proposition \ref{displacement} for $D=\mathrm{diam}(\Omega)$. 
Observe that for any $\vartheta\leq C_0(2k)^{-6}$, where $C_0=\min(1,\vartheta_1,\vartheta_3)$, we have
\begin{align*}
L_{\partial\Omega}(\AA')\leq \sum_{i=1}^kd_{\partial\Omega}(p_i',n_i')|
&\leq \sum_{i=1}^kd_{\partial\Omega}(p_i,n_i)+d_{\partial\Omega}(p_i,p_i')+d_{\partial\Omega}(n_i,n_i')
\\ &\leq L_{\partial\Omega}(\AA)+C \mathrm{diam}(\Omega) (2k)^6\vartheta.
\end{align*}
An analogous argument shows that $L_{\partial\Omega}(\AA)\leq L_{\partial\Omega}(\AA')+C\mathrm{diam}(\Omega) (2k)^6\vartheta$. Therefore
\begin{equation}\label{difMinConB}
|L_{\partial\Omega}(\AA)-L_{\partial\Omega}(\AA')|\leq C(2k)^6\vartheta.
\end{equation}
Throughout the proof $C>0$ denotes a constant depending only on $\partial\Omega$, that may change from line to line. 
	
Remember that by Lemma \ref{LemmaBCLBoundary} there exists a $1$-Lipschitz function $\zeta^*:\cup_{i=1,\dots,k}\{p_i',n_i'\}\to \R$ such that 
$$
L_{\partial\Omega}(\AA')=\sum_{i=1}^k\zeta^*(p_i')-\zeta^*(n_i').
$$
Define the function $\zeta$ for $d_{\partial\Omega}$ as in Definition \ref{def:funczetaBoundary}, i.e. set
$$
\zeta(x)\colonequals \max_{i\in \{1,\dots,k\}} \left(\zeta^*(p_i')-d_i(x,\partial\Omega)\right),
$$
where
$$
d_i(x,\partial\Omega)\colonequals 
\min\left[
\max \left( \max_{j\in\{1,\dots,2k\}}d_{(i,j)}(x),d(p_i',\partial\Omega)-d(x,\partial\Omega)\right),d(p_i',\partial\Omega)+d(x,\partial\Omega)\right],
$$
and
$$
d_{(i,j)}(x)\colonequals \langle p_i'-x,\nu_{(i,j)}\rangle,\quad \nu_{(i,j)}=\left\{\begin{array}{cl}\frac{p_i'-a_j'}{|p_i'-a_j'|}&\mathrm{if}\ p_i'\neq a_j'\\0&\mathrm{if}\ p_i'=a_j'\end{array}\right. .
$$
Lemma \ref{lem:funczetaBoundary} yields that $\zeta:\R^3\to \R$ is a $1$-Lipschitz function such that
$$
\displaystyle\sum_{i=1}^{k}\zeta(p_i')-\zeta(n_i')=\sum_{i=1}^k\zeta^*(p_i')-\zeta^*(n_i')= L_{\partial\Omega}(\AA').
$$
Next, we regularize the function $\zeta$. Let $\varphi\in C_c^\infty(B(0,1),\R_+)$ be a mollifier such that $\int_{\R^3} \varphi (x)dx=1$. Letting 
\begin{equation}\label{deflambdaB}
\la\colonequals \vartheta^{1/ \rho}\quad \mathrm{for} \ \rho\in(0,1/4),
\end{equation}
we define
$$
\zeta_\la(\cdot)\colonequals \varphi_\la * \zeta(\cdot)=\int_{\R^3} \varphi_\la(\cdot-y)\zeta(y)dy\quad \mathrm{with}\ \varphi_\la(\cdot)=\frac1\la \varphi\left(\frac{\cdot}{\la}\right).
$$

\medskip\noindent
{\bf Argument for the first statement.}
Observe that $\|\zeta-\zeta_\la\|_{L^\infty(\R^3)}\leq \la$ from which we deduce that 
\begin{equation}\label{difMinCon2B}
\left|L_{\partial\Omega}(\AA')-\sum_{i=1}^k \zeta_\la(p_i')-\zeta_\la(n_i')\right|\leq 2k\lambda.
\end{equation}
By combining \eqref{difMinConB} with \eqref{difMinCon2B}, we obtain
\begin{equation}\label{dif1B}
\left|L_{\partial\Omega}(\AA)-\sum_{i=1}^k \zeta_\la(p_i')-\zeta_\la(n_i')\right|\leq C (2k)^6\vartheta+ 2k\lambda.
\end{equation}
On the other hand, note that
$$
\left|\sum_{i=1}^k \zeta(p_i)-\zeta(n_i)-\sum_{i=1}^k \zeta(p_i')-\zeta(n_i')\right|\leq \sum_{i=1}^k|p_i-n_i|+|p_i'-n_i'|\leq C(2k)^6\vartheta.
$$
By combining the previous estimate with \eqref{difMinCon2B}, we get
\begin{equation}\label{dif2B}
\left|\sum_{i=1}^k \zeta(p_i)-\zeta(n_i)-\sum_{i=1}^k \zeta_\la(p_i')-\zeta_\la(n_i')\right|\leq C(2k)^6\vartheta+ 2k\lambda.
\end{equation}
Then by \eqref{dif1B} and \eqref{dif2B}, we deduce that
$$
\left|L_{\partial\Omega}(\AA)-\sum_{i=1}^k \zeta_\la(p_i)-\zeta_\la(n_i) \right|\leq C(2k)^6\la^\rho.
$$
	
\medskip\noindent
{\bf Argument for the second statement.}
Note that, for any $x\in \R^3$,
\begin{equation}\label{gradzetaB}
\nabla \zeta_\la(x)=\int_{B(x,\la)}\varphi_\la(x-y)\nabla \zeta (y) dy
\end{equation}
Let us observe that, since $\partial\Omega$ is assumed to be of class $C^2$, there exists a fixed number $\theta_0=\theta_0(\partial\Omega)$ such that if $y\in \Omega$ satisfies $d(y,\partial\Omega)< \theta_0$ then 
$$
\nabla d(y,\partial\Omega)=\nu(z_y),
$$
where throughout the proof $\nu(z_y)$ denotes the outer unit normal vector to $\partial\Omega$ at $z_y\colonequals \mathrm{proj}_{\partial\Omega}(y)$.
	
We define  
$$
\Lambda \colonequals \{(i,j)\ | \ 1\leq i\leq k,\ 1\leq j\leq 2k,\ i\neq j\}\quad\mathrm{and}\quad
c\colonequals \max_{i\in \{1,\dots,k\}}(\zeta^*(p_i')-d(p_i',\partial\Omega)).
$$
Then, letting
\begin{align*}
\zeta_{(i,j)}(\cdot)&\colonequals\zeta^*(p_i')-d_{(i,j)}(\cdot) \quad \mathrm{for}\ (i,j)\in \Lambda\\
\zeta_{+}(\cdot)&\colonequals c+d(\cdot,\partial\Omega)\\
\zeta_{-}(\cdot)&\colonequals c-d(\cdot,\partial\Omega), 
\end{align*}
observe that, for almost every $y\in \Omega$ such that $d(y,\partial\Omega)<\theta_0$,
$$
\nabla \zeta (y)=\left\{
\begin{array}{rl}
\nu_{(i,j)}  & \mathrm{if}\  \zeta(y)=\zeta_{(i,j)}(y)\quad \mathrm{for\ some}\ (i,j)\in \Lambda\\
\nu(z_y)     & \mathrm{if}\ \zeta(y)=\zeta_{+}(y) \\
-\nu(z_y)    & \mathrm{if}\ \zeta(y)=\zeta_{-}(y).
\end{array}\right.
$$
In particular $|\nabla \zeta(y)|=1$ for almost every $y$ as above. Thus 
$$
|\nabla \zeta_\la(x)|\leq \int_{B(x,\la)}\varphi_\la(x-y)|\nabla \zeta (y)| dy\leq \int_{B(x,\la)}\varphi_\la(x-y)dy=1
$$
for any 
\begin{equation}\label{setla}
x\in \Omega_\la\colonequals \{x \in \Omega \ | \ 2\lambda^{1/\rho} <d(x,\partial\Omega)<\theta_0-2\lambda^{1/\rho} \}. 
\end{equation}
	
\medskip\noindent
{\bf Argument for the third statement.} Observe that 
$$
\zeta(x)=c\quad \mbox{for any }x\in \partial\Omega.
$$
Thus
$$
|\zeta_\la(\{ x\in \Omega \ | \ d(x,\partial\Omega)\leq 2\lambda^{1/\rho}\})|\leq C \lambda^{1/\rho}.
$$
	
\medskip\noindent
{\bf Definition of the set $P_\la$.} From \eqref{gradzetaB}, we deduce that
\begin{multline}\label{grad}
\nabla \zeta_\la (x)= \sum_{\alpha\in\Lambda} \sigma_\alpha \nu_\alpha 
+\int_{B(x,\la )}\varphi_\la(x-y)\nu(z_y)\mathbf{1}_{\zeta(y)=\zeta_+(y)}dy \\+\int_{B(x,\la )}\varphi_\la(x-y)(-\nu(z_y))\mathbf{1}_{\zeta(y)=\zeta_-(y)}dy,
\end{multline}
where 
$$
\sigma_\alpha= \int_{B(x,\la )}\varphi_\la(x-y)\mathbf{1}_{\zeta(y)=\zeta_\alpha(y)}dy \quad \mathrm{for}\ \alpha\in\Lambda.
$$
Let us observe that if there exists a point $x\in \Omega_\la$ such that $\zeta_\la(x)=\pm d(x,\partial \Omega)$ then for any $y\in B(x,\la)$, $\zeta_\la(y)\neq \mp d(y,\partial\Omega)$. This implies that if the second term in the right-hand side of \eqref{grad} is different from zero then the third term vanishes, and viceversa. 
	
On the other hand, given $x\in \Omega_\la$, if there exists $y_x\in B(x,\la)$ such that $\zeta_\la(y_x)=\pm d(y_x,\partial \Omega)$ then
\begin{multline*}
\int_{B(x,\la )}\varphi_\la(x-y)\nu(z_y)\mathbf{1}_{\zeta(y)=\zeta_\pm (y)}dy=\nu(z_{y_x})\int_{B(x,\la )}\varphi_\la(x-y)\mathbf{1}_{\zeta(y)=\zeta_\pm (y)}dy\\
+\int_{B(x,\la )}\varphi_\la(x-y)(\nu(z_y)-\nu(z_{y_x}))\mathbf{1}_{\zeta(y)=\zeta_\pm (y)}dy.
\end{multline*}
But $|\nu(z_y)-\nu(z_{y_x})|\leq C\la$, and therefore
$$
\int_{B(x,\la )}\varphi_\la(x-y)\nu(z_y)\mathbf{1}_{\zeta(y)=\zeta_\pm (y)}dy=\alpha_\pm(1+O(\la))\nu(z_{y_x})
$$
where
$$
\alpha_\pm\colonequals \int_{B(x,\la )}\varphi_\la(x-y)\mathbf{1}_{\zeta(y)=\zeta_\pm (y)}dy.
$$
Thus, \eqref{grad} can be written as
$$
\nabla \zeta_\la (x)= \sum_{\alpha\in\Lambda} \sigma_\alpha \nu_\alpha 
+\alpha_{\pm}(1+O(\la))\nu(z_{y_x}),
$$
where $y_x$ is arbitrarily chosen among the points in $y\in B(x,\la)$ such that $\zeta_\la(y)=\pm d(y,\partial \Omega)$.
	
\medskip
To prove the fourth statement, we will follow the same strategy as in the proof of the third statement of Proposition \ref{prop:functionzeta}. For this, we need to make sure that the vectors appearing in the previous expression satisfy a good angle condition between each other, which forces us to remove from our analysis a set of ``bad'' points whose image by $\zeta_\la$ has small measure.
	
We define
$$
P^\Lambda _\la \colonequals \{z\in \R^3 \ | \ d(z,P)\leq 2\la \},
$$
where
$$
P\colonequals \cup_{1\leq i<j\leq k} P_{i,j}, \quad 
P_{i,j}\colonequals \{y\in\R^3 \ | \ \zeta_{(i,j)}(y)=\zeta_{(j,i)}(y) \}.
$$
Arguing as in the proof of Proposition \ref{prop:functionzeta}, we deduce that $|\zeta_\la (P^\Lambda_\la)| \leq 2\la k^2$.
	
\medskip
Let us also define
\begin{multline*}
P_{\la,\pm}^{\mathrm{dipole}}\colonequals \big \{ x\in \Omega_\la\setminus P_\la^\Lambda \ | \ \mathrm{there\ exist}\ y_1,y_2\in B(x,2\la), \alpha\in \Lambda \ \mathrm{such\ that}\\ \zeta_\la(y_1)=\zeta_\alpha,\ \zeta_\la(y_2)=\zeta_\pm,\ \mathrm{and} \ |\nu_\alpha \times \pm \nu(z_{y_2})|<\vartheta \big\}
\end{multline*}
and
\begin{multline*}
P_{\la,\pm}^\mathrm{tripole}\colonequals \big \{ x\in \Omega_\la\setminus P_\la^\Lambda \ | \ \mathrm{there\ exist}\ y_1,y_2,y_3\in B(x,2\la), \alpha,\beta\in \Lambda, \alpha\neq \beta, \ \mathrm{such\ that}\\ \zeta_\la(y_1)=\zeta_\alpha,\ \zeta_\la(y_2)=\zeta_\beta,\ \zeta_\la(y_3)=\zeta_\pm,\ \mathrm{and} \ |\det(\nu_\alpha,\nu_\beta,\pm \nu(z_{y_3})) |<\vartheta^2 \big\}.
\end{multline*}
We claim that 
$$
|\zeta_\la(P_{\la,+}^\mathrm{dipole}\cup P_{\la,-}^\mathrm{dipole}\cup P_{\la,+}^\mathrm{tripole}\cup P_{\la,-}^\mathrm{tripole}) |\leq C(2k)^4\la^{3\rho/4}.
$$
Our argument is based on the curvature of $\partial \Omega$. Given a point $z\in\partial \Omega$, we denote by $k_{\min}(z)$ the minimal principal curvature of $\partial\Omega$ at $z$. We also denote by $r_{\min}(z)$ (resp. $r_{\max}(z)$) the minimal (resp. maximal) principal radii of curvature at $z$. Let us observe that since $\Omega$ is of class $C^2$, for any point $z\in \partial\Omega$, $r_{\min}(z)\geq C>0$. 
	
We next study the sets $P_{\la,\pm}^\mathrm{dipole}$. Given $x\in P_\la$, let us first assume that $|k_{\min}(z_{y_2})|>\vartheta^{1/4}$. Simple geometric arguments show that there exists a constant $C_1(\partial\Omega)$ such that for any $y$ satisfying 
$$
|y_2-y|= C_1 r_{\max}(z_{y_2})\vartheta,
$$
we have
$$
|\nu_\alpha \times \pm\nu(z_y)|\geq \vartheta.
$$
Moreover, since $r_{\min}(z_{y_2})\geq C>0$, we deduce that there exists $C_2(\partial\Omega)$ such that, for any
$$
C_1r_{\max}(z_{y_2})\vartheta\leq |y_2-y|\leq C_2,
$$
we have
$$
|\nu_\alpha \times \pm\nu(z_y)|\geq \vartheta.
$$
Noting that $r_{\max}(z_{y_2})<1/\vartheta^{1/4}$, we deduce that $\{x\in P_{\la,\pm}^\mathrm{dipole}\ |\ |k_{\min}(z_{y_2})|> \vartheta^{1/4} \}$ can be covered by $\mathrm{vol}(\Omega_\la)C_2^{-3}|\Lambda|$ balls of radius $C_1\vartheta^{3/4}$.
	
Let us now assume that $|k_{\min}(z_{y_2})|\leq \vartheta^{1/4}$. In this case, the boundary ``looks'' flat around $z_{y_2}$ in the principal direction of minimal curvature. This, combined with the fact that the minimal principal radii of curvature is bounded below, immediately implies that the number of connected components of $\{ x\in P_{\la,\pm}^\mathrm{dipole} \ | \ |k_{\min}(z_{y_2})|\leq  \vartheta^{1/4} \}$ is bounded above by a constant that only depends on $\partial \Omega$. Then we easily deduce that
$$
|\zeta_\la(\{ x\in P_{\la,\pm}^\mathrm{dipole} \ | \ |k_{\min}(z_{y_2})|\leq  \vartheta^{1/4} \})|\leq C \vartheta.
$$
Thus 
$$
|\zeta_\la(P_{\la,\pm}^\mathrm{dipole})|\leq C(2k)^2 \vartheta^{3/4}.
$$
	
\medskip 
Let us now study the sets $P_{\la,\pm}^\mathrm{tripole}$. Arguing in a similar way, one can check that 
$$
\{ x\in P_{\la,\pm}^\mathrm{tripole} \ | \  |k_{\min}(z_{y_3})|> \vartheta^{1/4}\}
$$ 
can be covered by $\mathrm{vol}(\Omega_\la)\tilde C_2^{-3}\binom{|\Lambda|}{2}$ balls of radius $\tilde C_1\vartheta^{3/4}$. Also, one easily deduces that 
$$
|\zeta_\la(\{ x\in P_{\la,\pm}^\mathrm{tripole} \ | \ |k_{\min}(z_{y_3})|\leq \vartheta^{1/4} \})|\leq C \vartheta.
$$
Thus 
$$
|\zeta_\la(P_{\la,\pm}^\mathrm{tripole})|\leq C(2k)^4 \vartheta^{3/4}.
$$
	
\medskip
Finally, we set $P_\la\colonequals P^\Lambda _\la\cup P_{\la,+}^\mathrm{dipole}
\cup P_{\la,-}^\mathrm{dipole} \cup P_{\la,+}^\mathrm{tripole}\cup P_{\la,-}^\mathrm{tripole}$ and observe that
$$
|P_\la|\leq C(2k)^4 \vartheta^{3/4}=C(2k)^4\la^{3\rho/4}.
$$
	
\medskip\noindent
{\bf Argument for the fourth statement.} We now analyze the set of points whose gradient is small in modulus.
Consider a number $\kappa < \vartheta^2/3$ and a point
$$
x\in \{y\in \Omega_\la \ | \ |\nabla \zeta_\la(y)|<\vartheta \}\setminus P_\la.
$$
Arguing as in the proof of Proposition \ref{prop:functionzeta} we conclude that there exist four different functions 
$$
\zeta_1,\zeta_2,\zeta_3,\zeta_4\in \{\zeta_{(i,j)} \ | \ (i,j)\in\Lambda \}\cup \{\zeta_+\}\cup \{\zeta_-\},
$$ 
where if $\zeta_a=\zeta_{(i,j)}$ for some $(i,j) \in \Lambda$ and $a\in \{1,2,3,4\}$ then $\zeta_b\neq \zeta_{(j,i)}$ for any $b\in\{1,2,3,4\}\setminus \{a\}$ and if $\zeta_a=\zeta_\pm$ for some $a\in \{1,2,3,4\}$ then $\zeta_b\neq \zeta_\mp$ for any $b\in\{1,2,3,4\}\setminus \{a\}$. From the proof of Proposition \ref{prop:functionzeta}, we know that if $\zeta_a\in \{\zeta_{(i,j)} \ | \ (i,j)\in\Lambda \}$ for all $a=1,2,3,4$ then the unique solution $y\in\R^3$ to the linear system of equations \eqref{system} lies in the ball $B(x,C\la/(\vartheta^2-3\kappa))$. 
	
It remains to analyze the case when, without loss of generality, $\zeta_1=\zeta_+$. Let us write $\zeta_2=\zeta_\alpha$, $\zeta_3=\zeta_\beta$, and $\zeta_4=\zeta_\gamma$. We already know that
$$
\nabla \zeta_\la (x)=\sigma_+(1+O(\la))\nu(z_{y_x})+\sigma_\alpha \nu_\alpha +\sigma_\beta \nu_\beta+ \sigma_\gamma\nu_\gamma,
$$
for some $\sigma_+,\sigma_\alpha,\sigma_\beta,\sigma_\gamma\in (0,1)$ with $\sigma_+ + \sigma_\alpha+ \sigma_\beta+\sigma_\gamma=1$.
Arguing as in the proof of Proposition \ref{prop:functionzeta}, we deduce that
\begin{equation}\label{det}
|\det(\nu(z_{y_x})-\nu_\beta,\nu_\alpha-\nu_\beta,\nu_\gamma-\nu_\beta) |\geq 3(\vartheta^2-3\kappa).
\end{equation}
Let us consider the function $\Phi:\Omega_\la \to \R^3$ defined via
$$ 
\Phi(\cdot)=(\zeta_+-\zeta_\beta,\zeta_\alpha-\zeta_\beta,\zeta_\gamma-\zeta_\beta)(\cdot ).
$$
Observe that
$$
D \Phi (x)=(\nu(z_{y_x})-\nu_\beta,\nu_\alpha-\nu_\beta,\nu_\gamma-\nu_\beta),
$$
whose determinant satisfies \eqref{det}. Moreover, by noting that 
$$
|D \Phi (x)-D \Phi(y)|=|(\nu(z_{y_x})-\nu(z_y),0,0)|,
$$
we easily deduce that 
\begin{equation}\label{det2}
|\det(\nu(z_y)-\nu_\beta,\nu_\alpha-\nu_\beta,\nu_\gamma-\nu_\beta) |\geq (\vartheta^2-3\kappa),
\end{equation}
for any $y\in B(x,C_1(\vartheta^2-3\kappa))$, where $C_1$ is a constant depending only on $\partial\Omega$. Let us observe that our definition of $\Omega_\la$ (see \eqref{setla}) guaranties that $B(x,C_1(\vartheta^2-3\kappa))\subset B(x,C_1\la^{1/2\rho})\subset \{x \in \Omega \ | \ d(x,\partial \Omega)< \theta_0\}$, for any sufficiently small $\la$. We also observe that, since there exist $x_+,x_\alpha,x_\beta,x_\gamma \in B(x,\la)$ such that
$$
\zeta(x_+)=\zeta_+(x_+), \;
\zeta(x_\alpha)=\zeta_\alpha(x_\alpha), \;
\zeta(x_\beta)=\zeta_\beta(x_\beta), \;
\zeta(x_\gamma)=\zeta_\gamma(x_\gamma),
$$
we have
\begin{equation}\label{det3}
|\Phi(x)|\leq C\la.
\end{equation}
From a ``quantitative version'' of the inverse function theorem (see \cite{LangBook}*{Chapter XIV, Lemma 1.3}), we conclude that $\Phi$ is invertible in $B(x,C_1(\vartheta^2-3\kappa))$, and since $\rho<1/4$ in \eqref{deflambdaB} (i.e. $\la\ll \vartheta^4$), $0\in \Phi(B(x,C_1(\vartheta^2-3\kappa))$. In particular, there exists a unique $y\in B(x,C_1(\vartheta^2-3\kappa))$ such that $\Phi(y)=0$. Moreover
$$
|x-y|=|\Phi^{-1}(\Phi(x))-\Phi^{-1}(0)|\leq |D\Phi^{-1}(z)||\Phi(x)-0|
$$
for some $z\in \Phi(B(x,C_1(\vartheta^2-3\kappa))$. This combined with \eqref{det2} and \eqref{det3}, gives
$$
|x-y|\leq \frac{C\la}{\vartheta^2-3\kappa}.
$$
This means that the unique solution to $\Phi(y)=0$ in the ball $B(x,C_1(\vartheta^2-3\kappa))$ lies in the much smaller ball $B(x,\frac{C\la}{\vartheta^2-3\kappa})$. Since $\Omega_\la$ can be covered by $C\theta_0(\vartheta^2-3\kappa)^{-3}$ balls of radius $\vartheta^2-3\kappa$, we conclude that the set
$$
C_\kappa\colonequals \{x\in \Omega_\la \ | \ |\nabla \zeta_\la (x)|<\kappa\}\setminus P_\la
$$
can be covered by $\B_\kappa$, a collection of at most
$$
\binom{|\Lambda|}{4}+C\theta_0(\vartheta^2-3\kappa)^{-3}\binom{|\Lambda|}{3}
$$ 
balls of radius $C\la/(\vartheta^2-3\kappa)$. Observing that
$$
|D^2\zeta_\la(x)|\leq \frac C{\la^2}
$$
for any $x\in \Omega_\la$, and letting 
$$
T_\kappa\colonequals \zeta_\la (\cup_{B\in \B_\kappa}B),
$$
we deduce that, for any $t\in \zeta_\la(\Omega_\la)\setminus (T_\kappa\cup \zeta_\la(P_\la))$, $\{x \ | \ \zeta_\la(x)=t\}$ is a complete submanifold of $\R^3$ whose second fundamental form is bounded by 
$$
C \frac{\sup_{\Omega_\la} |D^2 \zeta_\la|}{\inf_{\Omega_\la\setminus ((\cup_{B\in \B_\kappa}B)\cup P_\la )} |\nabla \zeta_\la| }\leq\frac C{\la^2\kappa}.
$$
Recalling the relation between $\la$ and $\vartheta$ (see \eqref{deflambdaB}), the proposition follows.
\end{proof}

%%%%%%%%%%%%%%%%%%%%%%%%%%%%%%%%%%%%%%%%%%%%%%%%%%%%%%%%%%%%%%%%%%%%%%%%%%%%%%%%%%%%%%%%%%%%%%%%%

\section{Smooth approximation of the function \texorpdfstring{$\zeta$}{zeta} for the distance through the boundary -- Second Method}\label{Sec:AppendixC}
In this section of the Appendix we prove Proposition \ref{prop:functionzetaBoundary2}.
By assuming that the Gauss curvature of the boundary of $\Omega$ is strictly positive, we will provide a convex polyhedral approximation of $\partial \Omega$, very close in Hausdorff distance. We will then smoothly approximate the function $\zeta$ for the distance through the polyhedral approximation of $\partial \Omega$ by convolution, after performing a suitable displacement of the points of the collection. The main points of the proof are:
\begin{itemize}
\item The commodity of replacing the boundary of the domain by a convex polyhedron is that, where well-defined, the gradient of the function distance to the polyhedron is equal to the normal to one of its faces.
\item The strategy of proof is very similar to the one followed to prove Proposition \ref{prop:functionzeta}. But in this case to study the set of points whose gradient is small, we need to ensure that the normals to the faces of the convex polyhedral approximation of $\partial \Omega$ and the vectors $\nu_{(i,j)}$ satisfy ``good'' angle conditions between each other. To accomplish this we will carefully choose the approximating convex polyhedron, and then perform a displacement of the points of the configuration $\AA$.
\end{itemize}

\subsection{Polyhedral approximation of the boundary}
We denote by $\XX=\{x_1,\dots,x_n\}$ a collection of points belonging to $\partial\Omega$ such that
\begin{equation}\label{pointsBoundary}
\frac32\tau\leq\min_{1\leq i\neq j\leq n}\d(x_i,x_j) \quad \mathrm{and} \quad \max_{1\leq i\leq n}\d(z,x_i)\leq \frac52\tau \ \mathrm{for \ any}\ z\in\partial\Omega,
\end{equation} 
where from now on $\d$ denotes the geodesic distance on $\partial\Omega$ and $\tau\in (0,1)$ is a given number. For any $x_i\in \XX$ let us denote by $\nu(x_i)$ the outer unit normal to $\partial\Omega$. Define
$$
\Omega_\XX\colonequals \cap_{1\leq i\leq n}\{z \ | \ \langle z-x_i,\nu(x_i)\rangle<0 \}.
$$
It is easy to see that $\partial \Omega_\XX$ is a polyhedral approximation of $\partial\Omega$ which in addition is convex if $\Omega$ is convex.

In the next lemma we show that the points of the collection $\XX$ can be displaced in order to make the normals $\nu(x_i)$'s satisfy ``good'' angle conditions between each other, when $\Omega$ is assumed to have strictly positive Gauss curvature.

\begin{lemma}\label{Lemma:pointsBoundary} Let $\Omega$ be a $C^2$ domain and assume that $\partial\Omega$ has strictly positive Gauss curvature. Let $\XX=\{x_1,\dots,x_n\}$ be a collection of points belonging to $\partial\Omega$ satisfying \eqref{pointsBoundary} for a number $\tau\in (0,1)$. Then there exist constants $\tau_0,C_0,C>0$ depending only on $\partial\Omega$, such that for any $0<\tau<\tau_0$ there exists a collection $\XX'=\{y_1,\dots,y_n\}\subset \partial\Omega$ such that, for any $0<\vartheta<C_0\tau^{5}$, the following hold:
\begin{enumerate}[leftmargin=*,font=\normalfont]
\item $\tau\leq\min_{1\leq i\neq j\leq n}\d(y_i,y_j)$ and $\max_{1\leq i \leq n}\d(z,y_i)\leq 3\tau \ \mathrm{for \ any}\ z\in\partial\Omega$.
\item
Letting
$$
\Omega_{\XX'}\colonequals \cap_{1\leq i\leq n}\{z \ | \ \langle z-y_i,\nu(y_i)\rangle<0 \},
$$
where $\nu(y_i)$ is the outer unit normal to $\partial\Omega$ at $y_i$,
we have
$$
|d(z,\partial \Omega_{\XX'})-d(z,\partial\Omega)|\leq C\tau^2\quad \mbox{for any }z\in \R^3.
$$ 
\item For any $i,j,k\in\{1,\dots,n\}$ with $i\neq j\neq k$, we have
$$
|\nu(y_i)\times \nu(y_j)|\geq \vartheta\quad \mathrm{and}\quad |\mathrm{det}(\nu(y_i),\nu(y_j),\nu(y_k))|\geq \vartheta^2.
$$
\end{enumerate}
\end{lemma}
\begin{proof}
Since we assume that $\partial\Omega$ has strictly positive Gauss curvature, we deduce that there exists a constant $\kappa_0(\partial\Omega)>0$ such that for any point $x\in\partial\Omega$ the minimal principal curvature of $\partial\Omega$ at $x$ is bounded below by $\kappa_0$.

Simple geometric arguments show that there exist constants $C,\vartheta_0,R_0>0$ depending only on $\partial\Omega$, such that for any $x\in \partial\Omega$ and for any $0<\vartheta<\vartheta_0$ if $v\in\R^3$ with $|v|=1$ is such that
$$
\theta(\nu(x),v)<\vartheta,
$$
where $\theta(\nu(x),v)$ is the angle formed by $\nu(x)$ and $v$, then for any $y\in\partial \Omega$ satisfying
$$
C\kappa_0^{-1}\vartheta\leq \d(x,y)\leq R_0,
$$
we have
$$
\theta(\nu(y),v)\geq \vartheta.
$$
We easily deduce that, up to an adjustment of the constants, for any $x\in \partial\Omega$ and for any $0<\vartheta<\vartheta_0$ if $v\in\R^3$ with $|v|=1$ is such that
$$
|\nu(x)\times v|<\vartheta,
$$ 
then for any $y\in\partial \Omega$ satisfying $C\kappa_0^{-1}\vartheta\leq \d(x,y)\leq R_0$, we have
\begin{equation}\label{dipoleBoundary}
|\nu(y)\times v|\geq \vartheta.
\end{equation}
Moreover, for any $x\in \partial\Omega$ and for any $0<\vartheta<\vartheta_0$ if $v,w\in\R^3$ with $|v|=|w|=1$ are such that
$$
|v\times w|\geq \vartheta\quad \mathrm{and}\quad |\det(\nu(x),v,w)|<\vartheta^2,
$$
then for any $y\in\partial \Omega$ satisfying $C\kappa_0^{-1}\vartheta\leq \d(x,y)\leq R_0$, we have
\begin{equation}\label{tripleBoundary}
|\det(\nu(y),v,w)|\geq \vartheta^2.
\end{equation}

\medskip
Now, we proceed by induction. We let $y_1=x_1$. 
Assume that we have defined a collection $\{y_1,\dots,y_l\}\subset\partial\Omega$ with $1<l<n$ such that for any $i,j,k\in \{1,\dots,l\}$ with $i\neq j\neq k$, we have
$$
|\nu(y_i)\times \nu(y_j)|\geq \vartheta\quad \mathrm{and}\quad |\mathrm{det}(\nu(y_i),\nu(y_j),\nu(y_k))|\geq \vartheta^2.
$$
From our previous observations we deduce that there exists a point $y\in\partial\Omega$, such that $\d(x_{l+1},y)\leq C\kappa_0^{-1}\vartheta$, satisfying \eqref{dipoleBoundary} for $v=\nu(y_1)$.
By repeating this procedure at most $l$ times, we find a point $y_{l+1}'$ with
$$
\d(x_{l+1},y_{l+1}')\leq Cl\kappa_0^{-1}\vartheta
$$ 
such that, for any $i\in\{1,\dots,l\}$,
$$
|\nu(y_{i})\times \nu(y_{l+1}')|\geq \vartheta.
$$ 
We further displace the point $y_{l+1}'$ in order to additionally satisfy the condition on the determinants (when $l\geq 3$). Once again from our previous observations we deduce that there exists a point $y\in\partial\Omega$, such that $\d(y_{l+1}',y)\leq C\kappa_0^{-1}\vartheta$, satisfying \eqref{tripleBoundary} for $v=\nu(y_1)$ and $w=\nu(y_2)$.
By repeating this procedure at most $l^2$ times, we find a point $y_{l+1}$ with
$$
\d(y_{l+1}',y_{l+1})\leq Cl^2k_0^{-1}\vartheta
$$ 
such that the collection $\{y_1,\dots,y_{l+1}\}\subset\partial\Omega$ satisfies 
$$
|\nu(y_i)\times \nu(y_j)|\geq \vartheta\quad \mathrm{and}\quad |\mathrm{det}(\nu(y_i),\nu(y_j),\nu(y_k))|\geq \vartheta^2
$$
for any $i,j,k\in \{1,\dots,l+1\}$ with $i\neq j\neq k$. 
This concludes the induction step and the proof of the third assertion. Note that 
$$
\d(x_l,y_l)\leq 2C(l-1)^2\kappa_0^{-1}\vartheta\leq 2Cn^2\kappa_0^{-1}\vartheta
$$
for any $l\in \{1,\dots,n\}$. Observing that $n\leq C\tau^{-2}$ for a universal constant $C>0$, we deduce that, for any $1\leq l \leq n$,
$$
\d(x_l,y_l)\leq C\tau^{-4}\vartheta.
$$
Therefore, if $C\kappa_0^{-1}\tau^{-4}\vartheta\leq 1/2\tau$ then the first assertion is satisfied. For a proof of the second statement, see \cite{Gru}*{Theorem 4}. This concludes the proof of the lemma.
\end{proof}

\subsection{Displacement of the points}
With the aid of Lemmas \ref{dipole}, \ref{tripole}, and \ref{Lemma:pointsBoundary}, we perform the displacement of the points of the collection $\AA$.

\begin{proposition}\label{Lemma:displacementBoundary} Let $\Omega$ be a $C^2$ domain and assume that $\partial\Omega$ has strictly positive Gauss curvature. Let $\AA=\{a_1,\dots,a_m\}\subset\Omega$ be a collection of $m$ non necessarily distinct points. Consider a collection $\XX=\{x_1,\dots,x_n\}\subset\partial\Omega$ satisfying \eqref{pointsBoundary} and let $\XX'=\{y_1,\dots,y_n\}\subset \partial\Omega$ be the collection of points given by Lemma \ref{Lemma:pointsBoundary} for a number $\tau<\tau_0$, where $\tau_0$ is the constant appearing in the lemma. Then there exist constants $C_0,C_1>0$ depending only on $\partial\Omega$ and a collection of points $\AA'=\{b_1,\dots,b_m\}\subset \Omega$ such that, for any 
$$
\vartheta<C_0\min\{m^{-6},m^{-4}\tau^2,m^{-2}\tau^4,\tau^5\},
$$ 
the following hold:
\begin{enumerate}[font=\normalfont,leftmargin=*]
\item $b_i\neq b_j$ for any $i\neq j$.
\item Define
$$
\nu_{(i,j)}\colonequals \frac{b_i-b_j}{|b_i-b_j|}\quad\mathrm{for}\ (i,j)\in \Lambda_m\colonequals\{(p,q) \ |\ 1\leq p<q\leq m\}
$$
and 
$$
\mathscr V\colonequals \{\nu_{(i,j)} \ | \ (i,j)\in \Lambda_m \} \cup \{\nu(y_i) \ | \ y_i\in \XX'\}.
$$
Then for any $u,v,w\in \mathscr V$ with $u\neq v\neq w$, we have 
$$
|u\times v|\geq \vartheta \quad \mathrm{and}\quad |\mathrm{det}(u,v,w)|\geq \vartheta^2.
$$
\item $|a_l-b_l|\leq C_1(l^5+l^3\tau^{-2}+l\tau^{-4})\vartheta$ for any $l\in \{1,\dots,m\}$.
\end{enumerate}
\end{proposition}
\begin{proof} Assume $\vartheta\leq C_0\tau^5$, where $C_0$ is the constant appearing in Lemma \ref{Lemma:pointsBoundary}, so that for any $i,j,k\in\{1,\dots,n\}$ with $i\neq j\neq k$, we have
$$
|\nu(y_i)\times \nu(y_j)|\geq \vartheta\quad \mathrm{and}\quad |\mathrm{det}(\nu(y_i),\nu(y_j),\nu(y_k))|\geq \vartheta^2.
$$
We proceed by induction. We define $b_1=a_1$, $d_1=0$, and $D_1=\mathrm{diam}(\Omega)+1$. 
Assume that we have defined a collection $\{b_1,\dots,b_l\}$ with $2\leq l<m$ such that:
\begin{itemize}
\item Letting
$$
\mathscr V_l= \{\nu_{(i,j)} \ | \ (i,j)\in \Lambda_l \} \cup \{\nu(y_i) \ | \ y_i\in\XX' \},
$$
then for any $u,v,w\in \mathscr V_l$ with $u\neq v\neq w$, we have
$$
|u\times v|\geq \vartheta \quad \mathrm{and}\quad |\mathrm{det}(u,v,w)|\geq \vartheta^2.
$$
\item $|a_i-b_i|\leq d_l$ for any $i\in \{1,\dots,l\}$.
\item $|b_i-b_j|\leq D_l$ for any $i,j\in \{1,\dots,l\}$.
\end{itemize}
Using Lemmas \ref{dipole}, \ref{tripole}, and arguing as in the proof of Proposition \ref{displacement}, we find a point $b_{l+1}'$ such that the collection  $\{b_1,\dots,b_l,b_{l+1}'\}$ satisfies:
\begin{itemize}
\item For any $\alpha,\beta,\gamma\in \Lambda_{l+1}$ with $\alpha\neq \beta\neq \gamma$, we have
$$
|\nu_\alpha\times \nu_\beta|\geq \vartheta \quad \mathrm{and}\quad |\mathrm{det}(\nu_\alpha,\nu_\beta,\nu_\gamma)|\geq \vartheta^2.
$$
\item $|a_i-b_i|\leq d_{l+1}=6l^5(D_l+d_l)\vartheta$ for any $i\in \{1,\dots,l+1\}$.
\item $|b_i-b_j|\leq D_{l+1}=D_l+d_l$ for any $i,j\in \{1,\dots,l+1\}$.
\end{itemize}
We now displace the point $b_{l+1}'$ in order to additionally satisfy the conditions involving the vectors of the collection $\XX'$.

First, applying Lemma \ref{dipole} with the points $x_1=y_1$, $x_2=y_1+\nu(y_1)$, $x_3=b_1$, $x_4=b_{l+1}'$, and the number $D=D_l+d_l$, we find a point $x_4'$ satisfying \eqref{dipole1} and \eqref{dipole2}. We recall that $n\leq C\tau^{-2}$, where throughout the proof $C$ denotes a universal constant that may change from line to line. By repeating this argument at most $C\tau^{-2}l$ times, we find a point $b_{l+1}''$ with
$$
|a_{l+1}-b_{l+1}''|\leq (6l^5+3C\tau^{-2}l)(D_l+d_l)\vartheta
$$
such that the collection $\{b_1,\dots,b_l,b_{l+1}''\}$, in addition to the previous properties, satisfies
$$
|\nu_{(i,l+1)} \times \nu(y_j)|\geq \vartheta\quad \mathrm{for\ any}\ i\in \{1,\dots,l\}\ \mathrm{and}\ j\in \{1,\dots,n\}.
$$
When $l>2$, we further displace the point $b_{l+1}''$. Applying Lemma \ref{tripole} with the points $x_1=y_1$, $x_2=y_1+\nu(y_1)$, $x_3=b_1$, $x_4=b_2$, $x_5=b_1$, $x_6=b_{l+1}''$ and the number $D=D_l+d_l$, we find a point $x_6'$ satisfying \eqref{tripole1} and \eqref{tripole2}. By repeating this argument at most $C\tau^{-2}l^3$ times, we find a point $b_{l+1}'''$ with
$$
|a_{l+1}-b_{l+1}'''|\leq (6l^5+3C\tau^{-2}l+3C\tau^{-2}l^3)(D_l+d_l)\vartheta
$$
such that the collection $\{b_1,\dots,b_l,b_{l+1}'''\}$, in addition to the previous properties, satisfies
$$
|\mathrm{det}(\nu_{(i,l+1)},\nu_\alpha, \nu(y_j))|\geq \vartheta^2\quad \mathrm{for\ any}\ i\in \{1,\dots,l\}, \ \alpha\in \Lambda_l, \ \mathrm{and}\ j\in \{1,\dots,n\}.
$$
Finally, applying Lemma \ref{tripole} with the points $x_1=y_1$, $x_2=y_1+\nu(y_1)$, $x_3=y_2$, $x_4=y_2+\nu(y_2)$, $x_5=b_1$, $x_6=b_{l+1}'''$ and the number $D=D_l+d_l$, we find a point $x_6'$ satisfying \eqref{tripole1} and \eqref{tripole2}. By repeating this argument at most $C^2\tau^{-4}l$ times, we find a point $b_{l+1}$ with
$$
|a_{l+1}-b_{l+1}|\leq (6l^5+3C\tau^{-2}l+3C\tau^{-2}l^3+3C^2\tau^{-4}l)(D_l+d_l)\vartheta
$$
such that the collection $\{b_1,\dots,b_l,b_{l+1}\}$, in addition to the previous properties, satisfies
$$
|\mathrm{det}(\nu_{(i,l+1)},\nu(y_j), \nu(y_k))|\geq \vartheta^2\quad \mathrm{for\ any}\ i\in \{1,\dots,l\}\ \mathrm{and}\ j,k\in\{1,\dots,n\} \mathrm{\ with}\ j\neq k.
$$
Summarizing, the collection $\{b_1,\dots,b_l,b_{l+1}\}$ satisfies:
\begin{itemize}
\item  Letting
$$
\mathscr V_{l+1}= \{\nu_{(i,j)} \ | \ (i,j)\in \Lambda_{l+1} \} \cup \{\nu(y_i) \ | \ y_i\in\XX' \},
$$
then for any $u,v,w\in \mathscr V_l$ with $u\neq v\neq w$, we have
$$
|u\times v|\geq \vartheta \quad \mathrm{and}\quad |\mathrm{det}(u,v,w)|\geq \vartheta^2.
$$
\item $|a_i-b_i|\leq d_{l+1}=C(l^5+\tau^{-2}l^3+\tau^{-4}l)(D_l+d_l)\vartheta$ for any $i\in \{1,\dots,l+1\}$, where $C$ is a universal constant.
\item $|b_i-b_j|\leq D_{l+1}=D_l+d_l$ for any $i,j\in \{1,\dots,l+1\}$.
\end{itemize}
This concludes the induction step. Arguing as in the proof of Proposition \ref{displacement}, we find upper bounds for the recursively defined distances $d_l$ and $D_l$. Observe that 
$$
D_{l+1}\leq D_l(1+C(l^5+\tau^{-2}l^3+\tau^{-4}l)\vartheta),\quad  D_1=\mathrm{diam}(\Omega)+1.
$$
We immediately check that if $\vartheta\leq \min\{\vartheta_1,\vartheta_3,m^{-6}+\tau^{2}m^{-4}+\tau^{4}m^{-2}\}$, where $\vartheta_1$ and $\vartheta_3$ are the constants appearing in Lemma \ref{dipole} and Lemma \ref{tripole} respectively, then for any $l\in \{1,\dots,m\}$
$$
D_l\leq D_m\leq D_1(1+C(m^5+\tau^{-2}m^3+\tau^{-4}m)\vartheta)^m\leq D_1\left(1+\frac{C}{m}\right)^m\leq C(\mathrm{diam}(\Omega)+1).
$$
Moreover, if $\vartheta\leq m^{-6}+\tau^{2}m^{-4}+\tau^{4}m^{-2}$ then for any $l\in\{1,\dots,m\}$
$$
d_l\leq C(\mathrm{diam}(\Omega)+1)(l^5+\tau^{-2}l^3+\tau^{-4}l)\vartheta.
$$
Thus, provided that 
$$
\vartheta<C_0\min\{m^{-6},m^{-4}\tau^2,m^{-2}\tau^4,\tau^5\},
$$ 
where the constant $C_0$ depends only on $\partial\Omega$, the proposition follows.
\end{proof}

\subsection{Proof of Proposition \ref{prop:functionzetaBoundary2}}
\begin{proof} 
Let us begin by defining the function $\zeta_\la$.
Let $\XX=\{x_1,\dots,x_n\}\subset\partial\Omega$ be a collection of points satisfying \eqref{pointsBoundary} for a number $\tau\in(0,1)$.
Apply Lemma \ref{Lemma:pointsBoundary} to obtain a collection $\XX'=\{y_1,\dots,y_n\}$ for $0<\tau<\tau_0$, where $\tau_0$ is the constant appearing in the statement of the lemma. Then apply Proposition \ref{Lemma:displacementBoundary} with the collection of points $\AA\subset \Omega$ to obtain a collection $\AA'=\{p_1',\dots,p_k',n_1',\dots,n_k'\}\subset \Omega$. We consider a number 
$$
\vartheta<C_0\min\{(2k)^{-6},(2k)^{-4}\tau^2,(2k)^{-2}\tau^4,\tau^5\},
$$
where $C_0=C_0(\partial\Omega)$ is the constant appearing in the statement of the proposition. 
Observe that
\begin{align*}
L_{\partial\Omega}(\AA')\leq \sum_{i=1}^kd_{\partial\Omega}(p_i',n_i')&\leq \sum_{i=1}^kd_{\partial\Omega}(p_i,n_i)+d_{\partial\Omega}(p_i,p_i')+d_{\partial\Omega}(n_i,n_i')\\
&\leq L_{\partial\Omega}(\AA)+C(2k)\left((2k)^5+(2k)^3\tau^{-2}+(2k)\tau^{-4}\right)\vartheta,
\end{align*}
where throughout the proof $C$ denotes a constant depending only on $\partial\Omega$, that may change from line to line.
An analogous argument shows that 
$$
L_{\partial\Omega}(\AA)\leq L_{\partial\Omega}(\AA')+C(2k)\left((2k)^5+(2k)^3\tau^{-2}+(2k)\tau^{-4}\right)\vartheta.
$$
Therefore
\begin{equation}\label{difMinConBoundary}
|L_{\partial\Omega}(\AA)-L_{\partial\Omega}(\AA')|\leq C(2k)\left((2k)^5+(2k)^3\tau^{-2}+(2k)\tau^{-4}\right)\vartheta.
\end{equation}

Remember that by Lemma \ref{LemmaBCLBoundary} there exists a $1$-Lipschitz function $\zeta^*:\cup_{i=1,\dots,k}\{p_i',n_i'\}\to \R$ such that 
$$
L_{\partial\Omega}(\AA')=\sum_{i=1}^k\zeta^*(p_i')-\zeta^*(n_i').
$$
Define the function $\zeta$ for $d_{\partial\Omega}$ as in Definition \ref{def:funczetaBoundary}, i.e. set
$$
\zeta(x)\colonequals \max_{i\in \{1,\dots,k\}} \left(\zeta^*(p_i')-d_i(x,\partial\Omega)\right),
$$
where
$$
d_i(x,\partial\Omega)\colonequals 
\min\left[
\max \left( \max_{j\in\{1,\dots,2k\}}d_{(i,j)}(x),d(p_i',\partial\Omega)-d(x,\partial\Omega)\right),
d(p_i',\partial\Omega)+d(x,\partial\Omega)
\right]
$$
and
$$
d_{(i,j)}(x)\colonequals \langle p_i'-x,\nu_{(i,j)}\rangle,\quad \nu_{(i,j)}=\left\{\begin{array}{cl}\frac{p_i'-a_j'}{|p_i'-a_j'|}&\mathrm{if}\ p_i'\neq a_j'\\0&\mathrm{if}\ p_i'=a_j'\end{array}\right. .
$$
Lemma \ref{lem:funczetaBoundary} yields that $\zeta:\R^3\to \R$ is a $1$-Lipschitz function such that
$$
\displaystyle\sum_{i=1}^{k}\zeta(p_i')-\zeta(n_i')=\sum_{i=1}^k\zeta^*(p_i')-\zeta^*(n_i')= L_{\partial\Omega}(\AA').
$$
Recall that by Lemma \ref{Lemma:pointsBoundary}, letting 
$$
\Omega_{\XX'}\colonequals \cap_{1\leq l \leq n}\{z \ | \ \langle z-y_l,\nu(y_l)\rangle<0 \},
$$
where $\nu(y_l)$ is the outer unit normal to $\partial\Omega$ at $y_l$,
we have
\begin{equation}\label{difdomains}
|d(x,\partial \Omega_{\XX'})-d(x,\partial\Omega)|\leq C\tau^2\quad \mbox{for any }x\in \R^3. 
\end{equation}
Observe that, since $\Omega$ is convex, $\Omega\subset \Omega_{\XX'}$ and that for any $x\in \overline \Omega_{\XX'}$
$$
d(x,\partial \Omega_{\XX'})=\min_{1\leq l \leq n}\langle y_l-x,\nu(y_l) \rangle.
$$
In order to take advantage of this fact, we define a new function by replacing the distance to $\partial\Omega$ with the distance to $\partial \Omega_{\XX'}$. More precisely, we let
$$
\tilde \zeta (x)\colonequals 
\max_{i\in \{1,\dots,k \}}(\zeta^*(p_i')-d_i(x,\partial\Omega_{\XX'})).
$$
From \eqref{difdomains}, we deduce that
\begin{equation}\label{difMinConBoundary2}
\left| L_{\partial\Omega}(\AA')-\sum_{i=1}^k \tilde\zeta(p_i')-\tilde\zeta(n_i') \right|\leq C(2k)\tau^2.
\end{equation}
Next, we regularize the function $\tilde \zeta$. Let $\varphi\in C_c^\infty(B(0,1),\R_+)$ be a mollifier such that $\int_{\R^3} \varphi (x)dx=1$.
Letting 
\begin{equation}\label{deflambdaBoundary}
\la\colonequals \vartheta^{1/ \rho}\quad \mathrm{for} \ \rho\in(0,1/2),
\end{equation}
we define
$$
\zeta_\la(\cdot)\colonequals \varphi_\la * \tilde\zeta(\cdot)=\int_{\R^3} \varphi_\la(\cdot-z)\tilde\zeta(z)dz\quad \mathrm{with}\ \varphi_\la(\cdot)=\frac1\la \varphi\left(\frac{\cdot}{\la}\right).
$$

\medskip\noindent
{\bf Argument for the first statement}. Observe that $\|\tilde \zeta-\zeta_\la\|_{L^\infty(\R^3)}\leq \la$. We deduce that 
\begin{equation}\label{difMinConBoundary3}
\left|\sum_{i=1}^k \tilde\zeta(p_i')-\tilde\zeta(n_i')-\sum_{i=1}^k \zeta_\la(p_i')-\zeta_\la(n_i')\right|\leq 2k\lambda.
\end{equation}
By combining \eqref{difMinConBoundary} with \eqref{difMinConBoundary2} and \eqref{difMinConBoundary3}, we obtain
\begin{equation}\label{dif1Boundary}
\left|L_{\partial\Omega}(\AA)-\sum_{i=1}^k \zeta_\la(p_i')-\zeta_\la(n_i')\right|\leq C(((2k)^6+(2k)^4\tau^{-2}+(2k)^2\tau^{-4})\vartheta+2k(\tau^2 +\lambda)).
\end{equation}
On the other hand, note that
\begin{align*}
\left|\sum_{i=1}^k \zeta(p_i)-\zeta(n_i)-\sum_{i=1}^k \zeta(p_i')-\zeta(n_i')\right|&\leq \sum_{i=1}^k|p_i-n_i|+|p_i'-n_i'|\\
&\leq C(2k)\left((2k)^5+(2k)^3\tau^{-2}+(2k)\tau^{-4}\right)\vartheta.
\end{align*}
By combining the previous estimate with \eqref{difMinConBoundary2} and \eqref{difMinConBoundary3}, we get
\begin{equation}\label{dif2Boundary}
\left|\sum_{i=1}^k \zeta(p_i)-\zeta(n_i)-\sum_{i=1}^k \zeta_\la(p_i')-\zeta_\la(n_i')\right|\leq C(((2k)^6+(2k)^4\tau^{-2}+(2k)^2\tau^{-4})\vartheta+2k(\tau^2 +\lambda)).
\end{equation}
Then by \eqref{dif1Boundary} and \eqref{dif2Boundary}, we deduce that
$$
\left|L_{\partial\Omega}(\AA)-\sum_{i=1}^k \zeta_\la(p_i)-\zeta_\la(n_i) \right|\leq C(((2k)^6+(2k)^4\tau^{-2}+(2k)^2\tau^{-4})\la^\rho+2k\tau^2).
$$

\medskip\noindent
{\bf Argument for the second statement}. Note that 
\begin{equation}\label{gradzetaBoundary}
\nabla \zeta_\la(x)=\int_{B(x,\la)}\varphi_\la(x-z)\nabla \tilde \zeta (z) dz
\end{equation}
for any $x\in \R^3$. We define  
$$
\Lambda \colonequals \{(i,j)\ | \ 1\leq i\leq k,\ 1\leq j\leq 2k,\ i\neq j\}\quad\mathrm{and}\quad
c\colonequals \max_{i\in \{1,\dots,k\}}(\zeta^*(p_i')-d(p_i',\partial\Omega_{\XX'})).
$$
Then, letting
\begin{align*}
\zeta_{(i,j)}(\cdot)&\colonequals\zeta^*(p_i')-d_{(i,j)}(\cdot) \quad \;\; \mathrm{for}\ (i,j)\in \Lambda,\\
\zeta_{l,+}(\cdot)&\colonequals c+\langle \cdot-y_l,\nu(y_l)\rangle \quad \mathrm{for}\ l\in\{1,\dots,n\},\\
\zeta_{l,-}(\cdot)&\colonequals c-\langle \cdot-y_l,\nu(y_l)\rangle \quad \mathrm{for}\ l\in\{1,\dots,n\},
\end{align*}
observe that, for almost every $z\in \Omega_{\XX'}$,
$$
\nabla \tilde \zeta (z)=\left\{
\begin{array}{rl}
\nu_{(i,j)}& \mathrm{if}\ \tilde \zeta(z)=\zeta_{(i,j)}(z)\, \mathrm{for\ some}\ (i,j)\in \Lambda\\
\nu(y_l)     & \mathrm{if}\ \tilde \zeta(z)=\zeta_{l,+}(z)\ \, \mathrm{for\ some}\ l\in\{1,\dots,n\}\\
-\nu(y_l)    & \mathrm{if}\ \tilde \zeta(z)=\zeta_{l,-}(z)\ \, \mathrm{for\ some}\ l\in\{1,\dots,n\}. \end{array}\right.
$$
In particular $|\nabla \tilde \zeta(z)|=1$ for almost every $z\in \Omega_{\XX'}$. Thus 
$$
|\nabla \zeta_\la(x)|\leq \int_{B(x,\la)}\varphi_\la(x-z)|\nabla \tilde \zeta (z)| dz\leq \int_{B(x,\la)}\varphi_\la(x-z)dz=1
$$
for any $x\in\Omega_\la$.

\medskip\noindent
{\bf Argument for the third statement}. Observe that
$$
\tilde \zeta(x)=c\quad \mathrm{for\ any}\ x\in \partial\Omega_{\XX'}. 
$$
Thus
$$
|\zeta_\la(\overline{\Omega\setminus \Omega_\la
})|\leq C(\tau^2+\la).
$$

\medskip\noindent
{\bf Argument for the fourth statement}. We now analyze the set of points in $\Omega_\la$ whose gradient is small in modulus. From \eqref{gradzetaBoundary}, we deduce that
$$
\nabla \zeta_\la (x)= \sum_{\alpha\in\Lambda} \sigma_\alpha \nu_\alpha + \sum_{l=1}^n \sigma_{l,+} \nu(y_l)+\sum_{l=1}^n \sigma_{l,-}( -\nu(y_l)),
$$
where 
\begin{align*}
\sigma_\alpha &= \int_{B(x,\la )}\varphi_\la(x-z)\mathbf{1}_{\tilde \zeta(z)=\zeta_\alpha(z)}dz \quad \,\, \mathrm{for}\ \alpha\in\Lambda,\\
\sigma_{l,+} &= \int_{B(x,\la )}\varphi_\la(x-z)\mathbf{1}_{\tilde \zeta(z)=\zeta_{l,+}(z)}dz \quad \mathrm{for}\ l\in\{1,\dots,n\},\\
\sigma_{l,-} &= \int_{B(x,\la )}\varphi_\la(x-z)\mathbf{1}_{\tilde \zeta(z)=\zeta_{l,-}(z)}dz \quad \mathrm{for}\ l\in\{1,\dots,n\}.
\end{align*}
Observe that $\sigma_\alpha,\sigma_{l,+},\sigma_{l,-}\in [0,1]$ and that $\sum_{\alpha\in \Lambda}\sigma_\alpha+\sum_{l=1}^n \sigma_{l,+} +\sum_{l=1}^n \sigma_{l,-}=1$. We conclude that, for any $x\in \Omega_\la$, $\nabla \zeta_\la (x)$ is a convex combination of the vectors $\nu_\alpha$'s, $\nu(y_l)$'s, and $-\nu(y_l)$'s with $\alpha\in \Lambda$, $l\in \{1,\dots,n \}$. By Caratheodory's theorem, we deduce that $\nabla\zeta_\la(z)$ is a convex combination of at most four of them. 

We define
$$
P_\la \colonequals \{z\in \R^3 \ | \ d(z,P)\leq 2\la \},
$$
where
$$
P\colonequals \cup_{1\leq i<j\leq k} P_{i,j}, \quad 
P_{i,j}\colonequals \{z\in\R^3 \ | \ \zeta_{(i,j)}(z)=\zeta_{(j,i)}(z) \}.
$$
Arguing as in the proof of Proposition \ref{prop:functionzeta}, we deduce that $|\zeta_\la (P_\la)| \leq 2\la k^2$.

Consider a number $\kappa<\vartheta^2/3$ and a point
$$
x\in \{z\in \Omega_\la \ | \ |\nabla \zeta_\la (z)|<\kappa \}\setminus P_\la.
$$
We observe that, since $x\not\in P_\la$, if there exists a point $y\in \overline {B(x,\la)}$ and indices $i,j\in\{1,\dots,k\}$ with $i\neq j$ such that
$$
\zeta(y)=\zeta_{(i,j)}(y)
$$
then, for any $z\in \overline{B(x,\la)}$,
$$
\zeta(z)\neq \zeta_{(j,i)}(z)
$$

On the other hand, since $x\in \Omega_\la$, if there exist a point $z_+\in \overline{B(x,\la)}$ and an index $l\in\{1,\dots,n\}$ such that 
$$
\tilde \zeta(z_+)=\zeta_{l,+}(z_+)
$$
then, for any $z\in \overline{B(x,\la)}$,
$$
\tilde \zeta(z)\neq \zeta_{l,-}(z). 
$$
Arguing by contradiction, assume that there exist  points $z_+,z_-\in \overline{ B(x,\la)}$ and an index $l\in\{1,\dots,n\}$ such that
$$
\tilde \zeta (z_+)=\zeta_{l,+}(z_+)\quad\mathrm{and}\quad \tilde \zeta (z_-)=\zeta_{l,-}(z_-).
$$
Observe that
$$
|\tilde \zeta (z_+)-\tilde \zeta (z_-)|\leq \la\quad \mathrm{and}\quad |\zeta_{l,-}(z_-)  -\zeta_{l,-}(z_+)|\leq \la
$$
and that
$$
|\zeta_{l,+}(z_+)-\zeta_{l,-}(z_+)|=2d(z_+,\partial\Omega_{\XX'})\geq 2d(z_+,\partial\Omega)>2\la.
$$
But
$$
|\zeta_{l,+}(z_+)-\zeta_{l,-}(z_+)|= |\tilde \zeta(z_+)-\tilde \zeta(z_-) +\zeta_{l,-}(z_-)  -\zeta_{l,-}(z_+)|\leq 2\la,
$$
which yields a contradiction with the previous computation.

Analogously, if there exist a point $z_-\in \overline{B(x,\la)}$ and an index $l\in\{1,\dots,n\}$ such that 
$$
\tilde \zeta(z_-)=\zeta_{l,-}(z_-)
$$
then, for any $z\in \overline{B(x,\la)}$,
$$
\tilde \zeta(z)\neq \zeta_{l,+}(z). 
$$
This implies that $\nabla \zeta_\la(x)$ is a convex combination of at most four vectors, where if one them happens to be $\nu_{(i,j)}$ for some $i,j\in \{1,\dots,k\}$ with $i\neq j$ then all the other vectors are different from $\nu_{(j,i)}=-\nu_{(i,j)}$ and if one of them happens to be $\nu(y_l)$ (respectively $-\nu(y_l)$) for some $l\in\{1,\dots,n\}$ then all the other vectors are different from $-\nu(y_l)$ (respectively $\nu(y_l)$). Recalling that by Proposition \ref{Lemma:displacementBoundary}, we have
$$
|v_1\times v_2|\geq \vartheta \quad \mathrm{and}\quad |\det(v_1,v_2,v_3)|\geq \vartheta^2
$$
for any $v_1,v_2,v_3\in \{\nu_{(i,j)}\ | \ 1\leq i \leq k,\ 1\leq j\leq 2k,\ i<j\}\cup \{\nu(y_l)\ | \ 1\leq l\leq n\}$ with $v_1\neq v_2\neq v_3$ we deduce that $\nabla \zeta_\la (x)$ is a convex combination of at most four vector that satisfy the previous property.

Arguing as in the proof of Proposition \ref{prop:functionzeta} we conclude that if $
x\in \{y\in \Omega_\la \ | \ |\nabla \zeta_\la(y)|<\kappa\}\setminus P_\la$ with $\kappa<\vartheta^2/3$, then there exist four different functions 
$$
\zeta_1,\zeta_2,\zeta_3,\zeta_4\in \{\zeta_{(i,j)} \ | \ (i,j)\in\Lambda \}\cup \{\zeta_{l,+} \ | \ y_l\in \XX' \}\cup \{\zeta_{l,-} \ | \ y_l\in \XX' \},
$$ 
where if $\zeta_a=\zeta_{(i,j)}$ for some $(i,j) \in \Lambda$ and $a\in \{1,2,3,4\}$ then $\zeta_b\neq \zeta_{(j,i)}$ for any $b\in\{1,2,3,4\}\setminus \{a\}$ and 
if $\zeta_a=\zeta_{l,+}$ (respectively $\zeta_a=\zeta_{l,-}$) for some $l\in\{1,\dots,n\}$ and $a\in \{1,2,3,4\}$ then $\zeta_b\neq \zeta_{l,-}$ (respectively $\zeta_b\neq \zeta_{l,+}$) for any $b\in\{1,2,3,4\}\setminus \{a\}$, such that the unique solution $z\in \R^3$ to the linear system of equations
$$
\zeta_1(z)=\zeta_2(z)=\zeta_3(z)=\zeta_4(z)
$$
lies in the ball $B(x,C\lambda/(\vartheta^2-3\kappa))$.
We conclude that the set
$$
C_\kappa\colonequals \{x\in \Omega_\la \ | \ |\nabla \zeta_\la (x)|<\kappa\}\setminus P_\la
$$
can be covered by $\B_\kappa$, a collection of at most $\binom{|\Lambda|+2|\XX'|}{4}\leq C_0((2k)^8+\tau^{-8})$ balls of radius $C\la/(\vartheta^2-3\kappa)$. Observing that
$$
|D^2\zeta_\la(x)|\leq \frac C{\la^2}
$$
for any $x\in \Omega_\la$, and letting 
$$
T_\kappa\colonequals \zeta_\la (\cup_{B\in \B_\kappa}B),
$$
we deduce that, for any $t\in \zeta_\la(\Omega_\la)\setminus (T_\kappa\cup \zeta_\la(P_\la))$, $\{x \ | \ \zeta_\la(x)=t\}$ is a complete submanifold of $\R^3$ whose second fundamental form is bounded by 
$$
C \frac{\sup_{\Omega_\la} |D^2 \zeta_\la|}{\inf_{\Omega_\la\setminus ((\cup_{B\in \B_\kappa}B)\cup P_\la )} |\nabla \zeta_\la| }\leq\frac C{\la^2\kappa}.
$$
Recalling the relation between $\la$ and $\vartheta$ (see \eqref{deflambdaBoundary}), the proposition follows.
\end{proof}

\bibliography{referencesGL}
\end{document}